\pdfoutput=1
\RequirePackage[l2tabu, orthodox]{nag}
\documentclass[final,11pt]{article}

\usepackage[inner=3.1cm,outer=3.1cm,top=2cm,bottom=2.5cm,headsep=0.5cm]{geometry}

\usepackage{tikz, bm, mathrsfs}

\usepackage[final]{microtype}

\usepackage[small]{titlesec}

\usepackage{mathptmx, upgreek} 

\usepackage{amsmath, amssymb, amsfonts, amsthm, mathtools}
\usepackage{fixmath}
\mathtoolsset{centercolon}

\usepackage[latin1]{inputenc}
\usepackage[T1]{fontenc}

\usepackage[draft=false,pdftex]{hyperref}
\hypersetup{
    unicode=false,          
    pdftoolbar=true,        
    pdfmenubar=true,        
    pdffitwindow=false,     
    pdfstartview={FitH},    
    pdftitle={Combinatorial distance geometry in normed spaces},    
    pdfauthor={Konrad J. Swanepoel},     
    pdfnewwindow=false,      
    colorlinks=true,       
    linkcolor=black,          
    citecolor=black,        
    filecolor=black,      
    urlcolor=black           
}

\theoremstyle{plain}
\newtheorem{theorem}{Theorem}
\newtheorem{lemma}[theorem]{Lemma}
\newtheorem{corollary}[theorem]{Corollary}
\newtheorem{conjecture}{Conjecture}

\newtheorem{proposition}[theorem]{Proposition}

\theoremstyle{definition}

\newcommand{\un}[1]{\widehat{#1}}
\newcommand{\setbuilder}[2]{\left\{#1\colon#2\right\}}
\newcommand{\set}[1]{\left\{#1\right\}}
\newcommand{\sequencebuilder}[2]{\left(#1\colon#2\right)}

\newcommand{\epsi}{\varepsilon}

\renewcommand{\leq}{\leqslant}
\renewcommand{\geq}{\geqslant}
\newcommand{\norm}[1]{\left\lVert#1\right\rVert}

\newcommand{\ipr}[2]{\left\langle #1, #2 \right\rangle}

\newcommand{\abs}[1]{\left\lvert#1\right\rvert}
\newcommand{\myangle}{\sphericalangle}
\newcommand{\Line}[1]{#1}
\newcommand{\Segment}[1]{#1}
\newcommand{\card}[1]{\left\lvert#1\right\rvert}

\newcommand{\numbersystem}[1]{\mathbb{#1}}
\newcommand{\bE}{\numbersystem{E}}
\newcommand{\bN}{\numbersystem{N}}
\newcommand{\bR}{\numbersystem{R}}
\newcommand{\bZ}{\numbersystem{Z}}
\DeclareMathOperator{\conv}{conv}
\newcommand{\collection}[1]{{\mathcal#1}}
\newcommand{\CO}{\collection{O}}
\newcommand{\CP}{\collection{P}}

\DeclareMathOperator{\diam}{diam}
\newcommand{\bd}{\partial}

\DeclareMathOperator{\interior}{int}

\DeclareMathOperator{\lin}{span}

\newcommand{\vect}[1]{#1}
\newcommand{\va}{\vect{a}}
\newcommand{\vb}{\vect{b}}
\newcommand{\vc}{\vect{c}}
\newcommand{\vd}{\vect{d}}
\newcommand{\ve}{\vect{e}}
\newcommand{\vf}{\vect{f}}

\newcommand{\vo}{\vect{o}}
\newcommand{\vp}{\vect{p}}

\newcommand{\vv}{\vect{v}}
\newcommand{\vw}{\vect{w}}
\newcommand{\vx}{\vect{x}}
\newcommand{\vy}{\vect{y}}

\newcommand{\define}[1]{%
    \emph{#1}%
}

\title{Combinatorial distance geometry in normed spaces}
\author{Konrad J.\ Swanepoel\thanks{Department of Mathematics, London School of Economics and Political Science, Houghton Street, London WC2A 2AE, United Kingdom. Email: 
    \href{mailto:k.swanepoel@lse.ac.uk}{k.swanepoel@lse.ac.uk}
}
}
\date{}

\begin{document}
\maketitle

\begin{abstract}
We survey problems and results from combinatorial geometry in normed spaces, concentrating on problems that involve distances.
These include various properties of unit-distance graphs, minimum-distance graphs, diameter graphs, as well as minimum spanning trees and Steiner minimum trees.
In particular, we discuss translative kissing (or Hadwiger) numbers, equilateral sets, and the Borsuk problem in normed spaces.
We show how to use the angular measure of Peter Brass to prove various statements about Hadwiger and blocking numbers of convex bodies in the plane, including some new results.
We also include some new results on thin cones and their application to distinct distances and other combinatorial problems for normed spaces.
\end{abstract}

\microtypesetup{protrusion=false}
{\small \tableofcontents}
\microtypesetup{protrusion=true}

\section{Introduction}
Paul Erd\H{o}s \cite{Erdos} introduced many combinatorial questions into geometry.
Progress in solving these and many subsequent problems went hand-in-hand with corresponding advances in combinatorics and combinatorial number theory.
Recently, some spectacular results were obtained using the polynomial method, which introduced strong connections to algebra and algebraic geometry.
In this survey, we would like to explore a different direction, and consider combinatorial questions for other norms.
There have been sporadic attempts at generalising geometric questions of Erd\H{o}s to other normed spaces, an early example being a paper of Fullerton \cite{Fullerton1949}.
According to Erd\H{o}s~\cite{Erdos1985}, Ulam was also interested in generalizing certain distance problems to other metrics.
This survey is an attempt at presenting the literature in a systematic way.

We will also present new proofs of known results and give results that have not appeared in the literature before.
Since we will confine ourselves to normed spaces, it is natural that problems involving distances will play a special role.
However, many of these problems have alternative formulations in terms of packings and coverings of balls, or involve packings and coverings in their solutions, so there is some overlap with the general theory of packing and covering, as conceived by L\'aszl\'o Fejes T\'oth \cite{Fejes-Toth-book} and others.
Nevertheless, we make no attempt here to give a systematic treatment of packing and covering, apart from reviewing what is known about Hadwiger numbers (or translative kissing numbers) of convex bodies and some close relatives, as these numbers show up when we consider minimum-distance graphs and minimal spanning trees.

We have left out many topics with a combinatorial flavour,
due to limitations on space and time.
These include results on vector sums in normed spaces (such as in the papers \cite{ABG, Barany2008, Kleitman1965, LM2015, SwanepoelTAMS}), embeddings of metric spaces into normed spaces, a topic with applications in computer science (see \cite[Chapter~15]{Matousek2002} and \cite{Ostrovskii2013}), Menger-type results \cite{Bandelt1996, Bandelt1998, Malitz1992}, and isometries and variants such as unit-distance preserving maps and random geometric graphs (for instance \cite{BBGLW} and \cite{Geher}).
For recent surveys on covering and illumination, see Bezdek and Khan \cite{BK2016b} and Nasz\'odi \cite{Naszodi2016}.
For a recent survey on discrete geometry in normed spaces, see Alonso, Martini, and Spirova \cite{AMS2013}.

\subsection{Outline}
After setting out some terminology in the next subsection, we will survey the Hadwiger number of a convex body, as well as some variants of this notion in Section~\ref{section:Hadwiger}.
In Section~\ref{section:equilateral} we survey recent results on equilateral sets.
Although these two sections may at first not seem central to this paper, Hadwiger and equilateral numbers are often the best known general estimates for various combinatorial quantities.
Then we consider three graphs that can be defined on a finite point set in a normed space: the minimum-distance graph, the unit-distance graph and the diameter graph.
Section~\ref{section:mindist} covers minimum-distance graphs.
Since many results on unit-distance and diameter graphs have a similar flavour, we cover them together in Section~\ref{section:unitdistance}.
We briefly consider some other graphs such as geometric minimum spanning trees, Steiner minimum trees and sphere-of-influence graphs in Section~\ref{section:othergraphs}.
Then in Section~\ref{section:angle}, we present some applications of an angular measure introduced by Brass \cite{Brass1996}, in order to give simple proofs of various two-dimensional results on relatives of the Hadwiger number.
In particular, we prove a result of Zong \cite{Zong1993} that the blocking number of any planar convex disc equals four.
Finally, in Section~\ref{section:thincone} we give a systematic exposition of thin cones, introduced in \cite{Swanepoel1999} and rediscovered and named in \cite{Furedi2005}.
We build on an idea of F\"uredi \cite{Furedi2005} to give an up-to-now best upper bound for the cardinality of a $k$-distance set in a $d$-dimensional normed space when $k$ is very large compared to $d$ (Theorem~\ref{thm:kdist}).
(This bound has very recently been improved by Polyanskii \cite{Polyanskii2016}.)

\subsection{Terminology and Notation}\label{subsection:terminology}
For background on finite-dimensional normed linear spaces from a geometric point of view, see the survey \cite{Martini2001} or the first five chapters of \cite{Thompson1996}.
We denote a normed linear space by $X$, its unit ball by $B_X$ or just $B$, and the unit sphere by $\bd B_X$.
Our spaces will almost exclusively be finite dimensional.
We will usually refer to these spaces as normed spaces or just spaces when there is no risk of confusion.
If we want to emphasize the dimension $d$ of a normed space, we denote the space by $X^d$.
We will measure distances exclusively using the norm.

We write $\un{\vx}$ for the normalization $\frac{1}{\norm{\vx}}x$ of a non-zero $\vx\in X$.
If $A, B\subseteq X$ and $\lambda\in\bR$, then we define, as usual, $A+B := \setbuilder{a+b}{a\in A, b\in B}$, $\lambda A := \setbuilder{\lambda a}{a\in A}$, $-A:= (-1)A$, and $A-B := A + (-B)$.
The interior, boundary, convex hull, and diameter of $A\subseteq X$ are denoted by $\interior A$, $\bd A$, $\conv(A)$, and $\diam(A)$, respectively.
The translate of $A$ by the vector $v\in X$ is denoted by $A+v := A+\{v\}$.

The \define{dual} of the normed space $X$ is denoted by $X^*$.
All finite-dimensional normed spaces are reflexive: $(X^*)^*$ is canonically isomorphic to $X$.
A norm $\norm{\cdot}$ is called \define{strictly convex} if $\norm{\vx+\vy}<\norm{\vx}+\norm{\vy}$ whenever $\vx$ and $\vy$ are linearly independent, or equivalently, if $\bd B_X$ does not contain a non-trivial line segment.
A norm $\norm{\cdot}$ is called \define{smooth} if it is $C^1$ away from the origin $o$, or equivalently, if each boundary point of the unit ball has a unique supporting hyperplane.
Recall that a finite-dimensional normed space is strictly convex if and only if its dual is smooth.

For $p\in[1,\infty)$, we let $\ell_p^d=(\bR^d,\norm{\cdot}_p)$ be the $d$-dimensional $\ell_p$ space with norm \[\norm{(x_1,\dots,x_d)}_p:=(\sum_{i=1}^d\abs{x_i}^p)^{1/p}\] and denote its unit ball by $B_p^d$.
The space $\ell_\infty^d=(\bR^d,\norm{\cdot})$ has norm \[\norm{(x_1,\dots,x_d)}_\infty:=\max\abs{x_i}.\]
We also denote the Euclidean space $\ell_2^d$ by $\bE^d$, the Euclidean unit ball $B_2^d$ by $B^d$, the $d$-cube $B_\infty^d$ by $C^d$, and the $d$-dimensional cross-polytope $B_1^d$ by $O^d$. 
For any two normed spaces $X$ and $Y$ and $p\in[1,\infty]$, we define their $\ell_p$-sum $X\oplus_p Y$ to be the Cartesian product $X\times Y$ with norm $\norm{(\vx,\vy)}_p :=\norm{(\norm{\vx},\norm{\vy})}_p$.

We define $\lambda(X)=\lambda(B_X)$ to be the largest length (in the norm) of a segment contained in $\bd B_X$.
It is easy to see that $0\leq\lambda(X)\leq 2$, that $\lambda(X)=0$ if and only if $X$ is strictly convex, and if $X$ is finite-dimensional, $\lambda(X)=2$ if and only if $X$ has a $2$-dimensional subspace isometric to $\ell_\infty^2$ \cite{Brass1996}.

\section{The Hadwiger number (translative kissing number) and relatives}\label{section:Hadwiger}
In the next five subsections we discuss the Hadwiger number and four of its variants: the lattice Hadwiger number, the strict Hadwiger number, the one-sided Hadwiger number and the blocking number.
See also Section~\ref{section:angle} for a derivation of these numbers for $2$-dimensional spaces.

\subsection{The Hadwiger number}\label{subsection:Hadwiger}
Let $C$ be a convex body in a finite-dimensional vector space.
A \define{Hadwiger family} of $C$ is a collection of translates of $C$, all touching $C$ and with pairwise disjoint interiors.
The \define{Hadwiger number} (or \define{translative kissing number}) $H(C)$ of $C$ is the maximum number of translates in a Hadwiger family of $C$.
(The term \define{Hadwiger number} was introduced by L.~Fejes T\'oth \cite{FT1975}.)

Denote the \define{central symmetral} of $C$ by $B:=\frac12(C-C)$.
By a well-known observation of Minkowski, $\setbuilder{\vv_i+C}{i\in I}$ is a Hadwiger family if and only if $\setbuilder{\vv_i+B}{i\in I}$ is a Hadwiger family.
Also, $\setbuilder{\vv_i+B}{i\in I}$ is a Hadwiger family if and only if $\setbuilder{\vv_i}{i\in I}$ is a collection of unit vectors in the normed space with $B$ as unit ball, such that $\norm{\vv_i-\vv_j}\geq 1$ for all distinct $i,j\in I$.
We define the Hadwiger number $H(X)$ of a finite-dimensional normed space $X$ as the Hadwiger number $H(B_X)$ of the unit ball.

The Hadwiger number of $X$ is known to be a tight upper bound for the maximum degrees of minimum-distance graphs (Section~\ref{section:mindist}) and spanning trees (Section~\ref{section:MST}) in $X$, and this is why we survey what is known about this number, updating the earlier surveys of Zong \cite{ZongSurvey1998, Zong2008} and B\"or\"oczky Jr.~\cite[\S~9.6]{Boroczky2004}.

There is a recent comprehensive survey by Boyvalenkov, Dodunekov and Musin \cite{BDM2015} on the Hadwiger number (also known as kissing number) of Euclidean balls.
We only remind the reader of the following facts.
Wyner \cite[Section~V]{Wyner}, improving on Shannon \cite{Shannon}, determined the lower bound of $H(B^d)\geq (2/\sqrt{3})^{d+\mathrm{o}(d)}$ using a greedy argument.
This lower bound is essentially still the best known (see also the end of this Section~\ref{subsection:Hadwiger} below), as is the upper bound $H(B^d)\leq 2^{0.401d+\mathrm{o}(d)}$ by Kabatiansky and Levenshtein \cite{KL1978}.
The following exact numbers are known: $H(B^3)=12$ (with a long history culminating in Sch\"utte and Van der Waerden \cite{SvdW1953}), $H(B^4)=24$ (Musin \cite{Musin2003}), $H(B^8)=240$ and $H(B^{24})=196560$ (Levenshtein \cite{Levenstein1979} and Odlyzko and Sloane \cite{OS1979}).

Hadwiger \cite{Hadwiger} showed the upper bound $H(C)\leq 3^d-1$ for all $d$-dimensional convex bodies $C$, attained by an affine $d$-cube, and by a result of Groemer \cite{Groemer} only by affine $d$-cubes.
In particular, the Hadwiger number of a parallelogram is $8$.
Gr\"unbaum \cite{Grunbaum1961}, answering a conjecture of Hadwiger \cite{Hadwiger}, showed that $H(C)=6$ for any planar convex body $C$ that is not a parallelogram.
The non-trivial part is showing the upper bound $H(C)\leq 6$.
In Section~\ref{section:angle} we show how this follows from the existence of an angular measure introduced by Brass~\cite{Brass1996}.

Gr\"unbaum \cite{Grunbaum1961} conjectured that $H(C)$ is an even number for all convex bodies, as it is in the plane, but this turned out to be false.
Talata (unpublished) constructed a $3$-dimensional polytope with Hadwiger number $17$, and Jo\'os \cite{Joos2008} constructed one with Hadwiger number $15$.

Robins and Salowe \cite{Robins1995} showed that the octahedron has Hadwiger number $18$ (this was also independently discovered by Larman and Zong \cite{Larman1999} and Talata \cite{Talata1999b}).
Larman and Zong~\cite{Larman1999} showed that the rhombic dodecahedron has Hadwiger number $18$, and also gave results for certain elongated octahedra.
Robins and Salowe \cite{Robins1995} also obtained lower bounds for $\ell_p$-balls, in particular $H(\ell_1^d)\geq 2^{0.0312\dots d-\mathrm{o}(d)}$ and 
$H(\ell_p^d)\geq (2-\epsi_p)^d$ for all $p\in(1,\infty)$, where $\epsi_p\in(0,1)$ and $\epsi_p\to 0$ as $p\to\infty$; the latter was rediscovered by Xu \cite[Theorem~4.2]{Xu2007}, who also obtained some (weaker) constructive bounds from algebraic geometry codes.
Slightly better bounds for $p\leq 2$ and close to $2$ can be found in \cite{Swanepoel1999b}.
Larman and Zong \cite{Larman1999} also showed $H(\ell_p^d)\geq (9/8)^{d+\mathrm{o}(d)}$.
It follows from the main result in Talata \cite{Talata2000} that $H(\ell_1^d)\geq 1.13488^{d+\mathrm{o}(d)}$ (see the next paragraph).

Proving a conjecture of Zong \cite{Zong1996}, Talata \cite{Talata1999} showed that the Hadwiger number of the tetrahedron is $18$.
(This equals the Hadwiger number of the central symmetral of a tetrahedron, which is the affine cuboctahedron.)
Talata \cite{Talata2000} found a lower bound of  $1.13488^{d+\mathrm{o}(d)}$ for the $d$-dimensional simplex and more generally for \define{$d$-orthoplexes}, that is, the intersection of a $(d+1)$-dimensional cube with a hyperplane orthogonal to a diagonal).
Since the difference body of a $d$-dimensional simplex is the hyperplane section of a $(d+1)$-dimensional cross-polytope through its centre parallel to a facet, this also gives the best-known lower bound for the $\ell_1$-norm, as mentioned in the previous paragraph.
Talata \cite{Talata2000} conjectured an upper bound of $1.5^{d-\mathrm{o}(d)}$ for the $d$-dimensional simplex.

The inequality \[H(C_1\times C_2) \geq (H(C_1)+1)(H(C_2)+1)-1\] for the Cartesian product of the convex bodies $C_1$ and $C_2$ is straightforward.
Zong \cite{Zong1997} showed that equality holds if either $C_1$ or $C_2$ is at most $2$-dimensional, and presented some more general conditions where equality holds.
Talata \cite{Talata2005} gave examples of convex bodies $C_1$ and $C_2$ for any dimensions larger than $2$ for which this inequality is strict.
In the same paper he constructed strictly convex $d$-dimensional bodies $C$ such that $H(C)\geq\upOmega(7^{d/2})$ and made the following two conjectures:
\begin{conjecture}[Talata \cite{Talata2005}]
In each pair of dimensions $d_1,d_2\geq 3$ there exist $d_1$-dimensional convex bodies $K_1, K'_1$ and $d_2$-dimensional convex bodies $K_2, K'_2$ such that  $H(K_1)=H(K'_1)$ and $H(K_2)=H(K'_2)$, but $H(K_1\times K_2)\neq H(K'_1\times K'_2)$.
\end{conjecture}
\begin{conjecture}[Talata \cite{Talata2005}]\label{conj:talata}
There exists a constant $c>0$ such that $H(C)\leq (3-c)^d$ for all strictly convex $d$-dimensional convex bodies.
\end{conjecture}
By an old result of Swinnerton-Dyer \cite{Swinnerton-Dyer1953}, $H(B)\geq d^2+d$ for all $d$-dimensional $B$.
For $d=2,3$ the Euclidean ball attains this bound.
However, for sufficiently large $d$ it turns out that the Hadwiger number grows exponentially in $d$, independent of the specific body.
Bourgain (as reported in \cite{Furedi1994}) and Talata \cite{Talata1998} showed the existence of an exponential lower bound by using Milman's Quotient-Subspace Theorem \cite{Milman1985}.
An explicit exponential lower bound of $H(B)\geq \upOmega((2/\sqrt{3})^d)$ for any $d$-dimensional convex body $B$ follows from Theorem~1 in Arias-de-Reyna, Ball, and Villa \cite{Arias-de-Reyna1998}.
Note that this is essentially as large as the best known lower bound for the $d$-dimensional Euclidean ball found by Wyner~\cite{Wyner}.

\subsection{Lattice Hadwiger number}\label{subsection:L}
The \define{lattice Hadwiger number} $H_L(C)$ of a convex body $C$ is defined to be the largest size of a Hadwiger family $\setbuilder{\vv_i+C}{i\in I}$ of $C$ that is contained in a lattice packing $\setbuilder{\vv+C}{\vv\in \Lambda}$, where $\Lambda$ is a full-dimensional lattice.
By the observation of Minkowski mentioned in Section~\ref{subsection:Hadwiger}, $H_L(C)=H_L(B)$ where $B$ is the central symmetral of $C$.
We also define the lattice Hadwiger number of a finite-dimensional normed space $X$ as $H_L(X)=H_L(B_X)$.
The lattice Hadwiger number plays a role in bounding the maximum number of edges of a minimum-distance graph in $X$ (Section~\ref{section:mindistedges}).

Minkowski \cite{Minkowski} already showed that $H_L(C)\leq 3^d-1$ and $H_L(C)\leq 2(2^d-1)$ if $C$ is strictly convex.
It is easily observed that $H(C)=H_L(C)$ for planar convex bodies $C$, and for the $d$-dimensional cube, $H_L(C^d)=H(C^d)=3^d-1$.
The result of Swinnerton-Dyer \cite{Swinnerton-Dyer1953} mentioned earlier, actually shows that $H_L(C)\geq d^2+d$ for all $d$-dimensional convex bodies $C$.
This seems to be the best-known lower bound valid for all convex bodies.
Zong \cite{Zong2008} posed the problem to show that for all $d$-dimensional convex bodies $C$, $H_L(C)\geq\upOmega(c^d)$ for some constant $c>1$ independent of $d$.
The best asymptotic lower bound for the Euclidean ball is $H_L(B^d)\geq 2^{\upOmega(\log^2 d)}$, attained by the Barnes--Wall lattice, as shown by Leech \cite{Leech1964}.

Zong \cite{Zong1996} determined the lattice Hadwiger number of the tetrahedron $T$ in $3$-space: $H_L(T)=18$, and determined a lower bound of $d^2+d+6\lfloor d/3\rfloor$ for simplices.
For Euclidean space these numbers are known up to dimension $9$ (Watson \cite{Watson1971}): $H_L(B^3)=12$, $H_L(B^4)=24$, $H_L(B^5)=40$, $H_L(B^6)=72$, $H_L(B^7)=126$, $H_L(B^8)=240$, $H_L(B^9)=272$.
In particular, $H_L(\bE^9)=272 < 306 \leq H(\bE^9)$ is the smallest dimension where $H$ and $H_L$ differ for a Euclidean ball (although they are equal in dimension $24$).
Zong \cite{Zong1994} showed that in each dimension $d\geq 3$ there exists a convex body $C$ such that $H(C) > H_L(C)$.
His example is a $d$-cube with two opposite corners cut off.
Recall that Talata \cite{Talata2005} constructed $d$-dimensional strictly convex bodies $C$ with $H(C)\geq\upOmega(7^{d/2})$.
When compared with Minkowski's upper bound $H_L(C)\leq 2(2^d-1)$ for all strictly convex bodies $C$, this shows that the gap between $H(C)$ and $H_L(C)$ can be very large, even for strictly convex sets (see also \cite{Talata1998b}).

\subsection{Strict Hadwiger number}
A \define{strict Hadwiger family} of $C$ is a collection of translates of $C$, all touching $C$ and all pairwise disjoint (that is, no two overlap or touch).
The \define{strict Hadwiger number} $H'(C)$ of $C$ is the maximum number of translates in a strict Hadwiger family of $C$.
We also define the strict Hadwiger number of a finite-dimensional normed space $X$ as $H'(X)=H'(B_X)$.
Clearly, $H'(C)\leq H(C)$, and it is not difficult to see that the strict Hadwiger number of the $d$-dimensional cube is $H'(C^d)=2^d$.

Doyle, Lagarias, and Randall \cite{Doyle1992} showed that $H'(C)=5$ if $C$ is a planar convex body that is not a parallelogram.
(Robins and Salowe \cite{Robins1995} observed that $H'(C^2)=4$ for the parallelogram $C^2$).
See Section~\ref{section:angle} for a simple proof of this fact using angular measures.

Robins and Salowe \cite{Robins1995} studied $H'(X)$ in connection to minimal spanning trees in a finite-dimensional normed space $X$; see Section~\ref{section:MST}.
For the $3$-dimensional Euclidean ball $B^3$, $H'(B^3)=12$, as demonstrated by the many configurations of $12$ pairwise non-touching balls, all touching a central ball~\cite{KKLS2017}.
Robins and Salowe \cite{Robins1995} showed that for the regular octahedron $O^3$, $13\leq H'(O^3)\leq 14$, and that for each $d\geq 3$ there exists $p\in(1,\infty)$ such that $H'(\ell_p^d)>2^d=H'(C^d)$.
Talata \cite{Talata1998} showed that there is also an exponential lower bound for $H'$, and the explicit exponential lower bound of $H'(C)\geq \upOmega((2/\sqrt{3})^d)$ also follows from the results of Arias-de-Reyna, Ball, and Villa~\cite{Arias-de-Reyna1998} mentioned at the end of Section~\ref{subsection:Hadwiger}.

Talata \cite{Talata2005} studied $H'$ for Cartesian products of convex bodies.
In particular, he showed that if $C_1,\dots,C_n$ are convex discs, with $k$ parallelograms among them, then \[H'(C_1\times C_2\times\dots\times C_n) = 4^k(4\cdot 6^{n-k}+1)/5.\]
He also showed that there exist $d$-dimensional convex bodies $K_d$ for which $H'(K_d)=\upOmega(7^{d/2})$, from which
his example of a strictly convex body with Hadwiger number $\upOmega(7^{d/2})$ follows.
Indeed, given any convex body with a strict Hadwiger configuration, it is easy to modify the convex body so that it becomes strictly convex and the Hadwiger configuration stays strict.
Hence, Conjecture~\ref{conj:talata} would imply that $H'(C)\leq (3-c)^d$ for any $d$-dimensional convex body $C$.

\subsection{One-sided Hadwiger number}\label{section:onesided}
The \define{one-sided Hadwiger number} $H_{+}(C)$ of a convex body $C$ is the maximum number of translates in a Hadwiger family $\setbuilder{\vv_i+C}{i\in I}$ such that $\setbuilder{\vv_i}{i\in I}$ is contained in a closed half space with the origin on its boundary.
We also define the one-sided Hadwiger number $H_+(X)$ of the normed space $X$ to be the $H_+(B_X)$.
Clearly, for the circular disc $B^2$ we have $H_{+}(B^2)=4$.
It is easy to show that $H_{+}(B)=4$ for any convex disc $B$ except the parallelogram $C^2$, where $H_{+}(C^2)=5$ (see Section~\ref{section:angle} for proofs).
The \define{open one-sided Hadwiger number} $H_{+}^{o}(C)$ of $C$ is defined similarly by replacing `closed half space' by `open half space' in the definition.
We also define the open one-sided Hadwiger number $H_+^o(X)$ of the normed space $X$ to be the $H_+^o(B_X)$.
The open one-sided Hadwiger number bounds the minimum degree of a minimum-distance graph in $X$ (Section~\ref{section:mindistmindeg}).
 It is not hard to show that $H_{+}^{o}(X^2)=3$ for any normed plane $X^2$ with $\lambda(X^2)\leq 1$, and $H_{+}^{o}(X^2)=4$ otherwise (see again Section~\ref{section:angle} for proofs).
G.~Fejes T\'oth \cite{GFT1981} showed that for the $3$-dimensional Euclidean ball $B^3$ we have $H_{+}(B^3)=9$ (see also Sachs \cite{Sachs1986} and A.~Bezdek and K.~Bezdek~\cite{Bezdek1988} for alternative proofs).
Kert\'esz \cite{Kertesz1994} showed that $H_{+}^{o}(B^3)=8$.
Musin~\cite{Musin2006} showed that for the $4$-dimensional Euclidean ball, $H_{+}(B^4)=18$.
(K.~Bezdek~\cite{Bezdek2006} observed that it follows from Musin's determination of $H(B^4)$ \cite{Musin2003} that $18\leq H_{+}(B^4)\leq 19$.)
Bachoc and Vallentin \cite{Bachoc2009} found the exact value $H_{+}(B^8) = 183$ and upper bounds for Euclidean spaces of dimension up to $10$, improving earlier bounds by Musin~\cite{Musin2008}.

K.~Bezdek and Brass \cite{Bezdek2003} showed that $H_{+}(C)\leq 2\cdot 3^{d-1}-1$ for any $d$-dimensional convex body $C$, with equality attained only by the affine $d$-cube $C^d$.
They ask as an open problem for a tight upper bound of $H_{+}^{o}(C)$ valid for all $d$-dimensional convex bodies.
L\'angi and Nasz\'odi \cite{Langi2009} generalized some of the results in \cite{Bezdek2003}.

\subsection{Blocking number}
Zong \cite{Zong1993} introduced the blocking number of a convex body: the minimum number of non-overlapping translates of $C$, all touching $C$, and such that no other translate can touch $C$ without overlapping some of these translates.
Equivalently, the \define{blocking number} $B(C)$ of a convex body $C$ is the minimum number of translates in a maximal Hadwiger family of $C$.
The \define{strict blocking number} $B'(C)$ of $C$ is the minimum size of a maximal strict Hadwiger family of $C$.
Thus clearly, $B(C)\leq H(C)$ and $B'(C)\leq H'(C)$.

Zong \cite{Zong1993} showed that the blocking number of any convex disc equals $4$.
In Section~\ref{section:angle} we give a simple proof of this result, we determine the strict blocking number of all convex discs, and present some related results, all using angular measures.
Dalla, Larman, Mani-Levitska, and Zong \cite{DLMZ} determined the blocking numbers of the $3$- and $4$-dimensional Euclidean balls and of all cubes: $b(B^3)=6$, $b(B^4)=9$, $b(C^d)=2^d$.
For further results on blocking numbers, see Yu \cite{Yu2009}, Yu and Zong \cite{YZ2009}, and Zong \cite{Zong1995, ZongBook, ZongSurvey1998, Zong2008}.

\section{Equilateral sets}\label{section:equilateral}
A set of $S$ of points in a normed space $X$ is \define{equilateral} if $\norm{\vx-\vy}=1$ for any distinct $\vx,\vy\in S$.
Let $e(X)$ denote the largest size of an equilateral set of points in $X$ if it is finite.
Here we emphasize results that appeared after the survey \cite{Swanepoel2004}.

Petty \cite{Petty1971} and Soltan \cite{Soltan1975} observed that it follows from a celebrated result of Danzer and Gr\"unbaum \cite{Danzer1962} that $e(X)\leq 2^d$ for all $d$-dimensional $X$, and that equality holds iff $X$ is isometric to $\ell_\infty^d$ (equivalently, iff the unit ball is an affine $d$-cube).

The following conjecture has been made often \cite{Petty1971, Morgan, Thompson1996} (see also \cite{Grunbaum1961}):
\begin{conjecture}[Petty \cite{Petty1971}]\label{conj:equilateral}
For all $d$-dimensional $X$, $e(X)\geq d+1$.
\end{conjecture}
It is simple to see that this conjecture holds for $d=2$.
Petty \cite{Petty1971} established it for $d=3$.
He in fact proved that in any normed space of dimension at least $3$, any equilateral set of $3$ points can be extended to an equilateral set of $4$ points.
His proof uses the topological fact that the plane with a point removed is not simply connected.
V\"ais\"al\"a \cite{Vaisala2012} gave a more elementary proof that only uses the connectedness of the circle.
(Kobos \cite{Kobos2013} also gave an alternative proof that depends on the $2$-dimensional case of the Brouwer Fixed-Point Theorem.)
Makeev \cite{Makeev2007} showed that the conjecture is true for $d=4$.
Brass~\cite{Brass1999} and Dekster~\cite{Dekster2000} used the Brouwer Fixed-Point Theorem to show that the conjecture holds for spaces sufficiently close to $\bE^d$.
Swanepoel and Villa \cite{Swanepoel2008} used a variant of that argument to show that it holds for spaces sufficiently close to $\ell_\infty^d$.
Kobos \cite{Kobos2014} showed that the conjecture holds for norms on $\bR^d$ for which the norm is invariant under permutation of the coordinates, as well as for $d$-dimensional subspaces of $\ell_\infty^{d+1}$ and spaces sufficiently close to them.
There has also been work on bounding $e(X^d)$ from below in terms of $d$.
Brass \cite{Brass1999} and Dekster \cite{Dekster2000} combined their previously mentioned result on spaces close to Euclidean space with Dvoretzky's Theorem to show that $e(X^d)$ is bounded below by an unbounded function of the dimension.
In fact, their proof, when combined with the best known dimension \cite{Schechtman2006} in Dvoretzky's Theorem gives a lower bound $e(X^d)\geq\upOmega(\sqrt{\log d}/\log\log d)$.
Swanepoel and Villa \cite{Swanepoel2008} showed that $e(X^d)\geq\exp(\upOmega(\sqrt{\log d}))$ by using, instead of Dvoretzky's Theorem, a theorem of Alon and Milman \cite{Alon1983} on subspaces close to $\bE^d$ or $\ell_\infty^d$, together with a version of Dvoretzky's Theorem for spaces not far from Euclidean space, due to Milman \cite{Milman1971}.
Roman Karasev (personal communication), in the hope of finding a counterexample to Conjecture~\ref{conj:equilateral}, asked whether the above conjecture holds for $\bE^a\oplus_1\bE^b$, the $\ell_1$-sum of two Euclidean spaces.
The special case $(a,b)=(1,d-1)$ was considered by Petty \cite{Petty1971} (see also Section~\ref{section:maximal} below).
It is not difficult to show that for Petty's space we have $e(\bR\oplus_1\bE^{d-1})\geq d+1$ \cite{Swanepoel2004}.
Joseph Ling~\cite{Ling2006} has shown that for this space, $e(\bR\oplus_1\bE^{d-1})\leq d+2$ and that equality holds for all $d\leq 10$.
Aaron Lin~\cite{Lin} showed that $e(\bE^a\oplus_1\bE^b)\geq a+b+1$ for all $a\leq b$ such that $a\leq 27$ or $b\equiv 0,1,a\pmod{a+1}$ among other cases, with the open cases of lowest dimension being $\bE^{28}\oplus_1\bE^{40}$ and $\bE^{29}\oplus_1\bE^{39}$.
There are other results that cast some doubt on Conjecture~\ref{conj:equilateral}: the existence of small maximal equilateral sets (Section~\ref{section:maximal}) and the existence of infinite-dimensional normed spaces that do not have infinite equilateral sets, first shown by Terenzi \cite{Terenzi1987, Terenzi1989}; see also Glakousakis and Mercourakis \cite{Glakousakis}.
(For more on equilateral sets in infinite-dimensional space, see \cite{FOSS2014, Koszmider, Mercourakis2015, Mercourakis2014}.)

Gr\"unbaum \cite{Gr} showed that for a strictly convex space of dimension $3$, $e(X)\leq 5$.
Building on his work, Sch\"urmann and Swanepoel \cite{Schurmann2006} determined $e(X)$ for various $3$-dimensional spaces, and in particular showed the existence of a smooth $3$-dimensional space with $e(X)=6$.
They showed that $6$ is the maximum for smooth norms in dimension $3$ and characterized the $3$-dimensional norms that admit equilateral sets of $6$ and $7$ points (see also Bisztriczky and B\"or\"oczky \cite{BB2005} for more general results).

We say that a set $S$ of points in $\bR^d$ is \define{strictly antipodal} if for any two distinct $x,y\in S$ there exist distinct parallel hyperplanes $H_x$ and $H_y$ such that $x\in H_x$, $y\in H_y$, and $A\setminus\set{x,y}$ is contained in the open slab bounded by $H_x$ and $H_y$.
Let $A'(d)$ denote the largest size of a strictly antipodal set in a $d$-dimensional space \cite{martini-soltan-2005}.
It is easy to see that $e(X^d)\leq A'(d)$ for all strictly convex $X^d$, and that there exists a strictly convex and smooth $X^d$ such that $e(X^d)=A'(d)$.
Erd\H{o}s and F\"uredi \cite{ErdosFuredi1983} showed that $A'(d)\geq\upOmega((2/\sqrt{3})^d)$, thus implying that there exist strictly convex $d$-dimensional normed spaces $X^d$ with $e(X^d)\geq \upOmega(2/\sqrt{3})^d$.
Talata improved this by a construction (described in \cite[Lemma~9.11.2]{Boroczky2004}) to $A'(d)\geq\upOmega(3^{d/3})$, and announced that $A'(d)\geq\upOmega(5^{d/4})$ (see \cite[p.~271]{Boroczky2004}).
Subsequently, Barvinok, Lee, and Novik \cite{Barvinok2013} found another construction that shows $A'(d)\geq\upOmega(3^{d/2})$.
This is currently the best-known bound for $e(X^d)$ for strictly convex spaces.
\begin{conjecture}[Erd\H{o}s and F\"uredi \cite{ErdosFuredi1983}]\label{conj:EF}
There exists $c>0$ such that for all $d$-dimensional strictly convex spaces $X^d$, $e(X^d)\leq (2-c)^d$.
\end{conjecture}
In fact, there is no known proof even that $e(X^d)\leq 2^d-2$ for all strictly convex $X^d$, except in dimensions $d\leq 3$ (Gr\"unbaum \cite{Gr}).
It might also be interesting to look at rounded cubes such as the following.
For small $\epsi>0$, let $X^d$ have as unit ball the rounded $d$-cube $B_{\infty}^d+\epsi B_2^d$.
This space is smooth, but not strictly convex.
Using results from \cite{Schurmann2006} it can be shown that $e(X^3)=5$ for all sufficiently small $\epsi>0$.
Thus, there exist three-dimensional smooth spaces arbitrarily close to $\ell_\infty^3$, and with $e(\ell_\infty^3)-e(X^3)=3$.
It might be that for small $\epsi$, $e(X^d)$ is very far from $e(\ell_\infty^d)=2^d$, possibly even linear in $d$.
\begin{conjecture}
For some constant $C>0$, for each $d\in\bN$ there exists an $\epsi>0$ such that $e(X^d)\leq Cd$, where $X^d$ is the normed space with unit ball $B_{\infty}^d+\epsi B_2^d$.
\end{conjecture}
Note that the dual of $\ell_\infty^d$ is $\ell_1^d$, for which Alon and Pudl\'ak \cite{Alon2003} has shown $e(\ell_1^d)=\mathrm{O}(d\log d)$.
We propose the following conjecture:
\begin{conjecture}
For any $d$-dimensional normed space $X$ with dual $X^*$, $e(X)e(X^*)\leq 2^{d+\mathrm{o}(d)}$.
\end{conjecture}
It would already be interesting to show that $e(X)e(X^*)=\mathrm{o}(4^d)$.

\subsection[\texorpdfstring{Equilateral sets in $\ell_\infty$-sums and the Borsuk problem}{Equilateral sets in L-infinity-sums and the Borsuk problem}]{Equilateral sets in $\ell_\infty$-sums and the Borsuk problem}\label{section:Borsuk}
We next consider $\ell_{\infty}$-sums of normed spaces.
If $X$ and $Y$ are normed spaces, then the unit ball of $X\oplus_\infty Y$ is the Cartesian product $B_X\times B_Y$.
It is easy to see that $e(X\oplus_\infty Y)\geq e(X)e(Y)$.
For certain $X$ and $Y$ it is possible to show that equality holds.
The \define{Borsuk number} $b(X)$ of $X$ is defined to be the smallest $k$ such that any subset of $X$ of diameter $1$ can be partitioned into $k$ parts, each of diameter strictly smaller than $1$.
This notion was introduced by Gr\"unbaum~\cite{MR21:2209}.
The Borsuk number of Euclidean space received the most attention, ever since Borsuk \cite{Borsuk} conjectured that $b(\bE^d)$ equals $d+1$.
It is known that $b(\bE^d)=e(\bE^d)$ for $d=2$ (Borsuk~\cite{Borsuk}) and $d=3$ (Perkal \cite{Perkal1947} and Eggleston \cite{Eggleston1955}), $b(\bE^d)\geq(1.203\dots+\mathrm{o}(1))^{\sqrt{d}}$ (Kahn and Kalai \cite{KK1993}), $b(\bE^d)\leq 2^{d-1}+1$ (Lassak \cite{Lassak1982}), and $b(\bE^d)\leq (\sqrt{3/2}+\mathrm{o}(1))^{d}$ (Schramm \cite{Schramm1988} and Bourgain and Lindenstrauss \cite{BL1991}).
Currently, the smallest dimensions for which Borsuk's conjecture is known to be false, are $b(\bE^{65})\geq 83$ (Bondarenko \cite{Bondarenko2014}) and $b(\bE^{64})\geq 71$ (Jenrich and Brouwer \cite{Jenrich2014}).
See Raigorodskii \cite{Raigorodskii2007} and Kalai \cite{Kalai2015} for recent surveys.
Clearly, $b(X)\geq e(X)$, although as is shown by the counterexamples to Borsuk's conjecture, these two quantities are very different already for Euclidean spaces.
On the other hand, it is easy to see that $b(\ell_\infty^d)=e(\ell_\infty^d)=2^d$.
Gr\"unbaum~\cite{MR21:2209} showed that $b(X^2)=e(X^2)$ for all $2$-dimensional spaces.

Zong \cite{Zong2008} asked whether $b(X^d)\leq 2^d$ for all $d$-dimensional $X^d$.
It is well known \cite{Rogers-Zong} that a $d$-dimensional convex body $K$ can be covered by $\mathrm{O}(2^d d\log d)$ translates of $-(1-\epsi)K$, where $\epsi>0$ is arbitrarily small.
It follows that $b(X^d)\leq \mathrm{O}(2^d d\log d)$ for all $d$-dimensional~$X^d$.

We define the following variant for finite subsets of $X$.
Let the \define{finite Borsuk number} $b_f(X)$ of $X$ be the smallest number $k$ such that any finite subset of $X$ of diameter $1$ can be partitioned into $k$ parts, each of diameter strictly smaller than $1$.
Then $b(X)\geq b_f(X)$, although we have no evidence either way whether these two quantities can differ for some $X$ or not, although we note that $b_f(\ell_\infty^d)=2^d$ and $b_f(X^2)=b(X^2)=e(X^2)$ for any two-dimensional space $X^2$.
\begin{proposition}\label{prop:equilBorsuk}
For any two finite-dimensional normed spaces $X$ and $Y$,
\[ e(X)e(Y)\leq e(X\oplus_\infty Y)\leq e(X)b_f(Y).\]
\end{proposition}
\begin{proof}
If $S$ is an equilateral set in $X$, and $T$ an equilateral set in $Y$, with equal distances, then $S\times T$ is equilateral in $X\oplus_\infty Y$.
This shows the first inequality.
For the second inequality, let $E$ be an equilateral set with distance $1$ in $X\oplus_\infty Y$, and let $\pi_Y\colon X\oplus_\infty Y\to Y$ be the projection onto the second coordinate.
Then $\pi_Y(E)$ has diameter at most $1$ in $Y$, so can be partitioned into $k\leq b_f(Y)$ parts $E_1,\dots,E_k$, each of diameter $<1$.
It follows that $\pi_Y^{-1}(E_1),\dots,\pi_Y^{-1}(E_k)$ is a partition of $E$, and each $\pi_X(\pi_Y^{-1}(E_i))$ is equilateral.
Finally, note that $\card{\pi_Y^{-1}(E_i)} = \card{\pi_X(\pi_Y^{-1}(E_i))}$ for each $i$.
The second inequality follows.
\end{proof}
\begin{corollary}
If $X$ and $Y$ are finite-dimensional normed spaces, and one of $X$ or $Y$ is at most $2$-dimensional or Euclidean $3$-space $\bE^3$, then $e(X\oplus_\infty Y) = e(X)e(Y)$.
\end{corollary}
Perhaps the simplest $\ell_\infty$-sum for which this corollary does not determine $e(X)$ is the $\ell_{\infty}$-sum of two $4$-dimensional Euclidean spaces, with unit ball the Cartesian product of two $4$-dimensional Euclidean balls.
If $b_f(\bE^4)$ were equal to $5$, then Proposition~\ref{prop:equilBorsuk} would give that $e(\bE^4\oplus_\infty\bE^4)=25$.
Most likely it would be easier to determine the value of $e(\bE^4\oplus_\infty\bE^4)$ than to settle Borsuk's conjecture in Euclidean $4$-space.

\subsection{Small maximal equilateral sets}\label{section:maximal}
Petty \cite{Petty1971} showed that it is not always possible to extend an equilateral set of size at least $4$ to an equilateral set properly containing it.
In particular, he showed that $\bR\oplus_1\bE^{d-1}$ contains a maximal equilateral set of $4$ points for each $d\geq 3$.
Swanepoel and Villa \cite{Swanepoel2013} found many other spaces with the property of having small maximal equilateral sets.
In particular, for any $p\in[1,2)$ there exists a $C_p$ such that $\ell_p^d$ and $\ell_p$ have maximal equilateral sets of size at most $C_p$.
\begin{conjecture}[\cite{Swanepoel2013}]
Any $d$-dimensional normed space has a maximal equilateral set of size at most~$d+1$.
\end{conjecture}
This conjecture holds for all $\ell_p^d$, $p\in[1,\infty]$, and also for all spaces sufficiently close to one of these spaces \cite{Swanepoel2013}.
See also Kobos \cite{Kobos2013}, where smooth and strictly convex spaces with maximal equilateral sets of size $4$ are constructed.

\subsection{Subequilateral sets}
Lawlor and Morgan \cite{Lawlor1994} used the following weakening of equilateral sets.
A polytope $P$ in a normed space $X$ is called \define{subequilateral} if the length of each edge of $P$ equals the diameter of $P$ (in the norm)~\cite{Swanepoel2007b}.
We denote the maximum number of vertices in a subequilateral polytope in $X^d$ by $e_s(X^d)$.
For any equilateral set $S$, $\conv(S)$ is a subequilateral polytope, hence $e(X^d)\leq e_s(X^d)$.
Subequilateral polytopes were used in \cite{Lawlor1994} to construct certain energy-minimizing cones.
These polytopes turn out to be so-called edge-antipodal polytopes, introduced by Talata~\cite{Talata1999b}, who conjectured that an edge-antipodal $3$-polytope has a bounded number of vertices.
This was proved by Csik\'os \cite{Csikos2003}.
K.~Bezdek, Bistriczky and B\"or\"oczky \cite{BBB2005} determined the tight bound of $8$, which implies that $e_s(X^3)\leq 8$ for any $3$-dimensional normed space.
P\'or \cite{Por} proved the generalization of Talata's conjecture to all dimensions, by showing that for each $d$ there exists a $c_d$ such that any edge-antipodal $d$-polytope has at most $c_d$ vertices.
His proof is non-constructive, and only gives $e_s(X^d)<\infty$ for each $d$.
In \cite{Swanepoel2007b} it is shown that $e_s(X^d)\leq (1+d/2)^d$.
This in turn implies the same bound on the number of vertices of an edge-antipodal polytope.
\begin{conjecture}[\cite{Swanepoel2007b}]
A subequilateral set in a $d$-dimensional normed space has size at most $c^d$, where $c\geq 2$ is some absolute constant.
\end{conjecture}
The results of Bisztriczky and B\"or\"oczky \cite{BB2005} on edge-antipodal $3$-polytopes imply that for a strictly convex $X^3$, $e_s(X^3)\leq 5$.

\section{Minimum-distance graphs}\label{section:mindist}
Given any finite packing $\setbuilder{C+v_i}{i=1,\dots n}$ of non-overlapping translates of a $d$\nobreakdash-\hspace{0pt}dimensional convex body $C$, we define the \define{touching graph} of the packing to be the graph with a vertex for each translate, and with two translates joined by an edge if they intersect (necessarily in boundary points).
By the observation of Minkowski mentioned in Section~\ref{subsection:Hadwiger}, if $\setbuilder{C+v_i}{i=1,\dots n}$ is a packing of non-overlapping translates of $C$, then $\setbuilder{B+v_i}{i=1,\dots n}$ is a packing of non-overlapping translates of the central symmetral $B=\frac12(C-C)$ of $C$.
Since $B$ is $o$-symmetric, it is the unit ball of a $d$-dimensional normed space.
We therefore make the following definition.

Given a finite set $V$ in a normed space $X$ with minimum distance $d=\min_{x,y\in V}\norm{x-y}$, we define the \define{minimum-distance graph} of $V$ to be $G_m(V)=(V,E)$ by taking all \define{minimum distance pairs} $xy$ to be edges, that is, $xy\in E$ whenever $\norm{x-y}=d$.

We next consider a selection of parameters of these minimum-distance graphs.
As a first remark, the maximum clique number of a minimum-distance graph in $X$ equals $e(X)$, the maximum size of an equilateral set.
Note that in any $2$-dimensional normed space in which the unit ball is not a parallelogram, minimum-distance graphs are always planar.
In fact, no edge can intersect another edge in its relative interior \cite{Brass1996}.

\subsection{Maximum degree and maximum number of edges of minimum-distance graphs}\label{section:mindistedges}
The degree of any vertex in a minimum-distance graph is bounded above by the Hadwiger number $H(X)$ of $X$.
This bound is sharp when taken over all minimum-distance graphs, since the minimum-distance graph of a subset of $\bd B$, pairwise at distance at least $1$, together with the origin $o$, has degree exactly $H(X)$ at $o$.

Let $m(n,X)$ denote the maximum possible number of edges of a minimum-distance graph of $n$ points in $X$.
The above observation immediately gives the bound $m(n,X)\leq H(X)n/2$.
Erd\H{o}s~\cite{Erdos} mentioned that $m(n,\bE^2)=3n-\mathrm{O}(\sqrt{n})$.
Harborth \cite{Harborth1974}, answering a question of Reutter \cite{Reutter1972}, found the exact value $m(n,\bE^2)=\lfloor 3n-\sqrt{12n-3}\rfloor$ for all $n\geq 1$.
Brass \cite{Brass1996} showed that the same upper bound holds for all norms on $\bR^2$ except those isometric to $\ell_\infty^2$.
A key tool in his proof is the introduction of an angular measure with various properties mimicking the Euclidean angular measure (Section~\ref{section:angle}).
He also determined the maximum for $\ell_\infty^2$: $m(n,\ell_\infty^2)=\lfloor 4n-\sqrt{28n-12}\rfloor$ for all $n\geq 1$.

K.~Bezdek \cite{Bezdek2006} considered the problem of determining $m(n,\bE^3)$, and calls it the combinatorial Kepler problem.
In \cite{Bezdek2012} he showed that $6n-7.862n^{2/3} \leq m(n,\bE^3) \leq 6n - 0.695n^{2/3}$, and in \cite{Bezdek2013}, K.~Bezdek and Reid improved the upper bound to $6n - 0.926n^{2/3}$.
For more on Euclidean minimum-distance graphs, see the recent survey of K.~Bezdek and Khan \cite{BK2016}.
We next show how an isoperimetric argument gives a slight improvement to the bound $m(n,X)\leq H(X)n/2$.
\begin{proposition}
For any $d$-dimensional normed space $X^d$, $m(n,X^d)\leq H(X)n/2 - c_d n^{1-1/d}$, where $c_d>0$ depends only on $d$.
\end{proposition}
\begin{proof}
Consider a set $V$ of $n$ points in $X^d$ with unit ball $B=B_X$.
Let $G=(V,E)$ be the minimum-distance graph of $V$.
Without loss of generality, the minimum distance may be taken as $1$.
We may identify $X^d$ with $\bR^d$ in such a way that the ellipsoid of maximum volume contained in $B$ is the Euclidean ball $B^d$.
Then a reverse isoperimetric inequality of Ball \cite{Ball1991} states that
\begin{equation}\label{eq1}
\lambda_{d-1}(\bd B)\leq 2d\lambda_d(B)^{1-1/d},
\end{equation}
where we denote $k$-dimensional Lebesgue measure in $\bR^d$ by $\lambda_k$.

Let $W\subseteq V$ denote the set of all vertices of degree $<H(X)$.
Let $S=\bigcup_{v\in V} (B+v)$.
Then clearly, $\bd S\subseteq\bigcup_{v\in V} (\bd B+v)$.
We claim that $\bd S\subseteq\bigcup_{v\in W} (\bd B+v)$.
Indeed, let $x\in\bd S$, say $x\in\bd B+v_0$.
Since for each neighbour $v$ of $v_0$, $x\notin\interior B+v$, we have $\norm{x-v}\geq 1$.
It follows that any two points in $\setbuilder{v}{vv_0\in E}\cup\{x\}\subset B+v_0$ are at distance at least $1$.
Therefore, the degree of $v_0$ is strictly smaller than $H(X)$, hence $v_0\in W$.

It follows that
\begin{equation}\label{eq2}
\lambda_{d-1}(\bd S)\leq\card{W}\lambda_{d-1}(\bd B).
\end{equation}
Since the balls $\setbuilder{\frac12 B + v}{v\in V}$ form a packing and are contained in $S$, we have
\begin{equation}\label{eq3}
\lambda_d(S)\geq\card{V}\lambda_d(B)/2^d.
\end{equation}
By the isoperimetric inequality,
\begin{equation}\label{eq4}
\lambda_{d-1}(\bd S)\geq d\kappa_d^{1/d} \lambda_d(S)^{1-1/d},
\end{equation}
where $\kappa_d=\lambda_d(B^d)$ is the volume of the Euclidean unit ball.
If we put \eqref{eq1}--\eqref{eq4} together, we obtain $\card{W}\geq(\kappa_d^{1/d}/2^d) \card{V}^{1-1/d}$, and since $\card{E}\leq \frac12 \left(H(X)\card{V} - \card{W}\right)$, the Proposition follows.
\end{proof}

K.~Bezdek \cite{Bezdek2002} derived an upper bound with an improved $n^{1-1/d}$ term which also involves the density of a densest translative packing of $B_X$.
The main problem though, already in the Euclidean case, is the coefficient of $n$ in this upper bound, even when the Hadwiger number is known.
Indeed, if we consider a lattice packing of the unit ball, we obtain the following obvious lower bound in terms of the lattice Hadwiger number: $m(n,X)\geq H_L(X)n/2 - \upOmega(n^{1-1/d})$.
Therefore, whenever $H(X)=H_L(X)$, we have that $m(n,X)=H(X)n/2- \upTheta(n^{1-1/d})$.
However, when these numbers differ, for instance in $9$-dimensional Euclidean space where $H_L(\bE^9)<H(\bE^9)$, we do not even know the main term.
When $d>2$ and $n$ is large, it is also not clear if point sets that maximize $m(n,X)$ have to be pieces of lattices for which $H_L$ is attained.

\subsection{Minimum degree of minimum-distance graphs}\label{section:mindistmindeg}
Let $\delta(X)$ denote the largest minimum degree of a minimum-distance graph in $X$.
That is, \[\delta(X)=\max\setbuilder{\delta(G)}{\text{$G$ is a minimum-distance graph in $X$}}.\]
We can also define $\delta(X)$ as the largest $k$ such that all minimum-distance graphs have a vertex of degree at most $k$.
Another description found in the literature is the following.
A finite packing of translates of a convex body $C$ is called a \define{$k^+$-neighbour packing} if each translate has at least $k$ neighbours.
Then $\delta(X)$ is the largest $k$ such that there exists a finite $k^+$-neighbour packing of translates of the unit ball $B_X$.

By considering a vertex of the convex hull of the set of points, we see that $\delta(X)\leq H_+^o(X)$.
Even in $2$-dimensional spaces, there may be strict inequality.
For example, if the unit ball is a square with two opposite corners truncated a bit, then $\delta(X^2)=3$ by a result of Talata \cite{Talata2002} (Theorem~\ref{thm:talata} below), but $H_+^o(X^2)=4$ (Proposition~\ref{hadwprop} below).
Also, $\delta(\bE^3)\leq H_+^o(\bE^3)=8$ by the result of Kert\'esz \cite{Kertesz1994} mentioned in Section~\ref{section:onesided}, but it is unknown whether equality holds.
The best known lower bound $\delta(\bE^3)\geq 6$ is due to a construction of G.~Wegner of a $6$-regular minimum-distance graph on $240$ points in $\bE^3$, described in \cite{FTKu1993}.

Most of the results on $\delta(X)$ were obtained by Talata.
In \cite{Talata2002} he showed that $\delta(X^2)=3$ if $X^2$ is not isometric to $\ell_\infty^2$, and $\delta(\ell_\infty^2)=4$.
In Section~\ref{section:angle} we give a simple proof of this fact.
He also determined $\delta(X^2\oplus_\infty\bR)=10$ if $X^2$ is not isometric to $\ell_\infty^2$, and $\delta(\ell_\infty^3)=13$.
In~\cite{Talata2011} he considered $\delta(X)$ for an arbitrary finite-dimensional normed space, and showed that $\delta(X)\geq H_L(X)/2$, which implies the above-mentioned result of Wegner that $\delta(\bE^3)\geq 6$.
Talata also showed that $\delta(X) = H_L(X)/2$ if $X$ is the $\ell_\infty$-sum of spaces of dimension at most $2$, or equivalently, if the unit ball is the Cartesian product of segments and centrally symmetric convex discs.
In \cite{Talata2006} he showed that equality still holds if $X$ is the $\ell_\infty$ sum of spaces of dimension at most $2$ or $\ell_1^3$.
In particular, $\delta(\ell_1^3)=9$ and $\delta(\ell_\infty^d)=(3^d-1)/2$.
As mentioned in Section~\ref{subsection:L}, for high-dimensional Euclidean space the best-known lower bound for the lattice Hadwiger number is not particularly strong: $H_L(\bE^d)\geq 2^{\upOmega(\log^2 d)}$.
Alon~\cite{Alon1997b} improved the corresponding bound for $\delta(\bE^d)$ by showing that $\delta(\bE^d)\geq 2^{\sqrt{d}}$ if $d$ is a power of $4$, hence $\delta(\bE^d)\geq 2^{\sqrt{d}/2}$ in general.
(See also the stronger conjecture of Chen \cite{Chen2016} in Section~\ref{section:mindist-chr} below).
Talata \cite{Talata2011} conjectured that $\delta(X)\leq H(X)/2$, which holds in dimension $2$.
In both papers \cite{Talata2002, Talata2011}, Talata also estimated the smallest number of points in a minimum-distance graph with minimum degree $\delta(X)$.
In \cite{Talata2011} he considered a lattice version of $\delta(X)$.

\subsection{Chromatic number and independence number of minimum-distance graphs}\label{section:mindist-chr}
Let $\chi_m(X)$ denote the largest chromatic number of a minimum-distance graph in $X$ and $\alpha_m(n,X)$ the smallest independence number of a minimum-distance graph on $n$ points in $X$.
Then $\chi_m(X)\alpha_m(n,X)\geq n$.
Also, $\chi_m(X)\leq\delta(X)+1$, hence $\alpha_m(n,X)\geq n/(\delta(X)+1)$ \cite[Theorem~2]{BNV2003}.
Talata's conjecture above in Section~\ref{section:mindistmindeg} would imply the upper bound $\chi_m(X)\leq H(X)/2 + 1$.
We have no better lower bound for the chromatic number of a general $d$-dimensional normed space than $\chi_m(X^d)\geq e(X^d)\geq e^{\upOmega(\sqrt{\log d})}$.
The Euclidean minimum-distance graph in Figure~\ref{spindle} has chromatic number $4$, which gives $\chi_m(\bE^2)\geq 4$ (Maehara \cite{Maehara2007}).
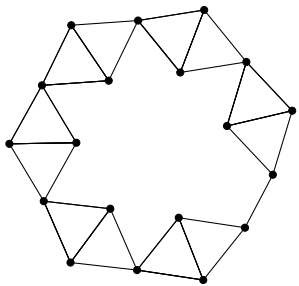
\begin{figure}
\centering
\begin{tikzpicture}[line cap=round,line join=round,scale=0.89]
\clip(0.37192598076702227,6.553322259652919) rectangle (4.809695562001004,10.787835218845641);
\draw(0.49049234362441857,8.687516791086049) -- (1.4903683664357448,8.703262870185439) -- (0.9767938505200114,9.5613078670253) -- cycle;
\draw(0.9767938505200114,9.5613078670253) -- (1.9743416633159308,9.631296160232647) -- (1.4149561170328941,10.460203760999844) -- cycle;
\draw(2.4125039298288113,10.530192054207188) -- (1.9743416633159274,9.631296160232644) -- (1.414956117032892,10.460203760999843) -- cycle;
\draw(0.4904923436244185,8.68751679108605) -- (1.4903683664357448,8.70326287018544) -- (1.0040668595401536,7.8294717942461896) -- cycle;
\draw(2.4125039298288113,10.530192054207188) -- (3.0427631524209464,9.753807283402429) -- (3.4000024757531606,10.687820166539026) -- cycle;
\draw(4.030261698345287,9.911435395734268) -- (3.0427631524209398,9.75380728340243) -- (3.400002475753153,10.68782016653903) -- cycle;
\draw(1.0040668595401536,7.8294717942461896) -- (1.403533675713026,6.912724139339613) -- (1.997727025635496,7.717046377567496) -- cycle;
\draw(2.3971938418083663,6.800298722660919) -- (1.403533675713025,6.912724139339613) -- (1.9977270256354962,7.717046377567497) -- cycle;
\draw(4.030261698345287,9.911435395734268) -- (3.7411063772941744,8.954153201024772) -- (4.714714737028677,9.18237844470981) -- cycle;
\draw(4.425559415977561,8.225096250000318) -- (3.7411063772941735,8.954153201024775) -- (4.714714737028676,9.182378444709814) -- cycle;
\draw(2.3971938418083663,6.800298722660919) -- (3.386061697726927,6.6515023398721285) -- (3.0204892172539735,7.582285215477847) -- cycle;
\draw(4.009357073172531,7.433488832689056) -- (3.386061697726924,6.651502339872128) -- (3.0204892172539703,7.582285215477845) -- cycle;
\draw (0.49049234362441857,8.687516791086049)-- (1.4903683664357448,8.703262870185439);
\draw (0.49049234362441857,8.687516791086049)-- (1.4903683664357448,8.703262870185439);
\draw (1.4903683664357448,8.703262870185439)-- (0.9767938505200114,9.5613078670253);
\draw (0.9767938505200114,9.5613078670253)-- (0.49049234362441857,8.687516791086049);
\draw (0.9767938505200114,9.5613078670253)-- (1.9743416633159308,9.631296160232647);
\draw (0.9767938505200114,9.5613078670253)-- (1.9743416633159308,9.631296160232647);
\draw (1.9743416633159308,9.631296160232647)-- (1.4149561170328941,10.460203760999844);
\draw (1.4149561170328941,10.460203760999844)-- (0.9767938505200114,9.5613078670253);
\draw (2.4125039298288113,10.530192054207188)-- (3.0427631524209464,9.753807283402429);
\draw (2.4125039298288113,10.530192054207188)-- (3.0427631524209464,9.753807283402429);
\draw (3.0427631524209464,9.753807283402429)-- (3.4000024757531606,10.687820166539026);
\draw (3.4000024757531606,10.687820166539026)-- (2.4125039298288113,10.530192054207188);
\draw (1.0040668595401536,7.8294717942461896)-- (1.403533675713026,6.912724139339613);
\draw (1.0040668595401536,7.8294717942461896)-- (1.403533675713026,6.912724139339613);
\draw (1.403533675713026,6.912724139339613)-- (1.997727025635496,7.717046377567496);
\draw (1.997727025635496,7.717046377567496)-- (1.0040668595401536,7.8294717942461896);
\draw (4.030261698345287,9.911435395734268)-- (3.7411063772941744,8.954153201024772);
\draw (4.030261698345287,9.911435395734268)-- (3.7411063772941744,8.954153201024772);
\draw (3.7411063772941744,8.954153201024772)-- (4.714714737028677,9.18237844470981);
\draw (4.714714737028677,9.18237844470981)-- (4.030261698345287,9.911435395734268);
\draw (2.3971938418083663,6.800298722660919)-- (3.386061697726927,6.6515023398721285);
\draw (2.3971938418083663,6.800298722660919)-- (3.386061697726927,6.6515023398721285);
\draw (3.386061697726927,6.6515023398721285)-- (3.0204892172539735,7.582285215477847);
\draw (3.0204892172539735,7.582285215477847)-- (2.3971938418083663,6.800298722660919);
\draw (4.425559415977561,8.225096250000318)-- (4.009357073172531,7.433488832689056);
\draw [fill=black] (0.49049234362441857,8.687516791086049) circle (1.5pt);
\draw [fill=black] (1.4903683664357448,8.703262870185439) circle (1.5pt);
\draw [fill=black] (0.9767938505200114,9.5613078670253) circle (1.5pt);
\draw [fill=black] (1.9743416633159308,9.631296160232647) circle (1.5pt);
\draw [fill=black] (1.4149561170328941,10.460203760999844) circle (1.5pt);
\draw [fill=black] (1.0040668595401536,7.8294717942461896) circle (1.5pt);
\draw [fill=black] (2.4125039298288113,10.530192054207188) circle (1.5pt);
\draw [fill=black] (3.0427631524209464,9.753807283402429) circle (1.5pt);
\draw [fill=black] (3.4000024757531606,10.687820166539026) circle (1.5pt);
\draw [fill=black] (4.030261698345287,9.911435395734268) circle (1.5pt);
\draw [fill=black] (1.403533675713026,6.912724139339613) circle (1.5pt);
\draw [fill=black] (1.997727025635496,7.717046377567496) circle (1.5pt);
\draw [fill=black] (2.3971938418083663,6.800298722660919) circle (1.5pt);
\draw [fill=black] (3.7411063772941744,8.954153201024772) circle (1.5pt);
\draw [fill=black] (4.714714737028677,9.18237844470981) circle (1.5pt);
\draw [fill=black] (4.425559415977561,8.225096250000318) circle (1.5pt);
\draw [fill=black] (3.386061697726927,6.6515023398721285) circle (1.5pt);
\draw [fill=black] (3.0204892172539735,7.582285215477847) circle (1.5pt);
\draw [fill=black] (4.009357073172531,7.433488832689056) circle (1.5pt);
\end{tikzpicture}
\caption{Maehara's minimum-distance graph with chromatic number $4$}\label{spindle}
\end{figure}
Maehara observed that the obvious generalization to higher dimensions gives $\chi_m(\bE^d)\geq d+2$.
Chen \cite{Chen2016} used strongly regular graphs to show that for any $d=q^3-q^2+q$, where $q$ is a prime power, $\chi_m(\bE^{d})\geq q^3+1$.
Chen conjectured that $\chi_m(\bE^d)\geq c^{\sqrt{d}}$ for some constant~$c>1$.

Since any minimum-distance graph for $\ell_\infty^d$ is a subgraph of the minimum-distance graph (in $\ell_\infty^d$) of the lattice $\bZ^d$ (L.~Fejes T\'oth and Sauer \cite{FT-Sauer}; see also Brass \cite{Brass1996}), we obtain $\chi_m(\ell_\infty^d)\leq 2^d$.
Since also $\chi_m(\ell_\infty^d)\geq e(\ell_\infty^d)=2^d$, we obtain the exact value $\chi_m(\ell_\infty^d) = 2^d$.

By Talata's result on the minimum degree of $2$-dimensional spaces mentioned above, we have $\chi_m(X^2)\leq\delta(X^2)+1 = 4$ for any $X^2$ not isometric to $\ell_\infty^2$.
(This also follows from the Four-Colour Theorem, since in this case the minimum-distance graph is planar.)
It is easily seen that Maehara's graph in Figure~\ref{spindle} can be realized in any normed plane.
Since also $\chi_m(\ell_\infty^2)=4$, we obtain $\chi_m(X^2)=4$ for all normed planes.
Consequently, $\alpha_m(n,X^2)\geq n/4$ for all $2$-dimensional $X^2$.
This was observed by Pollack \cite{Pollack1985} for the Euclidean plane.
Csizmadia \cite{Csizmadia1998} improved the Euclidean lower bound to $\alpha_m(n,\bE^2)\geq 9n/35$ and Swanepoel \cite{Swanepoel2002} to $\alpha_m(n,\bE^2)\geq 8n/31$.
Pach and T\'oth \cite{Pach1996} obtained the upper bound $\alpha_m(n,\bE^2)\leq\lceil 5n/16\rceil$.
Swanepoel \cite{Swanepoel2002} also showed the lower bound $\alpha_m(n,X^2)\geq n/(4-\epsi)$, where $\epsi>0$ depends on $X^2$, for each $X^2$ with $\lambda(X^2)\leq 1$.
Most likely this assumption on $X^2$ is unnecessary.
\begin{conjecture}
For each normed plane $X^2$, there exists $\epsi>0$ depending only on $\lambda(X^2)$ such that the independence number of any minimum-distance graph on $n$ points in $X^2$ is at least $\alpha_m(n,X^2)\geq n/(4-\epsi)$.
\end{conjecture}

K.~Bezdek, Nasz\'odi and Visy \cite{BNV2003} introduced a quantity that they call the \define{$k$-th Petty number for packings} $P_m(k,X)$:
this is the largest $n$ such that there exists a minimum-distance graph on $n$ points in $X$ with independence number $<k$.
Thus, $P_m(2,X)=e(X)$, $P_m(k,X)\geq (k-1)e(X)$, and by Ramsey's Theorem, $P_m(k,X) < R(e(X)+1,k)\leq\binom{e(X)+k-1}{k-1}$ \cite[Proposition~1]{BNV2003}.
Also, $P_m(k,X)\leq k(\delta(X)+1)-1$ \cite[Corollary~2]{BNV2003} and $P_m(k,\ell_\infty^d)=(k-1)2^d$ \cite[Theorem~3]{BNV2003}.

\section{Unit-distance graphs and diameter graphs}\label{section:unitdistance}
We consider unit-distance graphs and diameter graphs together, as they have similar extremal behaviour in high dimensions.
Given a finite set $V$ of points from a normed space $X$, we define the \define{unit-distance graph} on $V$ to be the graph with vertex set $V$ and edge set \[E=\setbuilder{ab}{a,b\in V, \norm{a-b}=1}.\]
We also define the \define{diameter graph} on $V$ to be the graph with vertex set $V$ and edge set \[E=\setbuilder{ab}{a,b\in V, \norm{a-b}=\diam(V)}.\]

We again consider a selection of parameters of unit-distance and diameter graphs.
Note that, as in the case of minimum-distance graphs, the maximum clique number of a unit-distance graph or a diameter graph in $X$ equals $e(X)$, the maximum size of an equilateral set.

\subsection{Maximum number of edges of unit-distance and diameter graphs}
Let $U(n,X)$ denote the maximum number of edges in a unit-distance graph on $n$ points in $X$, and let $D(n,X)$ denote the maximum number of edges in a diameter graph on $n$ points in $X$.
It is a difficult problem of Erd\H{o}s \cite{Erdos} to show that $U(n,\bE^2)=\mathrm{O}(n^{1+\epsi})$ for all $\epsi>0$, with the best upper bound known $U(n,\bE^2)=\mathrm{O}(n^{4/3})$ due to Spencer, Szemer\'edi and Trotter~\cite{MR86m:52015}, and the best known lower bound $U(n,\bE^2)=\upOmega(n^{1+c/\log\log n})$ due to Erd\H{o}s \cite{Erdos}.
Erd\H{o}s~\cite{Erdos1985} stated that $U(n,\ell_1^2)=(n^2+n)/4$ for all $n>4$ divisible by $4$.
Brass~\cite{Brass1996} determined $D(n,X^2)$ for all two-dimensional normed spaces $X^2$ and $U(n,X^2)$ whenever $X^2$ is not strictly convex:
\begin{enumerate}
\item $D(n,X^2)=n$ if $\lambda(X^2)=0$,
\item $U(n,X^2)=D(n,X^2)=\lfloor n^2/4\rfloor$ if $0<\lambda(X^2)\leq 1$, 
\item $U(n,X^2)=\lfloor (n^2+n)/4\rfloor$ and $D(n,X^2)=\lfloor n^2/4\rfloor+1$ if $1<\lambda(X^2)<2$, and
\item $U(n,X^2)=\lfloor (n^2+n)/4\rfloor$ and $D(n,X^2)=\lfloor n^2/4\rfloor+2$ if $\lambda(X^2)=2$ (that is, for $X^2$ isometric to $\ell_\infty^2$ and $\ell_1^2$).
\end{enumerate}
Brass observed that the same proofs from geometric graph theory that give the bounds $U(n,\bE^2)=\mathrm{O}(n^{4/3})$ and $D(n,\bE^2)=n$ for the Euclidean norm, still go through for all strictly convex norms.
Valtr~\cite{Valtr} constructed a strictly convex norm and examples of $n$ points with $\upOmega(n^{4/3})$ unit-distance pairs (improving earlier results of Brass \cite{Brass1998}).
This norm has a simple description: $\norm{(x,y)} = \abs{y}+\sqrt{x^2+y^2}$.
Its unit ball is bounded by two parabolic arcs with equations $y=\pm\frac12(1-x^2)$, $-1\leq x\leq 1$.
For this norm, the set \[\setbuilder{\left(\frac{i}{k},\frac{j}{2k^2}\right)}{i,j\in\bN, -k<i\leq k, -k^2<j\leq k^2}\]
of $4k^3$ points has $\upOmega(k^2)$ unit-distance pairs.
The existence of such a piecewise quadratic norm suggests that improving the $\mathrm{O}(n^{4/3})$ bound for the Euclidean norm will depend on subtler number-theoretic properties of the Euclidean norm.
(Another phenomenon pointing to the difficulty is the existence of $n$ points on the $2$-sphere of radius $1/\sqrt{2}$ in $\bE^3$ with $\upOmega(n^{4/3})$ unit-distance pairs \cite{EHP1989}.)

Matou\v{s}ek \cite{Matousek} showed the surprising result that for almost all two-dimensional $X^2$, $U(n,X^2)=\mathrm{O}(n\log n\log\log n)$.
Here, \define{almost all} means that the result holds for all norms except a meager subset of the metric space of all norms, metrized by the Hausdorff distance between their unit balls.
This bound is almost best possible, as for any $2$-dimensional normed space $X^2$, a suitable projection of the vertices and edges of a $k$-dimensional cube onto the plane gives a set of $2^k$ points with $k2^{k-1}$ unit-distance pairs, thus implying $U(n,X^2)=\upOmega(n\log n)$.

In \cite[\S~5.2, Problem~4]{BMP}, Brass, Moser, and Pach asks whether there is a general construction of $n$ points with strictly more than $\upOmega(n\log n)$ unit-distance pairs that can be carried out in all normed spaces of a given dimension $\geq 3$.
It might even be that in each dimension $d\geq 2$, for almost all $d$-dimensional norms, the number of unit-distance pairs is $\mathrm{O}_d(n\log n)$.

The determination of $U(n,\bE^3)$ seems to be as difficult as the planar case, with the best known bounds being $\mathrm{O}(n^{3/2})$ by Kaplan, Matou\v{s}ek, Safernov\'a, and Sharir \cite{KMSS2012} and Zahl \cite{Zahl2013}, and $\upOmega(n^{4/3}\log\log n)$ (Erd\H{o}s~\cite{Erdos2}), although $D(n,\bE^3)=2n-2$ is an old result of Gr\"unbaum, Heppes, and Straszewicz \cite{Grmax, Heppesmax, Strasmax}.

For $d\geq 4$, Erd\H{o}s \cite{Erdos2} determined $U(n,\bE^d)$ and $D(n,\bE^d)$ asymptotically.
By an observation of Lenz \cite{Erdos2}, in Euclidean space of dimension $d\geq 4$, the maximum number of unit-distance pairs in a set of $n$ points is at least $\frac12(1-1/\lfloor d/2\rfloor)n^2 + n - \lfloor d/2\rfloor$.
By an application of the Erd\H{o}s--Stone Theorem and some geometry, Erd\H{o}s found asymptotically matching upper bounds.
In \cite{Erdos1967} he found exact values for even $d\geq 4$ and all sufficiently large $n$ divisible by $2d$, showing that for such $n$, $U(n,\bE^d) = \frac12(1-1/\lfloor d/2\rfloor)n^2 + n$.
Brass \cite{Brass3} determined $U(n,\bE^4)$ for all $n\geq 1$.
Erd\H{o}s and Pach \cite{EP90} showed that $U(n,\bE^d) = \frac12(1-1/\lfloor d/2\rfloor)n^2 + \upTheta(n^{4/3})$ for odd $d\geq 5$.
In \cite{sw-lenz}, $U(n,\bE^d)$ is determined exactly for all even $d\geq 6$ and $D(n,\bE^d)$ for all $d\geq 4$, both for sufficiently large $n$ depending on $d$.
The Lenz construction can be adapted to give the same lower bound $U(n,\ell_p^d)\geq \frac12(1-1/\lfloor d/2\rfloor)n^2+n-\lfloor d/2\rfloor$ for all $p\in[1,\infty]$.
For $p\in(1,\infty)$, this lower bound is most likely the right value asymptotically, but for $p=1$ and $p=\infty$ the Lenz construction can be modified to give a larger lower bound.
To simplify the discussion of analogues of the Lenz construction in general, we introduce the following notion.
We say that a family of $k$ sets $A_1,\dots,A_k\subset X$ is an \define{equilateral family} in $X$ if for any two distinct $i,j\in\{1,\dots,k\}$ and $x\in A_i$, $y\in A_j$, $\norm{x-y}=1$.
Define $a(X)$ to be the largest $k$ such that for all $m\in\bN$, there exists an equilateral family of $k$ sets $A_1,\dots,A_k\subset X$, each of cardinality at least $m$.
Note that $U(n,X)\geq\frac12(1-1/a(X))n^2+\mathrm{O}(1)$.
\begin{proposition}\label{prop:aX}
Let $d\geq 2$.
Then $a(\ell_2^d)=\lfloor d/2\rfloor$, $a(\ell_1^d)\geq d$, $a(\ell_\infty^d) = 2^{d-1}$, and for each $p\in(1,\infty)$, $a(\ell_p^d)\geq\lfloor d/2\rfloor$.
For any $d$-dimensional normed space $X^d$, $a(X^d)\leq 2^d-1$.
\end{proposition}
\begin{proof}
First let $1\leq p<\infty$.
We describe the Lenz construction \cite{Erdos2}.
Let $e_1,\dots,e_d$ be the standard basis of $\bR^d$.
Represent $\bR^d$ as the direct sum of subspaces $V_1,\dots,V_k$, where $k=\lfloor d/2\rfloor$, $V_i=\lin\{e_{2i-1},e_{2i}\}$, $i=1,\dots,k-1$, and $V_k = \lin\{e_{d-1},e_d\}$ if $d$ is even and $V_k = \lin\{e_{d-2},e_{d-1},e_d\}$ if $d$ is odd.
For each $i=1,\dots,k$, let $C_i=V_i\cap \bd (2^{-1/p}B_p^d)$, the $\ell_p$-circle (or sphere if $d$ is odd and $i=k$) in $V_i$ around the origin and with radius $2^{-1/p}$.
Let $A_i$ consist of any $m$ points on $C_i$.
Then it is easy to see that $A_1,\dots,A_k$ form an equilateral family, and we obtain $a(\ell_p^d)\geq \lfloor d/2\rfloor$.

The upper bound $a(\ell_2^d)\leq\lfloor d/2\rfloor$ is well known \cite{Erdos2}.
Suppose $A_1,\dots,A_k\subset\ell_2^d$ form an equilateral family with three points in each set.
Then a simple calculation shows that the affine hulls of the $A_i$ are $2$-dimensional, pairwise orthogonal, and have a point in common.
It follows that $2k\leq d$.

We next consider the case $p=1$.
For $i=1,\dots,k-1$, let $A_{2i-1}$ consist of any $m$ points on a segment on $\bd C_i$ (which is the square $\bd O^2$), and $A_{2i}$ any $m$ points on the opposite segment on $C_i$.
If $d$ is even, do the same for $i=k$ to obtain an equilateral family of $2k=d$ sets, each of size at least $m$.
If $d$ is odd, then it is easy to find three edges of the octahedron $C_k=O^3$, such that the distance between any two points from different edges equals $1$.
(Any three pairwise disjoint edges will work.)
Again we obtain an equilateral family of $2k+1=d$ sets, each of size $m$, which gives $a(\ell_1^d)\geq d$.

Next we consider $\ell_\infty^d$.
As already observed by Gr\"unbaum \cite[p.~421]{Grunbaum-Polytopes} and Makai and Martini~\cite{Makai1991}, if we choose $m$ points on each of $2^{d-1}$ parallel edges of a ball of radius $1/2$ in $\ell_\infty^d$, we obtain $a(\ell_\infty^d)\geq2^{d-1}$.
For the upper bound, let $A_1,\dots,A_k$ be an equilateral family with each $\card{A_i}>3^d$.
Since the diameter of each $A_i$ is at most $2$, we can cover each $A_i$ with a ball of radius $1$.
Each such ball can be tiled with $3^d$ balls of radius $1/3$, and by the pigeon-hole principle, there are two points of $A_i$ inside one of these balls of radius $1/3$.
Therefore, we may replace each $A_i$ by an $A_i'$ consisting of two points at distance $<1$.
It follows that $\bigcup_{i=1}^k A_i'$ has diameter $1$, so is contained in $[0,1]^d$, without loss of generality.
For each $A_i'$, the $d$-cube $[0,1]^d$ has a smallest face $F_i$ that contains $A_i'$.
(Each $F_i$ is the join in the face lattice of $[0,1]^d$ of the unique faces that contain the two elements of $A_i'$ in their relative interiors.)
Any two of these faces are disjoint, otherwise there would be points from different $A_i$ that are at distance $<1$.
Since $\card{A_i'}\geq 2$, the dimension of each $F_i$ is at least $1$.
Therefore, the vertex sets of $F_1,\dots,F_k$ partition the $2^d$ vertices of $[0,1]^d$ into parts of size at least $2$.
It follows that $k\leq 2^{d-1}$.

For a general $X^d$, if $\{A_1,\dots,A_k\}$ is an antipodal family of $k=a(X^d)$ non-empty sets, then for any choice of $p_i\in A_i$, $\setbuilder{p_i}{i=1,\dots,k}$ is an equilateral set, hence $k\leq 2^d$.
If $k=2^d$, then $X^d$ is isometric to $\ell_\infty^d$, and $a(X^d)\leq 2^{d-1}$ as shown above.
Otherwise $a(X^d)<2^d$.
\end{proof}

\begin{theorem}\label{thm:UD0}
The maximum number of edges $U(n,X)$ in a unit-distance graph and the maximum number of edges $D(n,X)$ in a diameter graph on $n$ points in a $d$-dimensional normed space $X$ satisfy the following asymptotics: \[ \lim_{n\to\infty}\frac{U(n,X)}{n^2}=\lim_{n\to\infty}\frac{D(n,X)}{n^2}=\frac12\left(1-\frac{1}{a(X)}\right).\]
\end{theorem}
\begin{proof}
Consider an equilateral family $A_1,\dots,A_k$ in $X$, with each $\card{A_i}$ large.
Since the diameter of each $A_i$ is at most $2$, and we can cover a ball of radius $2$ by a finite number of balls of radius $0.49$, it follows there exist subsets $A_i'\subset A_i$ such that $\diam(A_i')<1$ and $\card{A_i'} \geq c_X\card{A_i}$.
It follows that $a(X)$ is also the largest $k$ such that there exists an equilateral family of $k$ sets such that each set has arbitrarily large cardinality and diameter $<1$.
By choosing $n/k$ points from each $A_i'$, we obtain that $U(n,X)\geq D(n,X)\geq\frac12(1-1/a(X))n^2$ for all $n$.
It follows that \[\liminf_{n\to\infty} \frac{U(n,X)}{n^2} \geq \liminf_{n\to\infty} \frac{D(n,X)}{n^2} \geq \frac12\left(1-\frac{1}{a(X)}\right).\]
Next, it follows from the Erd\H{o}s--Stone Theorem that for all $\epsi>0$ there exists $n_0$ such that $D(n,X)\leq U(n,X)\leq\frac12(1-a(X)^{-1}+\epsi)n^2$ for all $n>n_0$, that is,
\[\limsup_{n\to\infty} \frac{D(n,X)}{n^2} \leq \limsup_{n\to\infty} \frac{U(n,X)}{n^2} \leq \frac12\left(1-\frac{1}{a(X)}\right),\]
and the theorem follows.
\end{proof}
\begin{corollary}\label{cor:aX}
For any $d\geq 1$, the maximum number of edges in a unit-distance graph or a diameter graph on $n$ points in $\ell_\infty^d$ is asymptotically $U(n,\ell_\infty^d) = \frac12(1-2^{1-d})n^2 + \mathrm{o}(n^2)$ and $D(n,\ell_\infty^d) = \frac12(1-2^{1-d})n^2 + \mathrm{o}(n^2)$.
\end{corollary}

Brass conjectured that $\ell_{\infty}^d$ attains the maximum number of unit-distance pairs among $n$ points in a $d$-dimensional normed space, for sufficiently large $n$.
We introduce the following terminology for this maximum number.
Let $U_d(n)$ denote the maximum number of edges in a unit-distance graph of $n$ points in a $d$-dimensional normed space, where the maximum is taken over all norms.
Let $D_d(n)$ denote the analogous quantity for the maximum number of edges in a diameter graph.
\begin{conjecture}[{Brass, Moser, Pach \cite[\S~5.2, Conjecture~6]{BMP}}]\label{conj:brass}
For each $d\in\bN$ there exists $n_0(d)$ such that for all $n\geq n_0(d)$, $U_d(n)=U(n,\ell_\infty^d)$.
\end{conjecture}
We also write $U'_d(n)$ and $D'_d(n)$ for the analogues where we take the maximum only over all strictly convex $d$-dimensional spaces.
The asymptotics as $n\to\infty$ of the four quantities $U_d(n)$, $D_d(n)$, $U'_d(n)$, $D'_d(n)$ can be described in terms of antipodal families.
We say that a family $\setbuilder{A_i}{i=1,\dots,k}$ of subsets of $\bR^d$ is an \define{antipodal family} if for any pair $j,k$ of distinct indices and any $x\in A_j$, $y\in A_k$ there exist two distinct parallel hyperplanes $H_x$ and $H_y$, such that $x\in H_x$, $y\in H_y$, and $\bigcup_{i} A_i$ is in the closed slab bounded by $H_x$ and $H_y$.
We denote by $a(d)$ the largest $k$ such that for each $m\in\bN$ there exists an antipodal family $A_1,\dots,A_k$ in $\bR^d$ with at least $m$ points in each $A_i$.
We also say that the family $A_i$ is a \define{strictly antipodal family} if for any pair $j,k$ of distinct indices and any $x\in A_j$, $y\in A_k$ there exist two distinct parallel hyperplanes $H_x$ and $H_y$, such that $x\in H_x$, $y\in H_y$, and $\bigcup_{i} A_i\setminus\set{x,y}$ is in the open slab bounded by $H_x$ and $H_y$.
We denote by $a'(d)$ the largest $k$ such that for each $m\in\bN$ there exists a strictly antipodal family $A_1,\dots,A_k$ in $\bR^d$ with at least $m$ points in each $A_i$.
Note that $a'(d)\leq A'(d)$, the largest size of a strictly antipodal set (Section~\ref{section:equilateral}).
The following two results can be proved similarly to Theorem~\ref{thm:UD0}, using the following two observations:
An equilateral family $A_1,\dots,A_k$ with each $\diam(A_i)<1$, is an antipodal family, and if the norm is strictly convex, a strictly antipodal family.
Conversely, for any (strictly) antipodal family $A_1,\dots,A_k$ with $k\geq 2$ there exists a (strictly convex) norm that turns the antipodal family into an equilateral family with each $\diam(A_i)\leq 2$.
As in Section~\ref{section:equilateral}, $\conv(A_i)$ can be covered by $\mathrm{O}(3^d d\log d)$ translates of $-\frac12\conv(A_i)$ \cite[Eq.~(6)]{Rogers-Zong}, so by replacing the original $m$ by $m/\mathrm{O}(3^d d\log d)$, we may assume that each $\diam(A_i)\leq 1$.
\begin{theorem}\label{thm:UD}
The maximum number of edges $U_d(n)$ in a unit-distance graph and maximum number of edges $D_d(n)$ in a diameter graph on $n$ points in a $d$-dimensional normed space satisfy the following asymptotics: \[ \lim_{n\to\infty}\frac{U_d(n)}{n^2}=\lim_{n\to\infty}\frac{D_d(n)}{n^2}=\frac12\left(1-\frac{1}{a(d)}\right).\]
\end{theorem}
\begin{theorem}[{Swanepoel and Valtr \cite[Theorem~8]{Swanepoel2010}}]
The maximum number of edges $U'_d(n)$ in a unit-distance graph and maximum number of edges $D'_d(n)$ in a diameter graph on $n$ points in a strictly convex $d$-dimensional normed space satisfy the following asymptotics: \[ \lim_{n\to\infty}\frac{U'_d(n)}{n^2}=\lim_{n\to\infty}\frac{D'_d(n)}{n^2}=\frac12\left(1-\frac{1}{a'(d)}\right).\]
\end{theorem}
In \cite{Swanepoel2010} it is also shown that $W_d(n) = \frac12(1-1/a'(d))n^2 + \mathrm{o}(n^2)$, where $W_d(n)$ is the largest number of pairwise non-parallel unit distance pairs in a set of $n$ points in some strictly convex $d$-dimensional normed space.
Gr\"unbaum \cite[p.~421]{Grunbaum-Polytopes} and Makai and Martini \cite{Makai1991} showed that $2^{d-1}\leq a(d)\leq 2^d-1$ (see also Proposition~\ref{prop:aX}).
Makai and Martini \cite{Makai1991} also showed that $a(2)=2$.
(This also follows from Theorem~\ref{thm:UD0} applied to the determination of $D(n,X^2)$ for all $2$-dimensional $X^2$ of Brass \cite{Brass1996}.)
\begin{conjecture}[{Gr\"unbaum \cite[p.~421]{Grunbaum-Polytopes}, Makai and Martini \cite{Makai1991}}]
\label{conj:GMM}
For all $d\geq 1$, $a(d)\leq 2^{d-1}$.
That is, for each $d\in\bN$ there exists $m$ such that $\min_i\card{A_i}\leq m$ for any antipodal family $\setbuilder{A_i}{i=1,\dots,2^{d-1}+1}$ in $\bR^d$.
\end{conjecture}
Csik\'os et al.\ \cite{Csikos2009} showed that $a(3)\leq 5$.
In the light of Corollary~\ref{cor:aX} and Theorem~\ref{thm:UD}, Conjecture~\ref{conj:brass} would imply Conjecture~\ref{conj:GMM}.
Makai and Martini \cite{Makai1991} showed that $a'(3)\geq 3$, and conjectured that equality holds.
Barvinok, Lee, and Novik \cite{Barvinok2013} showed that $a'(d)\geq \upOmega(3^{d/2})$.
The best upper bound known is the almost trivial $a'(d)\leq 2^d-1$.
Conjecture~\ref{conj:EF} would imply that $a'(d)\leq (2-c)^d$ for some constant $c>0$.

\subsection{Chromatic number}
Denote the maximum chromatic number of all the unit-distance graphs in the normed space $X$ by $\chi_u(X)$ and the maximum chromatic number of all the diameter graphs in $X$ by $\chi_D(X)$.

Recall from Section~\ref{section:Borsuk} that the finite Borsuk number $b_f(X)$ of $X$ is defined to be the smallest $k$ such that any finite subset of $X$ of diameter $1$ can be partitioned into $k$ parts of diameter smaller than $1$.
It is clear that $\chi_u(X)\geq\chi_D(X)=b_f(X)\geq e(X)$.
In particular, as implied by the observations in Section~\ref{section:Borsuk}, the chromatic number of any diameter graph in a $d$-dimensional normed space is at most $(2+\mathrm{o}(1))^d$.
The space $\ell_\infty^d$ is an example where $2^d$ is attained.
Also, since $b(X^2)=b_f(X^2)=e(X^2)$ for all $X^2$, the maximum chromatic number of a diameter graph in $X^2$ is $3$ if $X^2$ is not isometric to $\ell_\infty^2$.

By the De Bruijn--Erd\H{o}s Theorem, $\chi_u(X)$ equals the chromatic number of the infinite unit-distance graph of the whole space $X$.
Clearly, $\chi_u(X)\geq\chi_m(X)$.
We are not aware of any lower bound for $\chi_u(X)$ valid for all $d$-dimensional norms, other than those for $\chi_m(X)$ stated in Section~\ref{section:mindist-chr}.
The chromatic number of the Euclidean plane is a famously difficult problem, with the easy bounds $4\leq\chi_u(\bE^2)\leq 7$ still the best known estimates more than $60$ years after this problem was first formulated by Nelson and Hadwiger \cite{Gardner1960, Hadwiger1961, Soifer2009}.

Chilakamarri \cite{Chilakamarri1991} considered general two-dimensional normed spaces, and showed that the bounds $4\leq\chi_u(X^2)\leq 7$ hold for all $X^2$.
The lower bound follows since the so-called Moser spindle (Fig.~\ref{spindle2}) still occurs as a unit-distance graph for any norm, and the upper bound comes from an appropriate tiling of the plane by a hexagon of sides lengths $1/2$ inscribed in the circle of radius $1/2$.
\begin{figure}
\centering
\begin{tikzpicture}[line cap=round,line join=round,x=1.0cm,y=1.0cm,scale=1.2]
\clip(0.37192598076702243,8.865366335372173) rectangle (2.0403240866889574,10.787835218845663);
\draw(0.4735542917876478,9.839304315986471) -- (1.473430314598974,9.855050395085861) -- (0.9598557986832407,10.713095391925723) -- cycle;
\draw(0.9598557986832407,10.713095391925723) -- (1.0374736269737017,9.716112206144) -- (1.8620774788613839,10.281822810120978) -- cycle;
\draw(1.939695307151843,9.284839624339256) -- (1.0374736269737013,9.716112206144) -- (1.8620774788613832,10.281822810120978) -- cycle;
\draw(0.47355429178764774,9.839304315986473) -- (1.4734303145989738,9.855050395085863) -- (0.9871288077033828,8.981259319146611) -- cycle;
\draw (0.4735542917876478,9.839304315986471)-- (1.473430314598974,9.855050395085861);
\draw (0.4735542917876478,9.839304315986471)-- (1.473430314598974,9.855050395085861);
\draw (1.473430314598974,9.855050395085861)-- (0.9598557986832407,10.713095391925723);
\draw (0.9598557986832407,10.713095391925723)-- (0.4735542917876478,9.839304315986471);
\draw (0.9598557986832407,10.713095391925723)-- (1.0374736269737017,9.716112206144);
\draw (0.9598557986832407,10.713095391925723)-- (1.0374736269737017,9.716112206144);
\draw (1.0374736269737017,9.716112206144)-- (1.8620774788613839,10.281822810120978);
\draw (1.8620774788613839,10.281822810120978)-- (0.9598557986832407,10.713095391925723);
\draw (0.9871288077033828,8.981259319146611)-- (1.939695307151843,9.284839624339256);
\begin{scriptsize}
\draw [fill=black] (0.4735542917876478,9.839304315986471) circle (1.5pt);
\draw [fill=black] (1.473430314598974,9.855050395085861) circle (1.5pt);
\draw [fill=black] (0.9598557986832407,10.713095391925723) circle (1.5pt);
\draw [fill=black] (1.0374736269737017,9.716112206144) circle (1.5pt);
\draw [fill=black] (1.8620774788613839,10.281822810120978) circle (1.5pt);
\draw [fill=black] (0.9871288077033828,8.981259319146611) circle (1.5pt);
\draw [fill=black] (1.939695307151843,9.284839624339256) circle (1.5pt);
\end{scriptsize}
\end{tikzpicture}
\caption{The Moser spindle is a unit-distance graph for any norm in the plane}\label{spindle2}
\end{figure}
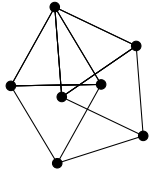
Chilakamarri notes that the chromatic number is exactly $4$ if the unit ball is a parallelogram or a hexagon, and at most $6$ if the unit ball is an octagon.
There is no known example of a normed plane for which the chromatic number is known to be more than $4$, and Brass, Moser, and Pach ask as a problem to find such a plane \cite[\S 5.9, Problem~4]{BMP}.

For Euclidean space, Larman and Rogers \cite{Larman1972} showed the exponential upper bound $\chi_u(\bE^d)\leq (3+\mathrm{o}(1))^d$, which is still the best known.
Frankl and Wilson \cite{Frankl1981} were the first to find a lower bound exponential in $d$: $\chi_u(\bE^d)\geq ((1+\sqrt{2})/2+\mathrm{o}(1))^d$.
The currently best known lower bound of $(1.239\dots+\mathrm{o}(1))^d$ is due to Raigorodskii \cite{Raigorodskii2001}.
There are many specific upper and lower bounds for low-dimensional $\bE^d$; see Raigorodskii's survey \cite{Raigorodskii2013}.

The lower bound of Frankl and Wilson uses $\set{0,1}$-vectors, and so also gives a lower bound for all $\ell_p^d$, or more generally, for any space $X^d$ with a norm that is invariant  under permuting and changing the signs of coordinates: $\chi_u(X^d)\geq ((1+\sqrt{2})/2+\mathrm{o}(1))^d$.
It is easy to see that $\chi_u(\ell_\infty^d)=e(\ell_\infty^d)=2^d$.
For $\ell_1^d$, the best known lower bound is $\chi_u(\ell_1^d)\geq(1.365+\mathrm{o}(1))^d$ due to Raigorodskii \cite{Raigorodskii2004}.
For other papers on $\chi_u(\ell_p^d)$, see Broere \cite{Broere1994} and F\"uredi and Kang \cite{FurediKang2004}.

F\"uredi and Kang \cite{FurediKang2008} showed that $\chi_u(X^d)\leq 5^{d+\mathrm{o}(d)}$ for any $d$-dimensional $X^d$.
This was improved by Kupavskiy \cite{Kupavskiy2011} to $\chi_u(X^d)\leq 4^{d+\mathrm{o}(d)}$.
He also showed $\chi_u(\ell_p^d)\leq 2^{(1+c_p+\mathrm{o}(1))d}$ for all $p>2$, where $0<c_p<1$ and $c_p\to0$ as $p\to \infty$.

\subsection{Independence number and minimum degree}
We define $\delta_u(X)$ to be the maximum over all minimum degrees of unit-distance graphs in $X$, if this maximum exists (otherwise we write $\delta_u(X)=\infty$).
Similarly, let $\delta_D(X)$ be the maximum over all minimum degrees of diameter graphs in $X$, if this maximum exists (otherwise $\delta_m(X)=\infty$).
In contrast to the case of minimum-distance graphs, very little is known about $\delta_u(X)$ or $\delta_D(X)$ in a normed space $X$.
We only make the following general remarks.

If $U(n,X)=\upOmega(n^2)$, or (by Theorem~\ref{thm:UD0}) equivalently, $D(n,X)=\upOmega(n^2)$, then the Erd\H{o}s--Stone Theorem implies that there is no upper bound for $\delta_u(X)$ or $\delta_D(X)$: if $U(n,X)=\frac12(1-1/a(X)+\mathrm{o}(1))n^2$, equivalently, if $D(n,X)=\frac12(1-1/a(X)+\mathrm{o}(1))n^2$, then there are diameter graphs on $n$ vertices with minimum degree $((a-1)/a+\mathrm{o}(1))n$ (and this is sharp).

K.~Bezdek, Nasz\'odi, and Visy \cite{BNV2003} considered the smallest independence numbers $\alpha_u(n,X)$ of a unit-distance graph on $n$ points in $X$.
We can also define $\alpha_D(n,X)$ to be the smallest independence number of a diameter graph on $n$ points in $X$.
Then $\alpha_D(n,X)\geq\alpha_u(n,X)$, and similarly to the case of the minimum-distance graph, $\alpha_D(n,X)\geq n/b_f(X)\geq n/(\delta_D(X)+1)$, $\chi_u(X)\alpha_u(n,X)\geq n$, $\chi_u(X)\leq\delta_u(X)+1$ and $\alpha_u(n,X)\geq n/(\delta_u(X)+1)$. 

They \cite{BNV2003} introduced the \define{$k$-th Petty number} $P(k,X)$ of $X$:
the largest $n$ such that there exists a unit-distance graph on $n$ points in $X$ with independence number $<k$.
This is closely related to their $k$-th Petty number for packings (discussed in Section~\ref{section:mindist-chr} above).
Thus, $P(2,X)=P_m(2,X)=e(X)$, $P_m(k,X)\leq P(k,X)$, hence $(k-1)e(X)\leq P(k,X) < R(e(X)+1,k)\leq\binom{e(X)+k-1}{k-1}$ \cite[Proposition~1]{BNV2003}.
They showed that $P(3,X^d)\leq 2\cdot 3^d$, $P(k,X^d)\leq (k-1)((k-1)3^d-(k-2))$ for all $k\geq 4$, $P(k,X^d)\leq (k-1)4^d$, $P(k,X^2)\leq 8(k-1)$, $P(k,\ell_\infty^d)=(k-1)2^d$, and $P(k,\bE^d)\leq(k-1)3^d$ for all $k\geq 2$; also $P(k,\ell_p^d)\leq(k-1)3^d$ for all $1<p<\infty$, $k\geq 2$ and $d\leq 2^p$.
They ask whether $P(k,X^d)\leq(k-1)2^d$ for all $d$-dimensional $X^d$ and $k\geq 3$ (which would be sharp).

\section{Other graphs}\label{section:othergraphs}
Here we briefly mention three other graphs that are defined for finite sets of points in a finite-dimensional normed space.
\subsection{Minimum spanning trees}\label{section:MST}
For any finite subset $S$ of a normed space $X$, any tree $T$ with vertex set $S$ and minimum total length (with the length of an edge measured in the norm) is called a \define{minimum spanning tree} of $S$.
For any finite subset $S$ of a normed space $X$ with minimum spanning tree $T$ (where distances between points are measured in the norm), 
let $\Delta(T)$ denote the maximum degree of $T$.
Define $\Delta(S)=\max\Delta(T)$ and $\Delta'(S)=\min\Delta(T)$, where the maximum and minimum is taken over all minimum spanning trees $T$ of $S$.
Finally, let $\Delta(X)=\max\Delta(S)$ and $\Delta'(X)=\max\Delta'(S)$, where the maxima are taken over all finite subsets $S$ of the normed space $X$.
Thus, all minimum spanning trees in $X$ have maximum degree at most $\Delta(X)$, and for each finite subset of $X$ there exists a minimum spanning tree with maximum degree at most $\Delta'(X)$.
Cieslik \cite{Cieslik} showed that $\Delta(X)=H(X)$ for all normed spaces $X$.
This was rediscovered by Robins and Salowe \cite{Robins1995}, who also showed that $\Delta'(\ell_p^d)=H'(\ell_p^d)$ for $1\leq p<\infty$.
Martini and Swanepoel \cite{MS2006} generalized the last result to all normed spaces: $\Delta'(X)=H'(X)$ for all finite-dimensional $X$.
The proof needs a general position argument that is made exact by means of the Baire Category Theorem.

\subsection{Steiner minimal trees}
For any finite subset $S$ of a finite-dimensional normed space $X$, any tree $T=(V,E)$ with $S\subseteq V\subset X$ and with each vertex in $V\setminus S$ of degree at least $3$, is called a \define{Steiner tree} of $S$.
The vertices in $V\setminus S$ are called the \emph{Steiner points} of $T$.
A Steiner tree of $S$ of minimum total length is called a \define{Steiner minimal tree} (SMT) of $S$.
(Since there are always at most $\card{S}-2$ Steiner points in a Steiner tree, there will always exist a shortest one by compactness.)
Steiner minimal trees are well studied, especially in the Euclidean plane.
An overview of the extensive literature on them can be found in the monographs of Hwang, Richards and Winter \cite{HRW}, Cieslik \cite{Cieslik2}, Pr\"omel and Steger \cite{PS}, and Brazil and Zachariasen \cite{BZ}.
For their history, see Boltyanski, Martini, and Soltan \cite{BMS} and Brazil, Graham, Thomas, and Zachariasen \cite{BGTZ}.

Denote the maximum degree of a Steiner point in the SMT $T$ by $\Delta_s(T)$ (and set it to $0$ if there are no Steiner points).
Also, denote the maximum degree of a non-Steiner points in $T$ by $\Delta_n(T)$.
Let $\Delta_s(X)=\max\Delta_s(T)$ and $\Delta_n(X)=\max\Delta_n(T)$, where both maxima are taken over all SMTs $T$ in the normed space $X$.
If $T$ is an SMT of $S$, then $T$ is clearly still an SMT of any $S'$ 
such that $S\subseteq S'\subseteq V(T)$.
It follows that $\Delta_s(X)\leq \Delta_n(X)$.

It is well known that $\Delta_s(\bE^d)=\Delta_n(\bE^d)=3$ for all $d\geq 2$ \cite[Section~6.1]{HRW}.
Since a Steiner minimal tree is a minimal spanning tree of its set of vertices, $\Delta_n(X)\leq\Delta(X)=H(X)$ (Cieslik~\cite{Cieslik}).
Since any edge joining two points in an SMT can be replaced by a 
piecewise linear path consisting of segments parallel to the vectors 
pointing to the extreme points of the unit ball, we obtain the following 
well-known lemma, going back to Hanan \cite{Hanan} for $X^2=\ell_1^2$.
\begin{lemma}\label{polyhedral}
If the unit ball of $X^d$ is a polytope with $v$ vertices, then 
$\Delta_s(X^d)\leq\Delta_n(X^d)\leq v$.
\end{lemma}
This gives the upper bounds in $\Delta_s(\ell_1^d)=\Delta_n(\ell_1^d)= 2d$ and $\Delta_s(\ell_\infty^d)=\Delta_n(\ell_\infty^d)= 2^d$.
For the lower bound for $\ell_1^d$, note that the vertex set $S$ of its unit ball $O^d$ is an equilateral set with distance $2$, and that $\setbuilder{O^d+v}{v\in S}$ is a packing in $2O^d$.
It follows that for any Steiner tree of $S$, the total length of edges or parts of edges in $\interior(O^d+v)$ has to be at least $1$ for each $v\in S$, hence the tree that joins each vertex in $S$ to $o$ is a SMT with $o$ a Steiner point of degree $2d$.
The lower bound $\Delta_s(\ell_\infty^d)\geq 2^d$ is shown similarly.

We denote the $d$-dimensional normed space on $\bR^d$ with unit ball $\conv([0,1]^d\cup[-1,0]^d)$ by $H_d$.
The unit ball of $H^2$ is an affine regular hexagon and that of $H^3$ 
an affine rhombic dodecahedron.
Cieslik \cite{Cieslik}, \cite[Conjecture~4.3.6]{Cieslik2} made the 
following conjecture:
\begin{conjecture}[Cieslik \cite{Cieslik, Cieslik2}]\label{conj:Cieslik}
The maximum degree of a vertex in an SMT in a $d$\nobreakdash-\hspace{0pt}dimensional 
normed space $X^d$ satisfies $\Delta_n(X^d)\leq 2^{d+1}-2$, with equality 
if and only if $X^d$ is isometric to the space~$H^d$.
\end{conjecture}
By Lemma~\ref{polyhedral}, $\Delta_n(H^d)\leq 2^{d+1}-2$.
Cieslik \cite{Cieslik4} proved the case $d=2$ of Conjecture~\ref{conj:Cieslik}.
In \cite{Swanepoel2000} the exact values of $\Delta_n(X^2)$ and $\Delta_s(X^2)$ are determined for all $2$-dimensional spaces (see also Martini, Swanepoel and de Wet \cite{MSO2009}).
In particular, up to isometry, $H^2$ is the only $2$-dimensional space that attains $\Delta_n(X^2)=6$, with all others satisfying $\Delta_n(X^2)\leq 4$.
In \cite{Swanepoel2007} it is shown that $\binom{d+1}{\lfloor (d+1)/2\rfloor}\leq\Delta_s(H^d)\leq \Delta_n(H^d)=\binom{d+2}{\lfloor (d+2)/2\rfloor} < 2^{d+1}-2$ for all $d\geq 3$, thus partially disproving the conjecture.
The spaces $H^d$ give the largest known degrees of SMTs in Minkowski spaces 
of dimensions $2$ to $6$, with $\Delta_n(H^2)=6$, $\Delta_n(H^3)=10$, $\Delta_n(H^4)=20$, 
$\Delta_n(H^5)=35$, and $\Delta_n(H^6)=70$, while $\Delta_n(\ell_\infty^d)=2^d$ is larger for $d\geq 7$.
It is not clear whether $H^d$ maximises $\Delta_n(X^d)$ for $2\leq d\leq 6$, and $\ell_\infty^d$ maximises $\Delta_n(X^d)$ for $d\geq 7$.

Morgan \cite[Section~3]{Morgan}, \cite[Chapter~10]{Morganbook} made a related conjecture.
\begin{conjecture}[Morgan \cite{Morgan, Morganbook}]
The maximum degree of a Steiner point in an SMT in any $d$-dimensional 
normed space $X^d$ satisfies $\Delta_s(X^d)\leq 2^d$.
\end{conjecture}
The space $\ell_\infty^d$ shows that this conjecture would be best possible.
The asymptotically best known upper bound for both conjectures is 
$\Delta_s(X^d)\leq \Delta_n(X^d)\leq \mathrm{O}(2^d d^2\log d)$ \cite{Swanepoel2005}.
It is known that $\Delta_s(X^2)\leq 4$ for all $X^2$ \cite{Swanepoel2000}.
There are many two-dimensional spaces attaining $\Delta_s(X^2)=4$, some of them with a unit circle that is piecewise $C^\infty$ \cite{Alfetal}.
They are characterised in \cite{Swanepoel2000}.

The sharp upper bound for differentiable norms is 
$\Delta_s(X^d)\leq \Delta_n(X^d)\leq d+1$ \cite{Alfetal, Lawlor1994, Swanepoel1999c}.
For the $\ell_p$ norm, $1<p<\infty$, we have 
$3\leq \Delta_s(\ell_p^d)\leq \Delta_n(\ell_p^d)\leq 7$ if $p>2, d\geq 2$, and 
$\min\{d,\frac{p}{(p-1)\ln 2}\}\leq \Delta_s(\ell_p^d)\leq\Delta_n(\ell_p^d)\leq 2^{p/(p-1)}$ if $1<p<2$ and $d\geq 3$; see \cite{Swanepoel1999c} for more detailed estimates.

Conger \cite{Conger} showed that 
$\Delta_s(\bR^3, \norm{\cdot}_1+\lambda\norm{\cdot}_2)\geq 6$ for all 
$0<\lambda\leq 1$.
In \cite{Alfetal} it is shown that 
$\Delta_s(\bR^2, \norm{\cdot}_1+\lambda\norm{\cdot}_2)=4$ for all 
$0<\lambda\leq 2+\sqrt{2}$.
The value $\lambda=2+\sqrt{2}$ is sharp, since it follows from the 
results in \cite{Swanepoel2000} that 
$\Delta_s(\bR^2, \norm{\cdot}_1+\lambda\norm{\cdot}_2)=3$ for all 
$\lambda>2+\sqrt{2}$.
In \cite{Swanepoel2007} it was shown that
for the space $X^d=(\bR^d,\norm{\cdot}_1+\lambda\norm{\cdot}_2)$,
$\Delta_s(X^d)=\Delta_n(X^d)=2d$ if $0<\lambda\leq 1$.
Conger made the following conjecture
\cite[Section~3]{Morgan}, \cite[Chapter~10]{Morganbook}.
\begin{conjecture}[Conger]
For any $X^d=(\bR^d,\norm{\cdot})$ such that for some $\epsi>0$, $\norm{\cdot}-\epsi\norm{\cdot}_2$ is still a norm, 
$\Delta_s(X^d)\leq 2d$.
\end{conjecture}

\subsection{Sphere-of-influence graphs}
Toussaint \cite{T88} introduced the sphere-of-influence graph of a finite set of points in Euclidean space for application to pattern analysis and image processing.
See Toussaint \cite{T2014} for a recent survey.
This notion was later generalized to so-called closed sphere-of-influence graphs by Harary et al.\ \cite{HJLM93} and to $k$-th closed sphere-of-influence graphs by Klein and Zachmann \cite{KZ2004}.
Some of their properties have been considered in normed spaces; see \cite{Furedi1994, GPS1994, MQ1994, MQ1999, MQ2003, NPS2016}.

Given $k\in\bN$ and a finite set $S$ in $X$, we define the \define{$k$-th closed sphere-of-influence graph} with vertex set $S$ as follows.
For each $p\in S$, let $r_k(p)$ be the smallest $r$ such that 
$\setbuilder{q\in S}{q\neq p, \norm{p-q}\leq r}$ has at least $k$ elements.
Then join two points $p,q\in S$ whenever the closed balls $p+r_k(p)B_X$ and $q+r_k(q)B_X$ intersect.
Although there is no upper bound on the maximum degree of a $k$-th closed sphere-of-influence graph, Nasz\'odi et al.\ \cite{NPS2016} showed that the minimum degree is bounded above by $k\vartheta(X)$, where $\vartheta(X)$ is the largest size of a set of points in $2B_X$ such that the distance between any two points is at least $1$ and one of the points is $o$.
A simple packing argument gives the upper bound of $\vartheta(X^d)\leq 5^d$, attained by $\ell_\infty^d$.
This then also gives the upper bound of $k\vartheta(X)\card{S}/2$ on the number of edges.

\section{Brass angular measure and applications}\label{section:angle}
\subsection{Angular measures}
Brass \cite{Brass1996} introduced a certain angular measure in any normed plane not isometric to $\ell_\infty^2$, and used it to determine the maximum number of edges in a minimum-distance graph on a set of $n$ points in that plane (see Section~\ref{section:mindistedges}).
Here we demonstrate how some other combinatorial results on translative packings of a planar convex body can be deduced with minimal effort using this measure.

An \define{angular measure} on $X^2$ is a measure $\mu$ on the unit circle $\bd B$ of $X^2$ such that $\mu(\bd B) = 2\pi$, $\mu(A) = \mu(-A)$ for all measurable $A\subseteq\bd B$, and $\mu(\{\vp\}) =  0$ for all $\vp\in \bd B$.
An angular measure $\mu$ is called \define{proper} if $\mu(A)>0$ for any non-trivial arc $A$ of $\bd B$.
We measure an angle in the obvious translation invariant way.
The following is a list of easily proved properties of angular measures.
\begin{lemma}\label{lemma:angular}
Let $\mu$ be an angular measure in any normed plane.
\begin{enumerate}
\item The sum of the measures of the interior angles of a simple closed $n$-gon equals $\pi(n-2)$.
\item Two parallel lines are cut at equal angles by a transversal.
The converse is also true if the measure is proper.
\item Let $\va\vb\vc\vd$ be a simple quadrilateral with $\norm{\va-\vb}=\norm{\vb-\vc}=\norm{\vc-\vd}<\norm{\vd-\va}$.
Then $\mu(\myangle\vb)+\mu(\myangle\vc)\geq\pi$, with strict inequality if the measure is proper.
\end{enumerate}
\end{lemma}
\begin{proof}
Only the last statement needs proof.
We first consider the case where $\va\vb\vc\vd$ is not convex.
If $\va$ or $\vd$ is in the convex hull of the remaining points, say $\va\in\triangle\vb\vc\vd$, let $\ve=\vd+\vb-\vc$.
Then $\vb\ve\vd\vc$ is a parallelogram.
Since $\norm{\vb-\vc}=\norm{\vb-\va}=\norm{\vb-\ve}$, we have $\va\not\in\interior\triangle\vb\vc\ve$.
Then $\va\in\triangle\vc\vd\ve$, and $\norm{\va-\vd}\leq\max(\norm{\vc-\vd},\norm{\vd-\ve})=\norm{c-d}$, a contradiction.
Therefore, $\vb$ or $\vc$ is in the convex hull of the remaining points, say $\vb\in\triangle\va\vc\vd$.
Then clearly $\mu(\myangle\vb)\geq\pi$, and also $\mu(\myangle\vc)>0$ if $\mu$ is proper.

Next we consider the case where $\va\vb\vc\vd$ is convex.
If $\va\vb\parallel\vc\vd$, then $\va\vb\vc\vd$ is a parallelogram and $\norm{\vb-\vc}=\norm{\va-\vd}$, a contradiction.
Therefore, the lines $\Line{\va\vb}$ and $\Line{\vc\vd}$ intersect (Fig.~\ref{fig1}).
Suppose that $\va\vb\cap\vc\vd$ and $\vb$ are on opposite sides of the line $\Line{\va\vd}$.
Then the lines $\Line{\va\vd}$ and $\Line{\vb\vc}$ intersect, otherwise $\va\vd\parallel\vb\vc$ and $\norm{\va-\vd}<\norm{\vb-\vc}$, a contradiction.
Assume without loss of generality that $\va\vd\cap\vb\vc$ and $\va$ are on opposite sides of $\Line{\vc\vd}$ (as in Fig.~\ref{fig1}).
\begin{figure}
\centering
\begin{tikzpicture}[line cap=round,line join=round,scale=0.8]
\clip(0.1,-0.1) rectangle (11,5.04);
\draw [dash pattern=on 5pt off 5pt] (1.935999022064666,2.8614265631685103)-- (6.195999022064666,2.90142656316851);
\draw [dash pattern=on 5pt off 5pt] (6.195999022064666,2.90142656316851)-- (4.9,0.24);
\draw [dash pattern=on 5pt off 5pt] (0.64,0.2)-- (4.173729034333038,1.8795027739808639);
\draw [dash pattern=on 5pt off 5pt] (4.173729034333038,1.8795027739808639)-- (8.433729034333039,1.9195027739808639);
\draw (8.433729034333039,1.9195027739808639)-- (4.9,0.24);
\draw (0.64,0.2)-- (2.88,4.8);
\draw (2.88,4.8)-- (4.9,0.24);
\draw (1.935999022064666,2.8614265631685103)-- (7.846971464859728,0.26767109356675794);
\draw (0.64,0.2)-- (7.846971464859728,0.26767109356675794);
\draw [fill=black] (0.64,0.2) circle (1.5pt);
\draw[color=black] (0.4,0.4) node {$\vb$};
\draw [fill=black] (4.9,0.24) circle (1.5pt);
\draw[color=black] (4.6,0.4) node {$\vc$};
\draw [fill=black] (1.935999022064666,2.8614265631685103) circle (1.5pt);
\draw[color=black] (1.6,2.9) node {$\va$};
\draw [fill=black] (4.173729034333038,1.8795027739808639) circle (1.5pt);
\draw[color=black] (4.4,2.16) node {$\vd$};
\draw [fill=black] (6.195999022064666,2.90142656316851) circle (1.5pt);
\draw[color=black] (7.3,3.2) node {$\ve=\va+\vc-\vb$};
\draw [fill=black] (8.433729034333039,1.9195027739808639) circle (1.5pt);
\draw[color=black] (9.6,2.3) node {$\vp=\vd+\vc-\vb$};
\draw [fill=black] (5.429941342925107,1.3282724006497713) circle (1.5pt);
\draw[color=black] (5.1,1.1) node {$\vf$};
\end{tikzpicture}
\caption{\protect Proof of Lemmas~\ref{lemma:angular} and \ref{brasslemma2}}\label{fig1}
\end{figure}
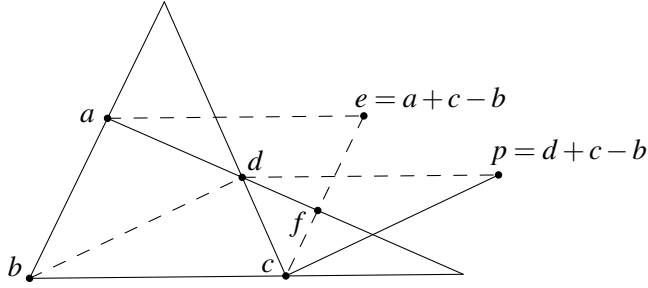
Let $\ve=\va+\vc-\vb$.
Then $\ve\vc\vb\va$ is a parallelogram with $\vd$ in its interior.
Let $\vf=\va\vd\cap\vc\ve$.
Then by the triangle inequality, $\norm{\va-\vd}+\norm{\vd-\vc}\leq\norm{\va-\vf}+\norm{\vf-\vc}\leq\norm{\va-\ve}+\norm{\ve-\vc}=\norm{\vb-\vc}+\norm{\va-\vb}$, a contradiction.

Therefore, lines $\Line{\va\vb}$ and $\Line{\vc\vd}$ intersect in a point on the same side of line $\Line{\va\vd}$ as $\vb$.
Then clearly $\mu(\myangle\vb)+\mu(\myangle\vc)\geq\pi$, with strict inequality if $\mu$ is proper.
\end{proof}

An angular measure is called a \define{Brass measure} if equilateral triangles (in the norm) are equiangular in the measure, that is, $\mu(\myangle\va\vb\vc)=\mu(\myangle\vb\vc\va)=\mu(\myangle\vc\va\vb)=\pi/3$ whenever $\norm{\va-\vb}=\norm{\vb-\vc}=\norm{\vc-\va}>0$.
Clearly, $\ell_\infty^2$ does not have a Brass measure, since in this plane we can find $8$ points $\va_1,\va_2,\dots,\va_8$ on $\bd B$ such that $\triangle\vo\va_i\va_{i+1}$ is equilateral for each $i=1,\dots,8$ (with $a_9=a_1$), and a Brass measure would give $8$ angles of measure $\pi/3$ around the origin.
Remarkably, any normed plane not isometric to $\ell_\infty^2$ has a Brass measure.
\begin{theorem}[Brass \cite{Brass1996}]\label{brass:measure}
A normed plane with unit ball $B$ admits a Brass measure iff $B$ is not a parallelogram.
\end{theorem}
It is not difficult to construct such a measure if the norm is strictly convex, or more generally, if $\lambda(X)\leq 1$ (where $\lambda$ is as defined in Section~\ref{subsection:terminology}), since then for any given point on $\bd B$ there are exactly two points on $\bd B$ at distance $1$ in the norm from the given point.
We sketch the proof of the slightly stronger Theorem~\ref{brass:proper} below.

We call a maximal segment contained in $\bd B$ of length strictly greater than $1$ a \define{long segment}.
Thus, a normed plane $X^2$ has a long segment iff $\lambda(X^2)>1$.
L.~Fejes T\'oth \cite{FejesToth1973} calls the direction of a long segment a \emph{critical direction} of the unit ball.
\begin{lemma}[Brass \cite{Brass1996}]\label{segmentlemma}
Let $X^2$ be a normed plane with unit ball $B$. Then
\begin{enumerate}
\item $\bd B$ contains at most two parallel pairs of long segments, and
\item any long segment on $\bd B$ has length at most $2$, with equality iff $B$ is a parallelogram.
\end{enumerate}
\end{lemma}
We define the \define{ends} of a long segment $ab$ to be the two closed subsegments $aa'$ and $bb'$ of $ab$, where $a'$ and $b'$ are the points on $ab$ such that $\norm{a-b'}=\norm{b-a'}=1$.
The following lemma is easy to prove.
\begin{lemma}\label{lemma:long}
For any Brass measure $\mu$ on a normed plane, any long segment has $\mu$-measure $\pi/3$, and the ends of any long segment have $\mu$-measure $0$.
\end{lemma}
We call a Brass measure \define{good} if all its non-trivial angles of measure $0$ are contained in the ends of long segments.
We note the following straightforward lemma.
\begin{lemma}\label{lemma:good}
All proper Brass measures on a normed plane $X^2$ are good.
All good Brass measures on $X^2$ are proper if $\lambda(X^2)\leq 1$.
\end{lemma}
Brass's proof of Theorem~\ref{brass:measure} actually gives the following strengthening.
\begin{theorem}[Brass \cite{Brass1996}]\label{brass:proper}
Any normed plane $X^2$ for which the unit ball is not a parallelogram, admits a good Brass measure.
\end{theorem}
Before we sketch the proof of Theorem~\ref{brass:proper}, we state the following technical result.
\begin{lemma}\label{brass:lemma}
Suppose that the unit ball $B$ of $X^2$ is not a parallelogram.
Let $S$ be the union of the ends of the long segments of $\bd B$ and the vectors parallel to long segments.
Then for each $x\in\bd B\setminus S$ there exists a unique $y=f(x)\in\bd B\setminus S$ such that $\norm{x-y}=1$ and the orientation of $\myangle xoy$ is positive.
Furthermore, $f$ is a bijection and satisfies $f\circ f\circ f(x)=-x$ for all $x\in\bd B\setminus S$.
\end{lemma}
\begin{proof}
For each $x\in\bd B$ there exists $y\in\bd B$ such that $\norm{x-y}=1$ and the orientation of $\myangle xoy$ is positive.
If $x\notin S$, then $x$ is not parallel to a long segment, and $y=:f(x)$ is unique.
It also follows from $x\notin S$ that $x$ is not on an end of a long segment.
Therefore, different $x\in\bd B\setminus S$ give different $y$.
Thus, $f$ is a strictly monotone function such that $f\circ f(x)=f(x)-x$, hence $-x=f(f(x))-f(x)=f\circ f\circ f(x)$.
\end{proof}

\begin{proof}[Proof sketch of Theorem~\ref{brass:proper}]
Choose a unit vector $x$ not parallel to a long segment and not on an end of a long segment.
(This is possible iff the unit ball $B$ is not a parallelogram.)
Consider the set $S$ and the function $f$ from Lemma~\ref{brass:lemma}.
Let $A$ be the open arc from $x$ to $f(x)$.
Choose any measure $\mu$ on $A\setminus S$ such that $\mu$ and the usual length measure on $A\setminus S$ are mutually absolutely continuous with respect to each other (thus each singleton has measure $0$ and each non-trivial subarc has positive measure) and with total measure $\pi/3$.
Note that $f$ yields not only injections, but also surjections among the six parts of $\bd B\setminus S$, as can be seen by considering $f^{-1}$.
Use the defining property of $f$ to extend this measure to the rest of $\partial B\setminus S$.
Finally, define the measure of $S$ to be $0$.
\end{proof}

We already mentioned the result of Petty \cite{Petty1971} and Soltan \cite{Soltan1975} that a $d$-dimensional space has an equilateral set of size at most $2^d$, with equality iff the unit ball is an affine $d$-cube.
The $2$-dimensional case follows easily from the existence of a Brass measure.
\begin{lemma}\label{lemma:4points}
If the unit ball of a normed plane is not a parallelogram, then there do not exist $4$ equidistant points.
\end{lemma}
\begin{proof}
Suppose that $\{a,b,c,d\}$ is an equilateral set in a normed plane with a Brass measure $\mu$.
Then no $3$ of the points are collinear.

If one of the points, say $d$, is in the convex hull of the other $3$, then on the one hand we would have $\mu(\myangle adb)+\mu(\myangle bdc)+\mu(cda)=2\pi$ from the definition of an angular measure, and on the other hand $\mu(\myangle adb)=\mu(\myangle bdc)=\mu(cda)=\pi/3$, because $\mu$ is a Brass measure.
This is a contradiction.

Otherwise, the $4$ points form a convex quadrilateral $abcd$, say.
Then the interior angle at each vertex equals $\pi/3$, but the sum of the $4$ interior angles has to equal $2\pi$ by Lemma~\ref{lemma:angular}, again a contradiction.
\end{proof}
The following are some useful properties of Brass measures.
\begin{lemma}\label{brasslemma1}
Let $X^2$ be a normed plane with a Brass measure $\mu$.
In $\triangle\vo\va\vb$ let $\norm{\vo-\va}=\norm{\vo-\vb}=1$.
\begin{enumerate}
\item\label{a} If $\norm{\va-\vb}>1$, then $\mu(\myangle\va\vo\vb)\geq\pi/3$.
If $\mu(\myangle\va\vo\vb)=\pi/3$ and $\mu$ is a good Brass measure, then $\va\vb$ is contained in a long segment of $\bd B_X$, with $a$ in one end and $b$ in the other end of the long segment, both different from the inner endpoints of the ends.
\item\label{b} If $\norm{\va-\vb}<1$, then $\mu(\myangle\va\vo\vb)\leq\pi/3$.
If $\mu(\myangle\va\vo\vb)=\pi/3$ and $\mu$ is a good Brass measure, then $\va\vb$ is contained in the relative interior of a long segment of $\bd B_X$, with $a$ in one end and $b$ in the other end of the long segment.
\end{enumerate}
\end{lemma}
The proof of the above lemma is straightforward, using the fact that for $a\in\bd B_X$, the function $x\mapsto\norm{x-a}$ is monotone on any of the two arcs of $\bd B_X$ from $a$ to $-a$.
\begin{lemma}\label{brasslemma2}
Let $X^2$ be a normed plane with a Brass measure $\mu$.
Let $\va\vb\vc\vd$ be a quadrilateral with
\[ \norm{\va-\vb}=\norm{\vb-\vc}=\norm{\vc-\vd}=\norm{\vd-\va}\leq\norm{\vb-\vd},\norm{\vc-\va}. \]
Then $abcd$ is convex and $\mu(\myangle\vb)+\mu(\myangle\vc)=\pi$.
\end{lemma}
\begin{proof}
Suppose that $\norm{\va-\vb}=\norm{\vb-\vc}=\norm{\vc-\vd}=\norm{\vd-\va}=1$.
The quadrilateral $\va\vb\vc\vd$ must be simple, otherwise the triangle inequality would give $\norm{\vb-\vd}=\norm{\va-\vc}=1$, which would contradict Lemma~\ref{lemma:4points}.

Suppose that the simple quadrilateral $\va\vb\vc\vd$ is not convex.
If $\vb\in\interior\triangle\va\vc\vd$, say, then $\norm{\vb-\vd}<\max(\norm{\vd-\va},\norm{\vd-\vc})$, a contradiction.
If $\vb\in\bd\triangle\va\vc\vd$, then $\norm{\vb-\vd}=1$ and $\Segment{\va\vc}$ is a long segment of length $2$ on the unit circle with centre $\vd$.
By Lemma~\ref{segmentlemma}, the unit ball is a parallelogram, which contradicts the existence of~$\mu$.

It follows that $\va\vb\vc\vd$ is convex, with all angles less than $\pi$.
If $\Line{\va\vb}\parallel\Line{\vd\vc}$ or $\Line{\vb\vc}\parallel\Line{\va\vd}$, then the result is obvious.
Assume without loss of generality that $\Line{\va\vb}$ and $\Line{\vc\vd}$ intersect on the side of $\Line{\va\vd}$ opposite $\vb$ and $\vc$, while $bc$ and $da$ intersect on the side of $cd$ opposite $a$ and $b$ (Fig.~\ref{fig1}).
Then, 
letting $e:=a+c-b$, $\ve\va\vb\vc$ is a parallelogram which contains $\vd$ in its interior.
Then the two unit circles with centres $\va$ and $\vc$ both contain $\vb$, $\vd$ and $\ve$ on their boundaries.
It follows that $\vb, \vd, \ve$ are collinear, and $\Segment{\vb\ve}$ is a long segment on both circles, since $\norm{\vb-\vd}\geq 1$.
Let $\vp=\vd+\vc-\vb$.
Then $\triangle cdp$ and $\triangle cep$ are equilateral, hence
$\mu(\myangle\ve\vc\vp)=\mu(\myangle\vd\vc\vp)=\pi/3$ and $\mu(\myangle\vd\vc\ve)=0$.
It follows that $\mu(\myangle\vb)+\mu(\myangle\vc)=\pi-\mu(\myangle\vd\vc\ve)=\pi$.
\end{proof}

\subsection{Applications}
The most striking application of the Brass measure was the original purpose for which Brass introduced it (as mentioned in Subsection~\ref{section:mindistedges}).
The proof, not repeated here, follows Harborth's proof \cite{Harborth1974} for the Euclidean case.
\begin{theorem}[Brass \cite{Brass1996}]
In a normed plane for which the unit ball is not a parallelogram, the number of edges of a minimum distance graph on $n$ points is at most $\lfloor 3n-\sqrt{12n-3}\rfloor$.
\end{theorem}

The following theorem of Talata can be also proved using the Brass measure.
\begin{theorem}[Talata \cite{Talata2002}]\label{thm:talata}
Let $S$ be a non-empty finite set of points in a two-dimensional normed space that is not isometric to $\ell_\infty^2$.
Then the minimum-distance graph on $S$ has a vertex of degree at most $3$.
If $\card{S}\leq 6$ then $S$ has a vertex of degree at most $2$.
\end{theorem}
\begin{proof}
We assume that $S$ is not collinear, otherwise the result is trivial.
Then $\conv(S)$ is a polygon $\vp_1\vp_2\dots\vp_k$, $k\geq 3$.
Denote the internal angle of $\vp_i$ by $\myangle\vp_i$.
Let $\mu$ be any Brass measure.
If $\vp_i$ has degree $d$, then by Lemma~\ref{brasslemma1}, $\mu(\myangle\vp_i)\geq(d-1)\pi/3$.
It follows that if each $\vp_i$ has degree at least $4$, then $\pi(k-2)=\sum_{i=1}^k\mu(\myangle\vp_i)\geq \pi k$, a contradiction.
Therefore, some $\vp_i$ has degree at most $3$.

Next, suppose that each $\vp_i$ has degree at least $3$ and that $\card{S}\leq 6$.
Then $\mu(\myangle\vp_i)\geq2\pi/3$ for all $i$, giving a total angle of $\pi(k-2)\geq 2\pi k/3$, hence $k\geq 6$.
It follows that $S=\set{\vp_1,\vp_2,\dots,\vp_6}$ and each $\vp_i$ has degree exactly $3$.
As mentioned in Section~\ref{section:mindist}, the minimum-distance graph is planar if the space is not isometric to $\ell_\infty^2$.
If $\vp_i$ is joined to $\vp_{i+2}$, then $\vp_{i+1}$ can only be joined to $\vp_i$ and $\vp_{i+2}$ without creating crossing edges.
This contradicts that $\vp_{i+1}$ has degree $3$.
Therefore, no $\vp_i$ is joined to $\vp_{i\pm 2}$.
Hence $\vp_i$ has to be joined to $\vp_{i\pm 1}$ and $\vp_{i+3}$.
However, the diagonals $\vp_i\vp_{i+3}$ and $\vp_{i+1}\vp_{i+4}$ of the hexagon intersect, a contradiction.
\end{proof}
With some more case analysis, the following can also be shown.
\begin{theorem}\label{thm:mindeg}
Let $\delta_n(X^2)$ denote the maximum value of the minimum degree $\delta(G)$ of a minimum-distance graph $G$ of $n$ points in the two-dimensional normed space $X^2$.
If $X^2$ is not isometric to $\ell_\infty^2$, then
\[ \delta_n(X^2) = 
   \left\{\begin{array}{ll} 2, & \text{if }3\leq n\leq 6 \text{ or } n=8, 9,\\
                          3, & \text{if }n=7\text{ or } n\geq 10.
    \end{array}\right.\]
Also,
\[ \delta_n(\ell_\infty^2) = 
   \left\{\begin{array}{ll} 3, & \text{if }4\leq n\leq 11 \text{ or } n=13,14,15,\\
                          4, & \text{if }n=12 \text{ or } n\geq 16.
    \end{array}\right.\]
\end{theorem}
Examples demonstrating the lower bounds in Theorem~\ref{thm:mindeg} can be found on a triangular lattice based on an equilateral triangle when $X^2$ is not isometric to $\ell_\infty^2$, except when $n=11$, where an example is shown in Fig.~\ref{fig:mindeg}.
Examples for $\ell_\infty^2$ can be found on the square lattice $\bZ^2$.
\begin{figure}
\centering
\begin{tikzpicture}[line cap=round,line join=round,scale=1]
\clip(4.902854847103256,4.08883571740275) rectangle (8.324341318130985,6.239968300672657);
\draw(7.972618686693036,4.5605594704481565) -- (8.183862848813968,5.4734418749487315) -- (7.287661414788179,5.199943483496331) -- cycle;
\draw(7.972618686693037,4.560559470448155) -- (7.07641725266725,4.287061078995752) -- (7.287661414788178,5.19994348349633) -- cycle;
\draw(7.498905576909117,6.112825887996911) -- (8.183862848813968,5.473441874948733) -- (7.287661414788178,5.199943483496332) -- cycle;
\draw(7.498905576909117,6.112825887996912) -- (6.6027041428833275,5.839327496544514) -- (7.287661414788177,5.199943483496332) -- cycle;
\draw(6.139417447211106,4.283939372397933) -- (5.922095327750581,5.1953939663685365) -- (5.241413554706269,4.551460193126143) -- cycle;
\draw(5.024091435245737,5.462914787096755) -- (5.922095327750574,5.195393966368544) -- (5.241413554706261,4.551460193126152) -- cycle;
\draw(5.024091435245737,5.462914787096755) -- (5.922095327750581,5.1953939663685365) -- (5.704773208290058,6.106848560339151) -- cycle;
\draw(5.704773208290058,6.106848560339151) -- (5.922095327750581,5.1953939663685365) -- (6.602777100794906,5.839327739610932) -- cycle;
\draw (6.139417447211106,4.283939372397933)-- (7.07641725266725,4.287061078995752);
\draw (6.139417447211106,4.283939372397933)-- (5.922095327750581,5.1953939663685365);
\begin{scriptsize}
\draw [fill=black] (6.139417447211106,4.283939372397933) circle (1.5pt);
\draw [fill=black] (7.07641725266725,4.287061078995752) circle (1.5pt);
\draw [fill=black] (6.139417447211106,4.283939372397933) circle (1.5pt);
\draw [fill=black] (5.922095327750581,5.1953939663685365) circle (1.5pt);
\draw [fill=black] (7.498905576909117,6.112825887996911) circle (1.5pt);
\draw [fill=black] (6.6027041428833275,5.839327496544514) circle (1.5pt);
\draw [fill=black] (5.922095327750581,5.195393966368541) circle (1.5pt);
\draw [fill=black] (5.922095327750579,5.195393966368542) circle (1.5pt);
\draw [fill=black] (5.922095327750581,5.1953939663685444) circle (1.5pt);
\draw [fill=black] (5.70477320829006,6.106848560339145) circle (1.5pt);
\draw [fill=black] (5.704773208290059,6.106848560339154) circle (1.5pt);
\draw [fill=black] (5.704773208290052,6.106848560339148) circle (1.5pt);
\draw [fill=black] (5.922095327750582,5.195393966368542) circle (1.5pt);
\draw [fill=black] (5.024091435245735,5.462914787096755) circle (1.5pt);
\draw [fill=black] (5.92209532775057,5.1953939663685444) circle (1.5pt);
\draw [fill=black] (5.922095327750578,5.1953939663685516) circle (1.5pt);
\draw [fill=black] (5.024091435245736,5.462914787096759) circle (1.5pt);
\draw [fill=black] (5.0240914352457375,5.462914787096758) circle (1.5pt);
\draw [fill=black] (5.922095327750578,5.195393966368541) circle (1.5pt);
\draw [fill=black] (5.9220953277505775,5.195393966368549) circle (1.5pt);
\draw [fill=black] (5.922095327750583,5.195393966368541) circle (1.5pt);
\draw [fill=black] (5.241413554706265,4.551460193126141) circle (1.5pt);
\draw [fill=black] (5.241413554706253,4.551460193126148) circle (1.5pt);
\draw [fill=black] (5.241413554706264,4.551460193126145) circle (1.5pt);
\draw [fill=black] (7.972618686693037,4.560559470448158) circle (1.5pt);
\draw [fill=black] (7.972618686693036,4.560559470448156) circle (1.5pt);
\draw [fill=black] (7.287661414788173,5.199943483496325) circle (1.5pt);
\draw [fill=black] (7.287661414788181,5.199943483496328) circle (1.5pt);
\draw [fill=black] (7.9726186866930355,4.560559470448158) circle (1.5pt);
\draw [fill=black] (7.287661414788175,5.199943483496336) circle (1.5pt);
\draw [fill=black] (8.183862848813968,5.473441874948733) circle (1.5pt);
\draw [fill=black] (8.183862848813968,5.473441874948731) circle (1.5pt);
\draw [fill=black] (8.183862848813975,5.473441874948729) circle (1.5pt);
\draw [fill=black] (7.287661414788178,5.199943483496328) circle (1.5pt);
\draw [fill=black] (7.287661414788181,5.199943483496329) circle (1.5pt);
\draw [fill=black] (7.287661414788179,5.19994348349633) circle (1.5pt);
\end{scriptsize}
\end{tikzpicture}
\caption{A minimum-distance graph on $11$ points with minimum degree $3$}\label{fig:mindeg}
\end{figure}
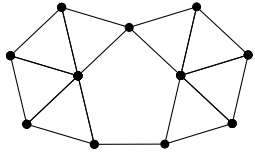

Next we use Brass measures to give simple proofs of results on the various Hadwiger and blocking numbers of convex discs that are not parallelograms.
We also show some related results that would also need elaborate proofs without using Brass measures.
The assertions in the next proposition were discussed in Section~\ref{section:Hadwiger}.
\begin{proposition}\label{hadwprop}
Let $C$ be a convex disc in the plane.
The Hadwiger number $H(C)$, strict Hadwiger number $H'(C)$, and one-sided Hadwiger number $H_{+}(C)$ of $C$ are given in the following table.
\begin{center}
\begin{tabular}{c|ll}
& non-parallelogram & parallelogram  \\ \hline
$H(C)$         & \quad\qquad $6$ \qquad \textup{\cite{Grunbaum1961}}  & \qquad $8$ \qquad \textup{\cite{Hadwiger}} \\
$H'(C)$        & \quad\qquad $5$ \qquad \textup{\cite{Doyle1992}}     & \qquad $4$ \qquad \textup{\cite{Robins1995}}\\
$H_{+}(C)$     & \quad\qquad $4$ & \qquad $5$ 
\end{tabular}
\end{center}
The open one-sided Hadwiger number of $C$ is $H_{+}^{o}(C)=4$ if $\lambda(\frac12(C-C))>1$, and $H_{+}^{o}(C)=3$ otherwise.
\end{proposition}
\begin{proof}
By the observation of Minkowski mentioned in Section~\ref{subsection:Hadwiger}, two translates $\vv+C$ and $\vw+C$ overlap, touch, or are disjoint iff the same holds for the corresponding translates $\vv+\frac12(C-C)$ and $\vw+\frac12(C-C)$ of the central symmetral of $C$, so we may assume without loss of generality that $C$ is $o$-symmetric and is the unit ball of the normed plane $X^2$.
We may then reformulate each of these quantities in terms of points on the unit circle.
For instance, the Hadwiger number is the largest number of points on the unit circle that are pairwise at distance at least $1$.

Note that a convex disc is a parallelogram iff its central symmetral is a parallelogram.
The proofs for $C$ a parallelogram, equivalently, when $X^2$ is isometric to $\ell_\infty^2$, are straightforward, and we only give the proofs for the case when $C$ is not a parallelogram.
Let $\mu$ be a good Brass measure for $X^2$.

If there are $7$ points on $\bd C$ at mutual distances at least $1$, then by Lemma~\ref{brasslemma1}, the sum of the angles spanned at $\vo$ by consecutive points is at least $7\pi/3>2\pi$, a contradiction.
Therefore, $H(C)\leq 6$.
Similarly, $H_+^o(C)\leq H_+(C)\leq 4$.

The existence of $6$ points on $\bd C$ at pairwise distance $\geq 1$ may be established using the well-known continuity argument, or by using a good Brass measure as follows:
Choose a point $\vx_0\in\bd C$ not on a long segment of $C$, nor with $ox_0$ parallel to a long segment of $C$ (there are infinitely many such points).
Then choose $\vx_i\in\bd C$ such that $\mu(\myangle\vx_0\vo\vx_i)=i\pi/3$, for $i=1,\dots,5$.
By Lemma~\ref{brasslemma1}, the distance between any two points is at least $1$, which shows $H(C)\geq 6$.
Similarly, the distance between any two points in $\set{x_0,x_1,x_2,-x_0}$ is at least $1$, which shows that $H_+(C)\geq 4$ and $H_+^o(C)\geq 3$.

Suppose that $H_+^o(C)\geq 4$.
Then there are $4$ points on $\bd C$ in an open half plane bounded by a line through $o$, at pairwise distance at least $1$.
It follows that the Brass measure is not proper, and by Lemma~\ref{brass:proper}, $\bd C$ contains a long segment.
Conversely, if $\bd C$ contains a long segment, then it is easy to find $4$ points on $\bd C$ in an open half plane with pairwise distances at least $1$.

We show that $H'(C)=5$ as follows.
Suppose there are at least $6$ points on $\bd C$ at distance $>1$.
By Lemma~\ref{brasslemma1}, the angle spanned at $\vo$ by consecutive points is $\geq\pi/3$, hence exactly $\pi/3$.
Again by Lemma~\ref{brasslemma1}, the line through any two consecutive points is parallel to a long segment.
Therefore, there are at least three parallel pairs of long segments on the unit circle, which contradicts Lemma~\ref{segmentlemma}, hence $H'(C)\leq 5$.
To find $5$ points on $\bd C$ at distance $>1$, choose any $\vx_1,\dots,\vx_5\in\bd C$ such that $\mu(\myangle\vx_i\vo\vx_{i+1})=2\pi/5>\pi/3$ for all $i=1,2,3,4,5$, and apply Lemma~\ref{brasslemma1}.
\end{proof}

A collection $\setbuilder{\vv_i+C}{i\in I}$ of translates of a convex disc $C$ that all touch $C$ has a natural cyclic ordering determined by the cyclic ordering of the translation vectors $\setbuilder{\vv_i}{i\in I}\subset\bd(\frac12(C-C))$.
We define a \define{dual Hadwiger family} of $C$ to be a collection of translates $C+\vx_i$ of $C$, all touching $C$, and such that any two consecutive (in the natural ordering) translates are not disjoint (i.e.\ they either touch or overlap), and furthermore, $\vo$ is in the convex hull of the translation vectors $\vx_i$.
The last condition is to exclude trivialities.
A \define{dual strict Hadwiger family} of $C$ is a collection of translates of $C$, all touching $C$, and such that any two consecutive translates overlap, and furthermore, $\vo$ is in the convex hull of the translation vectors.
The \define{dual Hadwiger number} $I(C)$ of $C$ is the minimum size of a dual Hadwiger family of $C$.
The \define{dual strict Hadwiger number} $I'(C)$ of $C$ is the minimum size of a dual strict Hadwiger family of $C$.
As before, the dual Hadwiger number and its strict version have equivalent definitions in terms of the norm $\norm{\cdot}$ with unit ball $B=\frac12(C-C)$.
The dual [strict] Hadwiger number equals the smallest number of points on $\bd B$ containing $\vo$ in their convex hull and such that consecutive points are at distance $\leq1$ [$<1$, respectively].
Dual Hadwiger families in the plane were considered by Gr\"unbaum \cite{Grunbaum1961}, where the first part of the next proposition appears without proof.
\begin{proposition}\label{blockingprop}
Let $C$ be a convex disc in the plane.
The dual Hadwiger number $I(C)$ and dual strict Hadwiger number $I'(C)$ are given by the following table.
\begin{center}
\begin{tabular}{c|cc}
& non-parallelogram & parallelogram \\ \hline
$I(C)$ & $6$ & $4$ \\
$I'(C)$ & $7$ & $8$
\end{tabular}
\end{center}
\end{proposition}
\begin{proof}
The parallelogram case is easy to prove and we omit it.
Without loss of generality, $C$ is $o$-symmetric.
Let $\mu$ be a good Brass measure on the plane $X^2$ with unit ball $C$.
Suppose that there exist $5$ points on $\bd C$ with consecutive distances $\leq 1$.
Then, by Lemma~\ref{brasslemma1}, the angles between consecutive vectors are all $\leq\pi/3$, a contradiction.
Therefore, $I(C)\geq 6$.
Equality is shown as before by inscribing a hexagon with sides of unit length to the unit circle.

Suppose there exist $6$ unit vectors with consecutive distances $<1$.
Then, by Lemma~\ref{brasslemma1}, all angles are $\leq\pi/3$, hence $=\pi/3$.
As in the proof that $H'(C)\leq 5$, we obtain at least three parallel pairs of long segments, contradicting Lemma~\ref{segmentlemma}.
This shows that $I'(C)\geq 7$.
To obtain $7$ unit vectors with consecutive distances $<1$, choose $7$ unit vectors with consecutive angles all $<\pi/3$, and apply Lemma~\ref{brasslemma1}.
\end{proof}

Next, we give a simple proof of Zong's result on the blocking number of convex discs.
Zong did not assume that the translates of the convex disc are non-overlapping, nor that they touch $C$, only that they do not overlap $C$.
We prove this stronger result.
\begin{lemma}[Zong \cite{Zong1993}]\label{lemma:zong}
Let $C_1,\dots,C_m$ be translates of a convex disc $C$ in the plane, not overlapping $C$, such that any translate of $C$ that touches $C$ overlaps with some $C_i$.
Then $m\geq 4$.
\end{lemma}
\begin{proof}
We again omit the case where $C$ is a parallelogram.
As before, we may assume without loss of generality that $C$ is $o$-symmetric, not a parallelogram and the unit ball of the normed plane $X^2$.
The statement of the lemma is equivalent to the following.
\begin{quote}
Suppose that there exist points $\vx_1,\dots,\vx_m\in X^2$ such that
\begin{enumerate}
\item $\norm{\vx_i}\geq 1$ for all $i=1,\dots,m$,
\item and for all $\vx\in\bd C$ there exists an $i=1,\dots,m$ such that $\norm{\vx-\vx_i}<1$.
\end{enumerate}
Then $m\geq 4$.
\end{quote}
Consider any $x_1,x_2,x_3\in X^2$ such that $\norm{x_1},\norm{x_2},\norm{x_3}\geq 1$.
To prove the lemma, it is sufficient to find an $x\in\bd C$ such that $\norm{x-x_1},\norm{x-x_2},\norm{x-x_3}\geq 1$.

Let $\mu$ be a good Brass measure on $X^2$.
Let $\un{\vx}_i:=\frac{1}{\norm{\vx_i}}\vx_i$ ($i=1,2,3$).
Then $\un{\vx}_1$, $\un{\vx}_2$, $\un{\vx}_3$ subdivide $\bd C$ into three arcs  $\un{\vx}_1\un{\vx}_2$, $\un{\vx}_2\un{\vx}_3$, $\un{\vx}_3\un{\vx}_1$.

Suppose that one of these arcs has Brass measure $>2\pi/3$, say $\mu(\myangle\vx_2\vo\vx_3)>2\pi/3$.
There exists $x\in\un{\vx}_2\un{\vx}_3$ such that $\mu(\myangle\vx_2\vo\vx)=\mu(\myangle\vx_3\vo\vx)>\pi/3$.
The angle $\myangle x_1ox$ contains either $\myangle x_2ox$ or $\myangle x_3ox$, hence $\mu(x_1ox)>\pi/3$.
By Lemma~\ref{brasslemma1}, $\norm{\vx-\un{\vx}_i}\geq 1$ ($i=1,2,3$).
By the triangle inequality,
\begin{align*}
1 \leq \norm{x-\un{x}_i} &= \norm{\frac{1}{\norm{x_i}}(x-x_i)+(1-\frac{1}{\norm{x_i}})x}\\
&\leq \frac{1}{\norm{x_i}}\norm{x-x_i} + 1 - \frac{1}{\norm{x_i}}.
\end{align*}
It follows that $\norm{\vx-\vx_i}\geq 1$, $i=1,2,3$.

In the remaining case, $\mu(\myangle x_iox_{i+1})=2\pi/3$, $i=1,2,3$ (modulo $3$).
If $\norm{-\un{x}_3-\un{x}_1}\geq 1$ and $\norm{-\un{x}_3-\un{x}_2}\geq 1$, then we have $\norm{-\un{x}_3-\un{x}_i}\geq 1$ for each $i=1,2,3$, and it follows from the triangle inequality as before that $\norm{-\un{x}_3-x_i}\geq 1$ ($i=1,2,3$).
Thus, we may assume without loss of generality that $\norm{-\un{x}_3-\un{x}_1} < 1$.
Since $\mu(\myangle x_1ox_3)=2\pi/3$, $\mu(\myangle x_1o(-x_3))=\pi/3$, and by Lemma~\ref{brasslemma1}, $-\un{x}_3$ and $\un{x}_1$ are in the relative interior of a long segment $S$.

Suppose that $\norm{-\un{x}_2-\un{x}_1} < 1$.
Then, similarly, $-\un{x}_2$ and $\un{x}_1$ are in the relative interior of a long segment.
This long segment is necessarily $S$.
However, since $S$ contains the arc from $-\un{x}_2$ to $-\un{x}_3$, it follows that $\mu(S)\geq 2\pi/3$, which contradicts Lemma~\ref{lemma:long}.
Therefore, $\norm{-\un{x}_2-\un{x}_1} \geq 1$, and similarly, $\norm{-\un{x}_3-\un{x}_2} \geq 1$.
It follows that $\norm{-\un{x}_2-\un{x}_i}\geq 1$ for each $i=1,2,3$, and we are done as before.

We conclude that $m\geq 4$.
\end{proof}

\begin{proposition}[Zong \cite{Zong1993}]
Let $C$ be a convex disc in the plane.
The blocking number $B(C)$ and the strict blocking number $B'(C)$ are given by the following table.
\begin{center}
\begin{tabular}{c|cc}
& non-parallelogram & parallelogram \\ \hline
$B(C)$ & $4$ & $4$ \\
$B'(C)$ & $3$ & $2$
\end{tabular}
\end{center}
\end{proposition}
\begin{proof}
As before, we assume that $C$ is $o$-symmetric and not a parallelogram, and $\mu$ is a good Brass measure in the normed plane $X^2$ with unit ball $C$.
By Lemma~\ref{lemma:zong}, $B(C)\geq 4$.
Next, we show that this is sharp.
Choose any $x_1,x_2\in\bd C$ such that $\mu(\myangle\vx_1\vo\vx_2) = \pi/2$.
Let $x_3=-x_1$ and $x_4=-x_2$.
Then for any $\vx\in\bd C$ there is an $i\in\{1,2,3,4\}$ such that $\mu(\myangle\vx\vo\vx_i)\leq\pi/4<\pi/3$.
Hence, $\norm{\vx-\vx_i}<1$.
Thus, $\{C+\vx_i:i=1,\dots,4\}$ is a maximal Hadwiger family of $C$, and we conclude that $B(C)=4$.

Given any two points $x_1,x_2\in\bd C$ at distance $1$, there exists a point $x_3$ outside the angular domain $\myangle x_1ox_2$ with $\mu(\myangle x_1ox_2)\leq\pi$ such that $\mu(\myangle x_3ox_1)>\pi/3$ and $\mu(\myangle x_3ox_2)>\pi/3$.
As before, $\norm{x_1-x_3},\norm{x_2-x_3}\geq 1$.
Thus, $B'(C)\geq 3$.

To find three points on $\bd C$, we take the vertices $x_0,x_2,x_4$ of the affine regular hexagon $x_0\cdots x_5$ from the proof of $H(C)\geq 6$ in the proof of Proposition~\ref{hadwprop}, inscribed in the unit circle.
Then the three angles $\myangle x_0ox_2$, $\myangle x_2ox_4$, $\myangle x_4ox_0$ each has Brass measure $2\pi/3$.
If we take any $\vx\in\bd C$, it will have an angle of at most $\pi/3$, hence a distance of at most $1$, to one of $x_0,x_2,x_4$.
\end{proof}

We now define the dual blocking number and its strict analogue.
In the definitions of the dual blocking and strict blocking numbers we again make use of the natural ordering of the translates of $C$ that touch $C$.
The \define{dual blocking number} $A(C)$ of a convex disc $C$ is the maximum size of a minimal dual Hadwiger family of $C$.
The \define{strict dual blocking number} $A'(C)$ of $C$ is the maximum size of a minimal strict dual Hadwiger family of $C$.
Note that for the dual notions we do not need the non-triviality requirement that $\vo$ is in the convex hull of the translation vectors, since such trivial collections of translates will not have the maximum size.
\begin{proposition}
Let $C$ be a convex disc in the plane.
The dual blocking number $A(C)$ and the dual strict blocking number $A'(C)$ are given by the following table.
\begin{center}
\begin{tabular}{c|cc}
& non-parallelogram & parallelogram \\ \hline
$A(C)$ & $11$ & $8$ \\
$A'(C)$ & $12$ & $12$
\end{tabular}
\end{center}
\end{proposition}
\begin{proof}
As before, we assume that $C$ is $o$-symmetric and not a parallelogram, and $\mu$ is a good Brass measure in the normed plane $X^2$ with unit ball $C$.
Suppose that we are given a set of $12$ points on $\bd C$ such that consecutive points are at distance $\leq 1$.
If this set is minimal, then the distance between every second vector is $>1$, which gives a strict Hadwiger family of $6$ points, contradicting Proposition~\ref{hadwprop}.
Thus, $A(C)\leq 11$.

To find $11$ points on $\bd C$ such that consecutive points are at distance $\leq 1$ and non-consecutive points at distance $>1$, choose any $11$ points with consecutive angles $2\pi/11$ according to a Brass measure.
Since $2\pi/11<\pi/3$, the distance between consecutive points is $<1$.
Since $2\cdot2\pi/11>\pi/3$, the distance between non-consecutive points is $>1$, and we have obtained a minimal dual Hadwiger family.
We conclude that $A(C)=11$.

Next, suppose that we are given a set of $13$ points on $\bd C$ such that consecutive points are at distance $< 1$.
For some two consecutive angles the sum of the angular measures is $\leq 2\cdot2\pi/13<\pi/3$.
It follows that we may remove the point shared among the two angles and still have all consecutive distances $<1$.
Therefore, the $13$ points do not form a minimal strict dual Hadwiger family, hence $A'(C)\leq 12$.

To find $12$ points on $\bd C$ with consecutive points at distance $< 1$ and non-consecutive points at distance $\geq 1$, choose any $\vx_1$, let $\vx_i$, $i=1,\dots,6$, be the vertices of an inscribed affine regular hexagon with sides of length $1$, let $\vy_1$ be such that $\mu(\myangle\vx_1\vo\vy)=\mu(\myangle\vx_2\vo\vy)=\pi/6$, and let $\vy_i$, $i=1,\dots,6$, be the vertices of an inscribed affine regular hexagon with sides of length $1$.
Then $\{x_i\}\cup\{y_i\}$ is a minimal strict dual Hadwiger family, and it follows that $A'(C)=12$.
\end{proof}

\section{Few-distance sets and thin cones}\label{section:thincone}
Erd\H{o}s \cite{Erdos} asked for the minimum number $g(n)$ of distinct distances that can occur in a set of $n$ points in the plane.
We can equivalently ask for the largest number of points in a given space in which only $k$ non-zero distances occur.
We say that a subset $S$ of a finite-dimensional normed space $X$ is a \define{$k$-distance set} if \[\card{\setbuilder{\norm{x-y}}{x,y\in S, x\neq y}}\leq k.\]
We have encountered $1$-distance sets in Section~\ref{section:equilateral} as equilateral sets.
Let \[f(k,X)=\max\setbuilder{\card{S}}{\text{$S$ is a $k$-distance subset of $X$}}.\]
Thus, $f(1,X)=e(X)$.

For the Euclidean plane, Erd\H{o}s \cite{Erdos} conjectured that $f(k,\bE^2)=\mathrm{O}(k^{1+\epsi})$ and showed that a square piece of the integer lattice gives $f(k,\bE^2)=\upOmega(k\sqrt{\log k})$.
Recently, Guth and Katz \cite{GuthKatz2015} used a striking combination of classical algebraic geometry and topological and combinatorial methods to show that $f(k,\bE^2)=\mathrm{O}(k\log k)$.
In higher dimensions, Erd\H{o}s observed that $c_1k^{d/2} < f(k,\bE^d) < c_2 k^d$.
It is conjectured that $f(k,\bE^d) = \mathrm{O}(k^{d/2+\epsi}))$.
The current best results are due to Solymosi and Vu \cite{Solymosi2008}, which, when combined with the result of Guth and Katz, are $f(k,\bE^3)=\mathrm{O}(k^{5/3+\mathrm{o}(1)})$ and $f(k,\bE^d)=\mathrm{O}(k^{(d^2+d-2)/(2d)+\mathrm{o}(1)})$ for fixed $d$.
Bannai, Bannai and Stanton \cite{BBS} and Blokhuis~\cite{Blokhuis1984} showed that $f(k,\bE^d)\leq\binom{k+d}{k}$, which is a useful bound if $d$ is large compared to $k$.

For general $2$-dimensional spaces $X^2$ we have the bound $f(2,X^2)\leq 9$, with equality iff $X^2$ is isometric to $\ell_\infty^2$ \cite{Swanepoel1999}.
D\"uvelmeyer \cite{Duvelmeyer2008} made a computer-assisted classification of all $2$-distance sets in all $2$-dimensional normed spaces.
This classification is quite involved, but the following general statements can be inferred from his results.
\begin{theorem}[D\"uvelmeyer \cite{Duvelmeyer2008}]
Let $X^2$ be a normed plane.
\begin{enumerate}
\item
If the unit ball of $X^2$ is not a polygon, then $f(2,X^2)\leq 5$.
\item If the unit ball of $X^2$ is not a polygon and $f(2,X^2)=5$, then any $2$-distance set of $5$ points is the vertex set of an affine regular pentagon, the ratio between the two distances is the golden ratio $(1+\sqrt{5})/2$, and the unit ball of $X^2$ has an inscribed affine regular decagon.
\item If $f(2,X^2)\geq 8$, then $X^2$ is isometric to $\ell_\infty^2$ and the $2$-distance set of eight points corresponds to a subset of $\{0,1,2\}^2$ in $\ell_\infty^2$.
\end{enumerate}
\end{theorem}
For general $d$-dimensional normed spaces $X^d$, the following conjecture was made in \cite{Swanepoel1999}.
\begin{conjecture}[\cite{Swanepoel1999}]
For all $k\geq 1$ and $d\geq 1$, for any $d$-dimensional normed space $X^d$ we have $f(k,X^d)\leq(k+1)^d$.
\end{conjecture}
This conjecture is known to hold for $k=1$ and arbitrary $d$ (by Petty and Soltan's result on equilateral sets) and for all $k$ and all $2$-dimensional spaces \cite{Swanepoel1999}.
It is not difficult to show that $f(k,\ell_\infty^d)=(k+1)^d$, with the section $\set{0,1,\dots,k}^d\subset\ell_\infty^d$ of the integer lattice giving the lower bound~\cite{Swanepoel1999}.

We can partition a $k$-distance set in $X^d$ into $b_f(X^d)$ many $(k-1)$-distance sets, hence $f(k,X^d)\leq b_f(X^d)f(k-1,X^d)$.
By induction, we obtain that $f(k,X^d)\leq e(X^d)b_f(X^d)^{k-1}\leq(2+\mathrm{o}_k(1))^{kd}$.
With a different inductive argument that involves the triangle inequality we can show that $f(k,X^d)\leq 2^{kd}$ \cite{Swanepoel1999}.
This is the best general upper bound known for fixed $k$ and large $d$.
Next, we present a bound for fixed $d$ and large $k$.
\begin{theorem}\label{thm:kdist}
For any $d$-dimensional normed space $X^d$, $f(k,X^d)\leq (k+1)^{5^{d+\mathrm{o}(d)}}$.
\end{theorem}
The basic idea of the proof is from \cite{Swanepoel1999} and refined by building on an idea of F\"uredi \cite{Furedi2005}.
In the next section we present a proof of this theorem, after introducing thin cones and their basic properties.
(We note that very recently Polyanskii \cite{Polyanskii2016} showed that $f(k,X^d)\leq k^{\mathrm{O}(d3^d)}$.)

\subsection{Thin cones}
We recall that an \define{ordered vector space} $(V,\leq)$ is a vector space with a partial order compatible with the vector space structure in the following sense:
If $a\leq b$ then $a+x\leq b+x$ and $\lambda a\leq \lambda b$ for all $a,b,x\in V$ and $\lambda\geq 0$.
We also recall that a subset $P$ of the vector space $V$ is a \define{convex cone} if $x+y, \lambda x\in P$ whenever $x,y\in P$ and $\lambda\geq 0$, and that a convex cone $P$ is called \define{proper} if $P\cap(-P)=\set{o}$.
We then have the well-known correspondence between partial orders on $V$ and proper convex cones in $V$:
If $\leq$ is a partial order then its \define{positive cone} $P_{\leq}=\setbuilder{v\in V}{v\geq o}$ is a proper convex cone, and conversely, if $P$ is a proper convex cone $V$, we can define $a\leq_P b$ by $b-a\in P$, and $(V,\leq_P)$ will be an ordered vector space.

Note that we do not assume that the positive cones of our partial orders are closed, and so cannot deduce from $a_n\leq b$ and $\lim_n a_n=a$ that $a\leq b$.
(For example, the cones defined in the proof of Theorem~\ref{thm:planarcone} below are not necessarily closed.)

We now connect the norm with the partial order.
We say that a partial order $\leq$ on a normed space $X$ is \define{monotone} if
$\norm{x+y}>\norm{x}$ for all $x,y\geq o$ with $y\neq o$.
A proper convex cone $P$ in a normed space $(X,\norm{\cdot})$ is called a \define{thin cone} if $\left((P\cap \bd B_X)-(P\cap \bd B_X)\right) \cap P=\{o\}$, or equivalently, if $a-b\notin P$ for any chord $ab$ of the unit sphere inside the cone $P$.
Thin cones were introduced in \cite{Swanepoel1999} and independently in \cite{Furedi2005}.
\begin{lemma}
A proper convex cone $P$ in $X$ is thin iff $\leq_P$ is a monotone partial order.
\end{lemma}
\begin{proof}
Suppose that $\leq_P$ is monotone.
Let $a,b\in P\cap\bd B_X$ such that $a-b\in P$.
Let $x=b$ and $y=a-b$.
If $y\neq o$, then 
$\norm{a}=\norm{x+y}>\norm{x}=\norm{b}$, which contradicts $\norm{a}=\norm{b}=1$.
Therefore, $y=a-b=o$, which shows that the cone $P$ is thin.

Conversely, suppose that $P$ is a thin cone.
Let $x,y\in P$ with $y\neq o$ and suppose that $\norm{x+y}\leq\norm{x}$.
Then $x\neq o$ and $\lambda = \frac{\norm{x+y}}{\norm{x}}\in[0,1]$.
Also, $x+y\neq o$, otherwise $x\in P\cap(-P)=\{o\}$, a contradiction.
Let $a=\frac{1}{\norm{x}}x$ and $b=\frac{1}{\norm{x+y}}(x+y)$.
Since $P$ is a convex cone, $a,b,b-a=\frac{1}{\norm{x+y}}((1-\lambda)x+y)\in P$.
It follows that $b-a\in ((P\cap\bd B_X)-(P\cap\bd B_X))\cap P$, and since $P$ is thin, $b-a=o$.
However, then $(1-\lambda)x=-y\in P\cap -P$, which contradicts that $P$ is a proper cone.
Therefore, $\norm{x+y}>\norm{x}$.
\end{proof}

We call a family $\CP$ of proper convex cones in a vector space $V$ \define{separating} if
\[ \bigcup_{P\in\CP} \bigl(P\cup (-P)\bigr) = V.\]
\begin{lemma}
A family $\CP$ of proper convex cones in the vector space $V$ is separating iff for all $x,y\in V$ there exists $P\in\CP$ such that $x$ and $y$ are comparable in $\leq_P$.
\end{lemma}
We omit the straightforward proof.
We also say that a family $\CO$ of partial orders on $V$ is \define{separating} if $\setbuilder{P_\leq}{\leq\in\CO}$ is a separating family of cones.
We are particularly interested in separating families of thin cones.
The space $\ell_\infty^d$ has a separating family of $d$ thin cones, namely the cones generated by any $d$ pairwise non-opposite facets of the unit ball $B_\infty^d$.
More generally, if the unit ball of $X$ is a polytope with $2f$ facets, then $X$ has a separating family of $f$ thin cones.

\begin{theorem}[\cite{Swanepoel1999}]\label{thm:planarcone}
Any two-dimensional normed space has a separating family of two thin cones.
\end{theorem}
\begin{proof}
Let $a,b\in\bd B$ be chosen such that the area of the triangle $\triangle oab$ is maximized.
Then $B$ is contained in the parallelogram with vertices $\pm a\pm b$.
Let $P_1$ be the cone generated by $\set{a,b}$, and $P_2$ the cone generated by $\set{-a,b}$.
If $\bd B$ does not contain a line segment parallel to $oa$ or $ob$, then $P_1$ and $P_2$ are both thin cones.

If, on the other hand, $\bd B$ contains line segments parallel to $oa$ or $ob$, then we show that $a$ and $b$ can be chosen in such a way that no line segment on $\bd B$ parallel to $oa$ or $ob$ will intersect the interiors of both $P_1$ and $P_2$.
Indeed, if $\bd B$ contains a maximal line segment $cd$ parallel to $oa$, then we can replace $b$ by either endpoint of $cd$ without changing the area of $\triangle aob$.
If $\bd B$ furthermore contains a maximal line segment $ef$ parallel to the new $ob$, then we can similarly replace $a$ by either endpoint of this maximal segment without changing the area of the triangle.
Note that changing $a$ in this way does not create a line segment parallel to the new $oa$ with $b$ in its interior, since then $b$ would be a smooth point of $B$, and we would also have two different lines through $b$ that support $B$, namely the lines parallel to the old $oa$ and the new $oa$.
It follows that no line segment on $\bd B$ parallel to the new $oa$ or the new $ob$ will intersect the interiors of both $P_1$ and $P_2$.

We next modify $P_1$ and $P_2$ so that they become thin.
If $\bd B$ contains a line segment parallel to $oa$ inside $P_i$, then we remove the set $\setbuilder{\lambda a}{\lambda>0}$ from $P_i$.
And if $\bd B$ contains a line segment parallel to $ob$ inside $P_i$, then we remove $\setbuilder{\lambda b}{\lambda>0}$ from $P_i$.
The family $\{P_1,P_2\}$ will stay a separating family, since we never remove the same set from both $P_1$ and~$P_2$.
\end{proof}

Unfortunately, there are $d$-dimensional spaces for which any separating family of thin cones will have size exponential in $d$.
A simple example is the Euclidean space $\bE^d$.
It is easily seen that a proper convex cone $P$ in $\bE^d$ is thin iff $\ipr{x}{y}\geq 0$ for all $x, y\in P$.
The orthants generate a separating family of $2^{d-1}$ thin cones for $\bE^d$.
Heppes \cite{Heppes1998} has shown that if the Euclidean unit sphere in $\bR^3$ is partitioned into parts of angular diameter at most $\pi/2$, then at least $8$ parts are needed.
Therefore, any separating family of thin cones in $\bE^3$ will contain at least $4$ cones.
In higher dimensions, we can make the following simple estimate.
By the isodiametric inequality for the Euclidean sphere, any thin cone in $\bE^d$ intersects the unit sphere in a set of surface measure at most that of a spherical cap of angular diameter $\pi/2$.
Since such a spherical cap is easily seen to be contained in a Euclidean ball of radius $1/\sqrt{2}$, which moreover covers the convex hull of the spherical cap and the centre of the ball, it follows that any separating family of thin cones for $\bE^d$ will contain at least $\frac12(\sqrt{2})^{d}$ cones.

The following result gives a simple sufficient condition for a convex cone to be thin.
\begin{lemma}\label{lemma:sufficient}
A convex cone $P$ in $X$ is thin if $\norm{x-y}<1$ for all $x,y\in P\cap\bd B_X$.
\end{lemma}
\begin{proof}
The hypothesis immediately implies that $P$ is a proper cone.

Let $a,b\in P\cap\bd B_X$ such that $a-b\in P$.
Suppose that $a-b\neq o$.
Then \[\un{a-b}:=\frac{1}{\norm{a-b}}(a-b)\in P\cap \bd B_X.\]
By hypothesis, $\norm{b-\un{a-b}}<1$.
However,
\begin{align*}
\norm{b-\un{a-b}} &= \norm{b-\frac{1}{\norm{a-b}}(a-b)}\\
&=\norm{\left(1+\frac{1}{\norm{a-b}}\right)b - \frac{1}{\norm{a-b}}a}\\
&\geq \norm{\left(1+\frac{1}{\norm{a-b}}\right)b}-\norm{\frac{1}{\norm{a-b}}a}\\
&= 1+\frac{1}{\norm{a-b}}-\frac{1}{\norm{a-b}}=1,
\end{align*}
a contradiction.
Therefore, $a-b=o$, and $P$ is a thin cone.
\end{proof}

Suppose that $S$ is a subset of the unit sphere of a $d$-dimensional normed space $X$ contained in some open ball of radius $1/2$.
Does it follow that for any $p,q\in\conv(S)$, $\norm{\un{p}-\un{q}}<1$?
By the above lemma, a positive answer would imply that the convex cone generated by the intersection of an open ball of radius $1/2$ and the unit sphere is thin.
However, this conclusion is false when $d\geq 3$ under the weaker assumption that $\norm{x-y}<1$ for all $x,y\in S$, as the following example shows.

Let $X$ be the $d$-dimensional subspace $\setbuilder{(\alpha_1,\dots,\alpha_d,\beta)}{\alpha_i,\beta\in\bR,\sum_{i=1}^d\alpha_i=0}$ of $\ell_1^{d+1}$.
Write $e_1,\dots,e_{d+1}$ for the standard basis of $\ell_1^{d+1}$.
Fix $\epsi\in(0,1)$.
For each $i=1,\dots,d$, define
\[ x_i = \frac{d(1-\epsi)}{2(d-1)} e_i - \frac{1-\epsi}{2(d-1)}\sum_{j=1}^d e_j + \epsi e_{d+1}.\]
Then simple calculations show that $S:=\set{x_1,\dots,x_d}$ is an equilateral set of unit vectors in $X$ where the distance between any two is \[\norm{x_i-x_j}_1= (1-\epsi)d/(d-1), \quad 1\leq i < j\leq d.\]
Also, if we let $p=\frac{1}{d-1}\sum_{i=1}^{d-1} x_i$ and $q=x_d$,
then $p,q\in\conv(S)$ and a calculation shows that $\norm{p}_1=\frac{1+(d-2)\epsi}{d-1}$, $\norm{q}_1=1$, and \[\norm{\un{p}-\un{q}}_1 = 2(1-\epsi)=\frac{2(d-1)}{d}\diam(S).\]
Thus, the diameter of $\setbuilder{\un{p}}{p\in\conv(S)}$ is almost double that of $S$.
This example is almost worst possible, at least for diameters up to about $1/2$, as the following theorem shows.
We first estimate the distance to the origin from the convex hull of a set of unit vectors.
This lemma is essentially Lemma~37 in \cite{SwanepoelTAMS}; see also the remark after the proof there.
\begin{lemma}\label{lemma:length}
Let $X$ be a $d$-dimensional normed space and $S\subseteq\bd B_X$ with $\diam(S) < 1+1/d$.
Then $\norm{p} \geq 1 - (1-1/d)\diam(S)$ for all $p\in\conv(S)$.
\end{lemma}
\begin{proof}
By Carath\'eodory's Theorem, it is sufficient to prove the lemma for finite $S$.
Thus, without loss of generality, $S$ is finite and $p$ is an element of $\conv(S)$ of minimum norm.
By Carath\'eodory's Theorem, $p$ is in the convex hull of $k\leq d+1$ points from $S$. 
Write $p=\sum_{i=1}^{k}\lambda_i x_i$ where $\sum_{i=1}^{k}\lambda_i=1$, $\lambda_i > 0$ and $x_i\in S$.
Then for any $i=1,\dots,k$, with $D:=\diam(S)$,
\begin{align*}
1-\norm{p} &= \norm{x_i}-\norm{p} \leq \norm{x_i-p}=\norm{\sum_{j=1}^k\lambda_j(x_i-x_j)}\\
&\leq \sum_{j\neq i}\lambda_j\norm{x_i-x_j} \leq (1-\lambda_i)D.
\end{align*}
In particular, since $D < 1+1/d$, $p\neq o$.
Since $p$ minimizes the norm of all points from $C:=\conv\set{x_1,\dots,x_k}$, it follows that $o\notin C$ and either $p$ is in some facet of $C$ or $C$ lies in a hyperplane of $X$.
Therefore, $p$ is in the convex hull of at most $d$ of these points, and we may suppose that $k\leq d$.
If we sum the inequality $1-\norm{p}\leq (1-\lambda_i)D$ over $i=1,\dots,k$, we obtain $k(1-\norm{p})\leq(k-1)D$, hence $\norm{p}\geq 1-(1-1/k)D\geq 1-(1-1/d)D$.
\end{proof}
\begin{theorem}\label{theorem:hull}
Let $S$ be a subset of the unit sphere of a $d$-dimensional normed space \textup{(}$d\geq 2$\textup{)} and let $0< D \leq d/(2d-1)$ be given such that $\norm{x-y}<D$ for all $x,y\in S$.
Then $\norm{\un{p}-\un{q}}<(2-\frac{1}{d})D$ for any $p,q\in\conv(S)$.
\end{theorem}
\begin{proof}
Without loss of generality, $\norm{p}\leq\norm{q}$.
Also, since $p,q\in\conv(S)$, $\norm{p-q}<D$ and $\norm{q}\leq 1$.
Lemma~\ref{lemma:length} and the given bound on $D$ imply that $\norm{p}\geq 1-(1-1/d)D\geq D$.
The triangle inequality then gives
\begin{align*}
\norm{\un{q}-\un{p}} &= \norm{\frac{1}{\norm{q}}(q-p)-\left(\frac{1}{\norm{p}}-\frac{1}{\norm{q}}\right)p}\\
&\leq \frac{1}{\norm{q}}\norm{q-p} + \left(\frac{1}{\norm{p}}-\frac{1}{\norm{q}}\right)\norm{p}\\
&< \frac{D}{\norm{q}} + 1 - \frac{\norm{p}}{\norm{q}}
\leq \frac{D-\norm{p}}{1} + 1\\
&\leq D-\left(1-\left(1-\frac{1}{d}\right)D\right)+1 = \left(2-\frac{1}{d}\right)D. \qedhere
\end{align*}
\end{proof}
The same proof shows that for $D$ up to $d/(d-1)$ we have $\norm{\un{p}-\un{q}}\leq D/(1-(1-1/d)D)$, which is non-trivial for $D$ up to about $2/3$.
We do not know whether the bound of this theorem still holds for $D$ larger than $1/2$ or if there are better counterexamples than the one described before Lemma~\ref{lemma:length}.
We need the theorem only for the case of $D=1/2$, in the proof of the next corollary.
\begin{corollary}\label{cor:thincone}
Any $d$-dimensional normed space has a separating family of $\mathrm{O}(5^{d +\mathrm{o}(d)})$ thin cones.
\end{corollary}
\begin{proof}
It is well known that the unit sphere $\bd B_X$ can be covered by $\mathrm{O}(5^d d\log d)$ open balls $B_i$ ($i=1,\dots,n$) of radius $1/4$ (see \cite[Eq.~(3)]{Rogers-Zong}).
We may assume that the collection $\set{B_i}$ is minimal, hence $o\notin B_i$.
Therefore, the convex cone $C_i$ generated by $B_i\cap\bd B_X$ is proper, and $\setbuilder{C_i}{i=1,\dots,n}$ is a separating family of cones.

We next show that $C_i$ is a thin cone for each $i=1,\dots,n$.
Let $a,b\in C_i\cap \bd B_X$.
Then $a=\un{p}$ and $b=\un{q}$ for some $p,q\in\conv(B_i\cap\bd B_X)$.
Theorem~\ref{theorem:hull}, applied to $S=B_i\cap\bd B_X$, $D=1/2$ and $p,q$, gives that $\norm{a-b}<1$.
By Lemma~\ref{lemma:sufficient}, $C_i$ is a thin cone.
\end{proof}
Let $P$ be a convex, proper cone and $S$ a finite subset of the vector space $V$.
Then $\leq_P$ restricted to $S$ gives a finite poset with \define{height} $h(S,{\leq_P})$ defined to be the largest cardinality of a chain in $(S,{\leq_P})$.
\begin{theorem}\label{thm:thincone}
Let $X$ be a finite-dimensional normed space with a finite separating family $\CP$ of thin cones.
Let $S$ be a finite subset of $X$.
Then \[\card{S}\leq \prod_{P\in\CP} h(S,{\leq_P}).\]
\end{theorem}
\begin{proof}
For each $x\in S$ and $P\in\CP$, let $h(x;S,{\leq_P})$ denote the largest $h$ such that there exist $x_1,\dots,x_h\in S$ such that $x =x_1 >_P x_2 >_P \dots >_P x_h$.
Then the mapping \[\eta\colon S\to \prod_{P\in\CP} \{1,2\dots,h(S,{\leq P})\}\] defined by $h(x)=\sequencebuilder{h(x,S,{\leq_P})}{P\in\CP}$ is injective.
Indeed, for any distinct $x,y\in X$ there exists $P\in\CP$ such that $y-x\in P\cup(-P)$.
Without loss of generality, $y-x\in P\setminus\{o\}$.
Let $H=h(x;S,{\leq_P})$.
There exist $x_1,\dots,x_H\in S$ such that $x=x_1 >_P x_2 >_P \dots >_P x_H$.
However, then $y > x_1 >_P x_2 >_P \dots >_P x_H$, hence $h(y;S,{\leq_P}) > H=h(x;S,{\leq_P})$ and $\eta(x)\neq \eta(y)$.
\end{proof}

\subsection{Applications}
We can now prove the upper bound on the size of a $k$-distance set.
\begin{proof}[Proof of Theorem~\ref{thm:kdist}]
For any $k$-distance set $S$ in $X^d$ and any thin cone $P$, $h(S,{\leq_P})\leq k+1$.
Now apply Corollary~\ref{cor:thincone} and Theorem~\ref{thm:thincone} to obtain the result.
\end{proof}

As a second application, we obtain an upper bound on the length of a sequence of spheres $p_i + r_i \bd B_X$, $i=1,\dots,n$, such that $p_{i+1},\dots,p_n\in p_i + r_i \bd B_X$ for each $i=1,\dots,n-1$.
\begin{theorem}
Let $\vp_1,\vp_2,\dots,\vp_n\in X^d$ such that $\norm{\vp_i-\vp_j}=\norm{\vp_i-\vp_k}$ whenever $i<j<k$.
Then $\displaystyle n\leq 2^{5^{d+\mathrm{o}(d)}}$.
\end{theorem}
\begin{proof}
For $S=\set{\vp_1,\vp_2,\dots,\vp_n}$ and any thin cone $P$, we have $h(S,{\leq_P})\leq 2$.
Then apply Corollary~\ref{cor:thincone} and Theorem~\ref{thm:thincone}.
\end{proof}
Using a different technique, Nasz\'odi, Pach and Swanepoel \cite{NPS2016b} recently obtained the much better upper bound of $\mathrm{O}(6^d d^2\log^2 d)$ in the above theorem, which was subsequently improved to $\mathrm{O}(3^d d)$ by Polyanskii \cite{Polyanskii2016}.

\section*{Acknowledgements}
We thank Tomasz Kobos, Istv\'an Talata and a very thorough anonymous referee for providing corrections to a previous version.
{\small
\bibliographystyle{amsplain}

\begin{thebibliography}{100}

\bibitem{Alfetal}
M.~Alfaro, M.~Conger, K.~Hodges, A.~Levy, R.~Kochar, L.~Kuklinski, Z.~Mahmood,
  and K.~von Haam, \emph{The structure of singularities in {$\Phi$}-minimizing
  networks in $\mathbf{R}^2$}, Pacific J. Math.\ \textbf{149} (1991), 201--210.
\MR{1105695 (92d:90106)}

\bibitem{Alon1997b}
Noga Alon, \emph{Packings with large minimum kissing numbers}, Discrete Math.\
  \textbf{175} (1997), no.~1-3, 249--251. \MR{1475852 (98f:05040)}

\bibitem{Alon1983}
N. Alon, V.~D. Milman, 
\emph{Embedding of $l_\infty^k$ in finite dimensional Banach spaces},
Israel J. Math.\ \textbf{45} (1983), no.~4, 265--280. 
\MR{0720303 (85f:46027)}

\bibitem{Alon2003}
N.~Alon and P.~Pudl{\'a}k, \emph{Equilateral sets in {$l\sp n\sb p$}}, Geom.\
  Funct.\ Anal.\ \textbf{13} (2003), no.~3, 467--482. \MR{1995795 (2004h:46011)}

\bibitem{AMS2013}
J.~Alonso, H.~Martini, and M.~Spirova, \emph{Discrete geometry in Minkowski spaces}, Discrete geometry and optimization, Fields Inst.\ Commun., \textbf{69}, Springer, New York, 2013. pp.~1--15. 
\MR{3156773}

\bibitem{ABG}
Gergely Ambrus, Imre B\'ar\'any, and Victor Grinberg, \emph{Small subset sums},
  Linear Algebra Appl.\ \textbf{499} (2016), 66--78. \MR{3478885}

\bibitem{Arias-de-Reyna1998}
Juan Arias-de-Reyna, Keith Ball, and Rafael Villa, \emph{Concentration of the
  distance in finite-dimensional normed spaces}, Mathematika \textbf{45}
  (1998), no.~2, 245--252. \MR{1695717 (2000b:46013)}

\bibitem{Bachoc2009}
Christine Bachoc and Frank Vallentin, \emph{Semidefinite programming,
  multivariate orthogonal polynomials, and codes in spherical caps}, European
  J. Combin.\ \textbf{30} (2009), no.~3, 625--637. \MR{2494437 (2010d:90065)}

\bibitem{BBGLW}
Paul Balister, B{\'e}la Bollob{\'a}s, Karen Gunderson, Imre Leader, and Mark
  Walters, \emph{Random geometric graphs and isometries of normed spaces},
  arXiv:1504.05324 (2015).

\bibitem{Ball1991}
Keith Ball,
\emph{Volume ratios and a reverse isoperimetric inequality},
J. London Math.\ Soc.\ (2) \textbf{44} (1991), no.~2, 351--359. 
\MR{1136445 (92j:52013)}

\bibitem{Bandelt1996}
H.-J. Bandelt, V. Chepoi, \emph{Embedding metric spaces in the rectilinear plane: a six-point criterion}, Discete Comput.\ Geom.\ \textbf{15} (1996), 107--117. \MR{1367834 (97a:51022)}

\bibitem{Bandelt1998}
H.-J. Bandelt, V.~Chepoi, and M.~Laurent, \emph{Embedding into rectilinear
  spaces}, Discrete Comput.\ Geom.\ \textbf{19} (1998), no.~4, 595--604.
  \MR{1620076 (99d:51017)}

\bibitem{BBS}
E.~Bannai, E.~Bannai, and D.~Stanton, \emph{An upper bound for the cardinality
  of an $s$-distance subset in real {Euclidean} space {II}}, Combinatorica
  \textbf{3} (1983), 147--152. \MR{0726452 (85e:52013)}

\bibitem{Barany2008}
Imre B{\'a}r{\'a}ny, \emph{On the power of linear dependencies}, Building
  Bridges, Bolyai Soc.\ Math.\ Stud., vol.~19, Springer, Berlin, 2008,
  pp.~31--45. \MR{2484636 (2010b:05003)}

\bibitem{Barvinok2013}
Alexander Barvinok, Seung~Jin Lee, and Isabella Novik, \emph{Explicit
  constructions of centrally symmetric {$k$}-neighborly polytopes and large
  strictly antipodal sets}, Discrete Comput.\ Geom.\ \textbf{49} (2013), no.~3,
  429--443. \MR{3038522}

\bibitem{Bezdek1988}
A.~Bezdek and K.~Bezdek, \emph{A note on the ten-neighbour packings of equal
  balls}, Beitr\"age Algebra Geom.\ (1988), no.~27, 49--53. \MR{984401
  (90a:52025)}

\bibitem{Bezdek2002}
K. Bezdek, \emph{On the maximum number of touching pairs in a finite
  packing of translates of a convex body}, J. Combin.\ Theory Ser.\ A \textbf{98}
  (2002), no.~1, 192--200. \MR{1897933 (2003c:52026)}

\bibitem{Bezdek2006}
K. Bezdek, \emph{Sphere packings revisited}, European J. Combin.\ \textbf{27}
  (2006), no.~6, 864--883. \MR{2226423 (2007a:52021)}

\bibitem{Bezdek2012}
K. Bezdek, \emph{Contact numbers for congruent sphere packings in {E}uclidean
  $3$-space}, Discrete Comput.\ Geom.\ \textbf{48} (2012), no.~2, 298--309.
  \MR{2946449}

\bibitem{BBB2005}
K{\'a}roly Bezdek, Tibor Bisztriczky, and K{\'a}roly B{\"o}r{\"o}czky,
  \emph{Edge-antipodal $3$-polytopes}, Combinatorial and computational geometry,
  Math.\ Sci.\ Res.\ Inst.\ Publ., vol.~52, Cambridge Univ.\ Press, Cambridge, 2005.
  pp.~129--134. \MR{2178317 (2007a:52009)}

\bibitem{Bezdek2003}
K{\'a}roly Bezdek and Peter Brass, \emph{On {$k^+$}-neighbour packings and
  one-sided {H}adwiger configurations}, Beitr\"age Algebra Geom.\ \textbf{44}
  (2003), no.~2, 493--498. \MR{2017050 (2004i:52017)}

\bibitem{BK2016}
K.~Bezdek and M.~A. Khan, \emph{Contact numbers for sphere packings},
  arXiv:1601.00145, 2016.

\bibitem{BK2016b}
K.~Bezdek and M.~A. Khan, \emph{The geometry of homothetic covering and illumination},
  arXiv:1602.06040, 2016.

\bibitem{BNV2003}
K{\'a}roly Bezdek, M{\'a}rton Nasz{\'o}di, and Bal{\'a}zs Visy, \emph{On the
  {$m$}th {P}etty numbers of normed spaces}, Discrete geometry, Monogr.\
  Textbooks Pure Appl.\ Math., vol.~253, Dekker, New York, 2003, pp.~291--304.
  \MR{2034723 (2005a:51004)}

\bibitem{Bezdek2013}
K{\'a}roly Bezdek and Samuel Reid, \emph{Contact graphs of unit sphere packings
  revisited}, J. Geom.\ \textbf{104} (2013), no.~1, 57--83. \MR{3047448}

\bibitem{BB2005}
Tibor Bisztriczky and K{\'a}roly B{\"o}r{\"o}czky, \emph{On antipodal
  3-polytopes}, Rev.\ Roumaine Math.\ Pures Appl.\ \textbf{50} (2005), no.~5-6,
  477--481. \MR{2204128 (2006k:52004)}

\bibitem{Blokhuis1984}
A.~Blokhuis, \emph{Few-distance sets}, {CWI} {Tract} 7, Stichting Mathematisch
  Centrum, Amsterdam, 1984. \MR{0751955 (87f:51023)}

\bibitem{BMS} V. Boltyanski, H. Martini, and V. Soltan, \emph{Geometric methods and optimization problems}, Combinatorial Optimization \textbf{4}, Kluwer, Dordrecht, 1999. \MR{1677397 (2000c:90002)}

\bibitem{Bondarenko2014}
Andriy Bondarenko, \emph{On {B}orsuk's conjecture for two-distance sets},
  Discrete Comput.\ Geom.\ \textbf{51} (2014), no.~3, 509--515. \MR{3201240}

\bibitem{Borsuk}
Karol Borsuk, \emph{Drei S\"atze \"uber die $n$-dimensionale euklidische Sph\"are}, Fund.\ Math.\ \textbf{20} (1933), 177--190.

\bibitem{Boroczky2004}
K{\'a}roly B{\"o}r{\"o}czky, Jr., \emph{Finite packing and covering}, Cambridge
  Tracts in Mathematics, vol.~154, Cambridge University Press, Cambridge, 2004.
  \MR{2078625 (2005g:52045)}

\bibitem{BL1991}
J.~Bourgain and J.~Lindenstrauss, \emph{On covering a set in {$R^N$}
  by balls of the same diameter}, Geometric aspects of functional analysis
  (1989--90), Lecture Notes in Math., vol.~1469, Springer, Berlin, 1991,
  pp.~138--144. \MR{1122618 (92g:52018)}

\bibitem{BDM2015}
P.~Boyvalenkov, S.~Dodunekov, and O.~Musin, \emph{A survey on the kissing numbers},
  Serdica Math.\ J. \textbf{38} (2012), 507--522. \MR{3060792}

\bibitem{Brass1996}
P.~Brass, \emph{{Erd{\H{o}}s} distance problems in normed spaces}, Comput.\
  Geom.\ \textbf{6} (1996), 195--214. \MR{1392310 (97c:52036)}

\bibitem{Brass3}
P.~Brass, \emph{On the maximum number of unit distances among $n$ points in
  dimension four}, Intuitive Geometry (Budapest, 1995), Bolyai Soc.\ Math.\ Studies \textbf{6}, J\'anos Bolyai Math.\
  Soc., Budapest, 1997, pp.~277--290. \MR{1470764 (98j:52030)}

\bibitem{Brass1998}
P.~Brass, \emph{On convex lattice polyhedra and pseudocircle arrangements},
  Charlemagne and his heritage. 1200 years of civilization and science in
  {E}urope, {V}ol. 2 ({A}achen, 1995), Brepols, Turnhout, 1998, pp.~297--302.
  \MR{1672425 (2000a:52031)}

\bibitem{Brass1999}
P.~Brass, \emph{On equilateral simplices in normed spaces}, Beitr{\"a}ge Algebra
  Geom.\ \textbf{40} (1999), 303--307. \MR{1720106 (2000i:52012)}

\bibitem{BMP}
P.~Brass, W.~O.~J. Moser, and J.~Pach, \emph{Research Problems in Discrete Geometry}, Springer--Verlag, New York, 2005. \MR{2163782 (2006i:52001)}

\bibitem{BGTZ} M. Brazil, R. L. Graham, D. A. Thomas, and M. Zachariasen, \emph{On the history of the Euclidean Steiner problem}, Arch.\ Hist.\ Exact. Sci.\ \textbf{68} (2014), 327--354.
\MR{3200931}

\bibitem{BZ} M. Brazil and M. Zachariasen, \emph{Optimal interconnection trees in the plane}, Algorithms and Combinatorics \textbf{29}, Springer, Cham, 2015. \MR{3328741}

\bibitem{Broere1994}
Izak Broere, \emph{Colouring {$\textbf{R}^n$} with respect to different metrics},
  Geombinatorics \textbf{4} (1994), no.~1, 4--9. \MR{1279706 (95g:05044)}


\bibitem{Chen2016}
Hao Chen, \emph{Ball packings with high chromatic numbers from strongly regular graphs},
 arXiv:1502.02070 (2015).

\bibitem{Chilakamarri1991}
Kiran~B. Chilakamarri, \emph{Unit-distance graphs in {M}inkowski metric
  spaces}, Geom.\ Dedicata \textbf{37} (1991), no.~3, 345--356. \MR{1094697
  (92b:05036)}

\bibitem{Cieslik}
D.~Cieslik, \emph{Knotengrade k\"urzester {B\"aume} in endlichdimensionalen
  {Banachr\"aumen}}, Rostock Math.\ Kolloq.\ \textbf{39} (1990), 89--93.
\MR{1090608 (92a:05039)}

\bibitem{Cieslik4}
D.~Cieslik, \emph{The vertex-degrees of {Steiner} minimal trees in 
{Minkowski}
  planes}, Topics in Combinatorics and Graph Theory 
(R.~Bodendiek and R.~Henn,
  eds.), Physica-Verlag, Heidelberg, 1990, pp.~201--206.
\MR{1100038 (91m:05059)}

\bibitem{Cieslik2}
D.~Cieslik, \emph{Steiner {M}inimal {T}rees}, Nonconvex {O}ptimization and its
  {A}pplications, vol.~23, Kluwer, Dordrecht, 1998. \MR{1617288 (99i:05062)}


\bibitem{Conger}
M.~Conger, Energy-Minimizing Networks in $\mathbf{R}^n$, 
Honours Thesis, Williams College, Williams\-town MA, 1989.

\bibitem{Csikos2003}
Bal{\'a}zs Csik{\'o}s, \emph{Edge-antipodal convex polytopes---a proof of
  {T}alata's conjecture}, Discrete geometry, Monogr.\ Textbooks Pure Appl.\
  Math., vol.~253, Dekker, New York, 2003. pp.~201--205. \MR{2034716 (2004m:52026)}

\bibitem{Csikos2009}
Bal{\'a}zs Csik{\'o}s, Gy{\"o}rgy Kiss, Konrad~J. Swanepoel, and P.~Oloff
  de~Wet, \emph{Large antipodal families}, Period.\ Math.\ Hungar.\ \textbf{58}
  (2009), no.~2, 129--138. \MR{2531160 (2010m:52058)}

\bibitem{Csizmadia1998}
G.~Csizmadia, \emph{On the independence number of minimum distance graphs},
  Discrete Comput.\ Geom.\ \textbf{20} (1998), 179--187. \MR{1637884 (99e:05044)}

\bibitem{DLMZ}
L.~Dalla, D.~G. Larman, P.~Mani-Levitska, and C.~Zong, \emph{The blocking
  numbers of convex bodies}, Discrete Comput.\ Geom.\ \textbf{24} (2000),
  no.~2-3, 267--277. \MR{1758049 (2001d:52011)}

\bibitem{Danzer1962}
L.~Danzer and B.~Gr{\"u}nbaum, \emph{\"{U}ber zwei {P}robleme bez\"uglich
  konvexer {K}\"orper von {P}. {E}rd{\H o}s und von {V}. {L}. {K}lee}, Math.\ Z.
  \textbf{79} (1962), 95--99. \MR{0138040 (25 \#1488)}

\bibitem{Dekster2000}
B.~V. Dekster, \emph{Simplexes with prescribed edge lengths in {M}inkowski and
  {B}anach spaces}, Acta Math.\ Hungar.\ \textbf{86} (2000), no.~4, 343--358.
  \MR{1756257 (2001b:52001)}

\bibitem{Doyle1992}
P.~G. Doyle, J.~C. Lagarias, and D.~Randall, \emph{Self-packing of centrally
  symmetric convex bodies in {$\mathbb{R}\sp 2$}}, Discrete Comput.\ Geom.\
  \textbf{8} (1992), 171--189. \MR{1162392 (93e:52038)}


\bibitem{Duvelmeyer2008}
N.~D{\"u}velmeyer, \emph{General embedding problems and two-distance sets in Minkowski planes}, Beitr\"age Algebra Geom.\ \textbf{49} (2008), 549--598. \MR{2468075 (2009j:52007)}

\bibitem{Eggleston1955}
H.~G. Eggleston, \emph{Covering a three-dimensional set with sets of smaller
  diameter}, J. London Math.\ Soc.\ \textbf{30} (1955), 11--24. \MR{0067473 (16,734b)}

\bibitem{Erdos}
P.~Erd{\H{o}}s, \emph{On sets of distances of $n$ points}, Amer.\ Math.\ Monthly
  \textbf{53} (1946), 248--250. \MR{0015796 (7,471c)}

\bibitem{Erdos2}
P.~Erd{\H{o}}s, \emph{On sets of distances of $n$ points in Euclidean space},
Magyar Tud.\ Akad.\ Mat.\ Kutat\'o Int. K\"{o}zl.\ \textbf{5} (1960), 165--169.
\MR{MR0141007 (25 \#4420)}

\bibitem{Erdos1967}
P.~Erd{\H{o}}s, \emph{On some applications of graph theory to geometry}, Canad.\
  J. Math.\ \textbf{19} (1967), 968--971. \MR{0219438 (36 \#2520)}

\bibitem{Erdos1985}
P.~Erd{\H{o}}s, \emph{Problems and results in combinatorial geometry}, Discrete
  geometry and convexity ({N}ew {Y}ork, 1982), Ann.\ New York Acad.\ Sci., vol.~440,
  New York Acad.\ Sci., New York, 1985, pp.~1--11. \MR{809186 (87g:52001)}

\bibitem{ErdosFuredi1983}
P.~Erd{\H{o}}s and Z.~F{\"u}redi, \emph{The greatest angle among {$n$} points
  in the {$d$}-dimensional {E}uclidean space}, Combinatorial mathematics
  ({M}arseille-{L}uminy, 1981), North-Holland Math.\ Stud., vol.~75,
  North-Holland, Amsterdam, 1983, pp.~275--283. \MR{841305 (87g:52018)}

\bibitem{EHP1989}
Paul Erd\H{o}s, Dean Hickerson, and J\'anos Pach,
\emph{A problem of Leo Moser about repeated distances on the sphere}, 
Amer.\ Math.\ Monthly \textbf{96} (1989), no.~7, 569--575. 
\MR{1008787 (90h:52008)}

\bibitem{EP90}
P.~Erd\H{o}s and J.~Pach, \emph{Variations on the theme of repeated distances}, Combinatorica \textbf{10} (1990), no.~3, 261--269. \MR{1092543 (92b:52037)}

\bibitem{GFT1981}
G.~Fejes~T{\'o}th, \emph{Ten-neighbour packing of equal balls}, Period.\ Math.\
  Hungar.\ \textbf{12} (1981), no.~2, 125--127. \MR{603405 (82e:52013)}

\bibitem{FTKu1993}
G{\'a}bor Fejes~T{\'o}th and W{\l}odzimierz Kuperberg, \emph{A survey of recent
  results in the theory of packing and covering}, New trends in discrete and
  computational geometry, ed.~J\'anos Pach, Algorithms and Combinatorics \textbf{10}, Springer, Berlin, 1993,
  pp.~251--279. \MR{1228046 (94h:52037)}

\bibitem{Fejes-Toth-book}
L{\'a}szl{\'o} Fejes~T{\'o}th, \emph{Lagerungen in der {E}bene, auf der {K}ugel
  und im {R}aum}, Zweite verbesserte und erweiterte Auflage,
  Die Grundlehren der mathematischen Wissenschaften,
  Band 65, Springer-Verlag, Berlin-New York, 1972. \MR{0353117 (50 \#5603)}

\bibitem{FejesToth1973}
L. Fejes T\'oth, \emph{Five-neighbour packing of convex discs}, Period.\ Math.\ Hungar.\ \textbf{4} (1973), 221--229. \MR{0345006 (49 \#9745)}

\bibitem{FT1975}
L.~Fejes~T{\'o}th, \emph{On {H}adwiger numbers and {N}ewton numbers of a convex
  body}, Studia Sci.\ Math.\ Hungar.\ \textbf{10} (1975), no.~1--2, 111--115.
  \MR{0440469 (55 \#13344)}

\bibitem{FT-Sauer}
L.~Fejes T\'oth and N.~Sauer,
\emph{Thinnest packing of cubes with a given number of neighbours},
Canad.\ Math.\ Bull.\ 20 (1977), no.~4, 501--507. 
\MR{0478017 (57 \#17513)}

\bibitem{Frankl1981}
P.~Frankl and R.~M. Wilson, \emph{Intersection theorems with geometric
  consequences}, Combinatorica \textbf{1} (1981), no.~4, 357--368. \MR{0647986 (84g:05085)}

\bibitem{FOSS2014}
D.~Freeman, E.~Odell, B.~Sari, and Th. Schlumprecht, \emph{Equilateral sets in
  uniformly smooth {B}anach spaces}, Mathematika \textbf{60} (2014), no.~1,
  219--231. \MR{3164528}

\bibitem{Fullerton1949}
R.~E. Fullerton, \emph{Integral distances in {B}anach spaces}, Bull.\ Amer.\
  Math.\ Soc.\ \textbf{55} (1949), 901--905. \MR{0032934 (11,369c)}

\bibitem{Furedi2005}
Z.~F\"uredi, \emph{Few-distance sets in $d$-dimensional normed spaces},
  Oberwolfach Rep. \textbf{2} (2005), no.~2, 947--950, Abstracts from the
  Discrete Geometry workshop held April 10--16, 2005, Organized by Martin Henk,
  Ji{\v{r}}{\'{\i}} Matou{\v{s}}ek and Emo Welzl, Oberwolfach Reports. Vol. 2,
  no. 2. \MR{2216216}

\bibitem{FurediKang2004}
Zolt{\'a}n F{\"u}redi and Jeong-Hyun Kang, \emph{Distance graph on
  {$\mathbb{Z}^n$} with {$l_1$} norm}, Theoret.\ Comput.\ Sci.\ \textbf{319}
  (2004), no.~1-3, 357--366. \MR{2074960 (2005c:05079)}

\bibitem{FurediKang2008}
Z.~F{\"u}redi and J.-H. Kang, \emph{Covering the {$n$}-space by convex bodies
  and its chromatic number}, Discrete Math.\ \textbf{308} (2008), no.~19,
  4495--4500. \MR{2433777 (2009c:52031)}

\bibitem{Furedi1994}
Zolt{\'a}n F{\"u}redi and Peter~A. Loeb, \emph{On the best constant for the
  {B}esicovitch covering theorem}, Proc.\ Amer.\ Math.\ Soc.\ \textbf{121} (1994),
  no.~4, 1063--1073. \MR{1249875 (95b:28003)}

\bibitem{Gardner1960}
Martin Gardner, \emph{Mathematical {Games}}, Scientific American \textbf{203}
  (1960), no.~4, 172--180.

\bibitem{Geher}
Gy{\"o}rgy~P{\'a}l Geh{\'e}r, \emph{A contribution to the Aleksandrov
  conservative distance problem in two dimensions}, Linear Algebra Appl.\
  \textbf{481} (2015), 280--287.
\MR{3349657}

\bibitem{Glakousakis}
E.~Glakousakis and S.~Mercourakis, \emph{Examples of infinite dimensional
  {B}anach spaces without infinite equilateral sets}, Serdica Math.\ J. 
  \textbf{42} (2016), no.~1, 65--88. \MR{3523955}

\bibitem{Groemer}
H.~Groemer, \emph{Absch\"atzungen f\"ur die {Anzahl} der konvexen {K\"orper},
  die einen konvexen {K\"orper} ber\"uhren}, Monatsh.\ Math.\ \textbf{65} (1961),
  74--81. \MR{0124819 (23 \#A2129)}

\bibitem{Grmax}
B.~Gr{\"u}nbaum, \emph{A proof of {V}{\'a}zsonyi's conjecture}, Bull.\ Res.\
  Council Israel.\ Sect.~A \textbf{6} (1956), 77--78. \MR{0087115 (19,304d)}

\bibitem{MR21:2209}
B.~Gr{\"u}nbaum, \emph{Borsuk's partition conjecture in {M}inkowski planes},
  Bull.\ Res.\ Council Israel.\ Sect.~F \textbf{7F} (1957/1958), 25--30.
  \MR{0103440 (21 \#2209)}

\bibitem{Grunbaum1961}
B.~Gr{\"u}nbaum, \emph{On a conjecture of {H.} {Hadwiger}}, Pacific J. Math.\
  \textbf{11} (1961), 215--219. \MR{0138044 (25 \#1492)}

\bibitem{Gr}
B.~Gr{\"u}nbaum, \emph{Strictly antipodal sets}, Israel J. Math.\ \textbf{1} (1963),
  5--10. \MR{0159263 (28 \#2480)}

\bibitem{Grunbaum-Polytopes}
Branko Gr\"unbaum,
\emph{Convex polytopes}, Second ed. Graduate Texts in Mathematics, 221,
Springer-Verlag, New York, 2003.
\MR{1976856 (2004b:52001)}

\bibitem{GPS1994}
L.~Guibas, J.~Pach, and M.~Sharir, \emph{Sphere-of-influence graphs in higher
  dimensions}, Intuitive geometry ({S}zeged, 1991), Colloq.\ Math.\ Soc.\ J\'anos
  Bolyai, vol.~63, North-Holland, Amsterdam, 1994, pp.~131--137. \MR{1383618 (97a:05183)}

\bibitem{GuthKatz2015}
Larry Guth and Nets~Hawk Katz, \emph{On the {E}rd{\H o}s distinct distances
  problem in the plane}, Ann.\ of Math.\ (2) \textbf{181} (2015), no.~1,
  155--190. \MR{3272924}

\bibitem{Hadwiger}
H.~Hadwiger, \emph{{\"U}ber {Treffanzahlen} bei translationsgleichen
  {Eik\"orpern}}, Arch.\ Math.\ \textbf{8} (1957), 212--213. \MR{0091490 (19,977e)}

\bibitem{Hadwiger1961}
H.~Hadwiger, \emph{Ungel{\"o}ste {Probleme} {No.} 40}, Elem.\ Math.\
  \textbf{16} (1961), 103--104.

\bibitem{Hanan}
M.~Hanan, {On {Steiner's} problem with rectilinear distance}, SIAM J.
  Appl.\ Math.\ {14} (1966), 255--265.
\MR{0224500 (37 \#99)}

\bibitem{HJLM93}
F. Harary, M. S. Jacobson, M. J. Lipman, and F. R. McMorris, \emph{Abstract sphere-of-influence graphs}, Math.\ Comput.\ Modelling \textbf{17} (1993), no. 11, 77--83.
\MR{1236512 (94f:05119)}

\bibitem{Harborth1974}
H.~Harborth, \emph{L\"osung zu {Problem} {664A}}, Elem.\ Math.\
  \textbf{29} (1974), 14--15.

\bibitem{Heppesmax}
A.~Heppes, \emph{Beweis einer {V}ermutung von {A}. {V}\'azsonyi}, Acta Math.\
  Acad.\ Sci.\ Hungar.\ \textbf{7} (1956), 463--466. \MR{0087116 (19,304e)}

\bibitem{Heppes1998}
A.~Heppes, \emph{Decomposing the {$2$}-sphere into domains of smallest possible
  diameter}, Period.\ Math.\ Hungar.\ \textbf{36} (1998), no.~2-3, 171--180.
  \MR{1694597 (2000f:52023)}

\bibitem{HRW} F. K. Hwang, D. S. Richards, and P. Winter, \emph{The Steiner Tree Problem}, Annals of Discrete Mathematics \textbf{53}, North Holland, Amsterdam, 1992. \MR{1192785 (94a:05051)}

\bibitem{Jenrich2014}
Thomas Jenrich and Andries~E. Brouwer, \emph{A $64$-dimensional counterexample to
  {B}orsuk's conjecture}, Electron.\ J. Combin.\ \textbf{21} (2014), no.~4, Paper
  4.29, 3 pp. \MR{3292266}

\bibitem{Joos2008}
A.~Jo{\'o}s, \emph{On a convex body with odd {H}adwiger number}, Acta Math.\
  Hungar.\ \textbf{119} (2008), no.~4, 307--321. \MR{2429292 (2009f:52044)}

\bibitem{KL1978}
G.~A. Kabatiansky and V.~I. Levenshtein, \emph{Bounds for
  packings on the sphere and in space}, Problemy Peredachi Informatsii
  \textbf{14} (1978), no.~1, 3--25; English translation: Problems of Information Transmission 14 (1978), no. 1, 1--17. \MR{0514023 (58 \#24018)}

\bibitem{KK1993}
Jeff Kahn and Gil Kalai, \emph{A counterexample to {B}orsuk's conjecture},
  Bull.\ Amer.\ Math.\ Soc.\ (N.S.) \textbf{29} (1993), no.~1, 60--62.
  \MR{1193538 (94a:52007)}

\bibitem{Kalai2015}
Gil Kalai, \emph{Some old and new problems in combinatorial geometry I: Around
  Borsuk's problem}, Surveys in Combinatorics 2015, London Math.\ Soc.\ Lecture
  Note Ser., vol.~424, Cambridge Univ.\ Press, Cambridge, 2015. pp.~147--174.
  \MR{3497269}

\bibitem{KMSS2012}
Haim Kaplan, Ji{\v{r}}{\'{\i}} Matou{\v{s}}ek, Zuzana Safernov{\'a}, and Micha
  Sharir, \emph{Unit distances in three dimensions}, Combin.\ Probab.\ Comput.\
  \textbf{21} (2012), no.~4, 597--610. \MR{2942731}

\bibitem{Kertesz1994}
G.~Kert{\'e}sz, \emph{Nine points on the hemisphere}, Intuitive geometry
  ({S}zeged, 1991), Colloq.\ Math.\ Soc.\ J\'anos Bolyai, vol.~63, North-Holland,
  Amsterdam, 1994, pp.~189--196. \MR{1383625 (97a:52031)}

\bibitem{KZ2004}
J. Klein and G. Zachmann, \emph{Point cloud surfaces using geometric proximity graphs},
Comput.\ Graph.\ \textbf{28} (2004), no.~6, 839--850.

\bibitem{Kleitman1965}
Daniel~J. Kleitman, \emph{On a lemma of {L}ittlewood and {O}fford on the
  distribution of certain sums}, Math.\ Z. \textbf{90} (1965), 251--259.
  \MR{0184865 (32 \#2336)}

\bibitem{Kobos2013}
Tomasz Kobos, \emph{An alternative proof of {P}etty's theorem on equilateral
  sets}, Ann.\ Polon.\ Math.\ \textbf{109} (2013), no.~2, 165--175. \MR{3103122}

\bibitem{Kobos2014}
Tomasz Kobos, \emph{Equilateral dimension of certain classes of normed spaces},
  Numer.\ Funct.\ Anal.\ Optim.\ \textbf{35} (2014), no.~10, 1340--1358.
  \MR{3233155}

\bibitem{Koszmider}
P.~Koszmider, \emph{Uncountable equilateral sets in {B}anach spaces of the form
  $C(K)$}, Israel J. Math., to appear, 2016. arXiv:1503.06356

\bibitem{KKLS2017}
Rob Kusner, W\"oden Kusner, Jeffrey C. Lagarias, and Senya Shlosman, \emph{The twelve spheres problem}, 
this volume, pp.~... -- .... arXiv:1611.10297, 2017.

\bibitem{Kupavskiy2011}
Andrey Kupavskiy, \emph{On the chromatic number of {$\mathbb{R}^n$} with an
  arbitrary norm}, Discrete Math.\ \textbf{311} (2011), no.~6, 437--440.
  \MR{2799896 (2012d:52028)}

\bibitem{Minkowski}
H.~Minkowski, \emph{Diophantische Approximationen}, Chelsea Publishing Co., New York, 1957.
\MR{0086102 (19,124f)}

\bibitem{Langi2009}
Zsolt L{\'a}ngi and M{\'a}rton Nasz{\'o}di, \emph{On the {B}ezdek-{P}ach
  conjecture for centrally symmetric convex bodies}, Canad.\ Math.\ Bull.\
  \textbf{52} (2009), no.~3, 407--415. \MR{2547807 (2010j:52068)}

\bibitem{Larman1972}
D.~G. Larman and C.~A. Rogers, \emph{The realization of distances within sets
  in {E}uclidean space}, Mathematika \textbf{19} (1972), 1--24. \MR{0319055 (47 \#7601)}

\bibitem{Larman1999}
D.~G. Larman and C.~Zong, \emph{On the kissing numbers of some special convex
  bodies}, Discrete Comput.\ Geom.\ \textbf{21} (1999), no.~2, 233--242.
  \MR{1668102 (99k:52030)}

\bibitem{Lassak1982}
Marek Lassak, \emph{An estimate concerning {B}orsuk partition problem}, Bull.\
  Acad.\ Polon.\ Sci.\ S\'er.\ Sci.\ Math.\ \textbf{30} (1982), no.~9-10, 449--451
  (1983). \MR{0703571 (84j:52014)}

\bibitem{Lawlor1994}
Gary Lawlor and Frank Morgan, \emph{Paired calibrations applied to soap films,
  immiscible fluids, and surfaces or networks minimizing other norms}, Pacific
  J. Math.\ \textbf{166} (1994), no.~1, 55--83. \MR{1306034 (95i:58051)}

\bibitem{Leech1964}
J.~Leech, \emph{Some sphere packings in higher space}, Canad.\ J. Math.\ 
\textbf{16} (1964), 657--682. \MR{0167901 (29 \#5166)}

\bibitem{Levenstein1979}
V.~I. Levenshtein, \emph{Bounds for packings in {$n$}-dimensional
  {E}uclidean space}, Dokl.\ Akad.\ Nauk SSSR \textbf{245} (1979), no.~6,
  1299--1303. \MR{529659 (80d:52017)}

\bibitem{Lin}
Aaron Lin, \emph{Equilateral sets in the {$\ell_1$} sum of {Euclidean} spaces},
  manuscript, 2016.

\bibitem{Ling2006}
Joseph~M. Ling, \emph{On the size of equilateral sets in spaces with the
  double-cone norm}, manuscript, 2006.

\bibitem{LM2015}
Ben Lund and Alexander Magazinov, \emph{The sign-sequence constant of the
  plane}, Acta Math.\ Hungar.\ (2016), to appear. arXiv:1510.04536

\bibitem{Maehara2007}
Hiroshi Maehara,
\emph{On configurations of solid balls in 3-space: chromatic numbers and knotted cycles},
Graphs Combin.\ \textbf{23} (2007), suppl.~1, 307--320. 
\MR{2320637 (2008c:05068)}

\bibitem{Makai1991}
E.~Makai, Jr. and H.~Martini, \emph{On the number of antipodal or strictly
  antipodal pairs of points in finite subsets of {$\textbf{R}^d$}}, Applied
  geometry and discrete mathematics, DIMACS Ser.\ Discrete Math.\ Theoret.\
  Comput.\ Sci., vol.~4, Amer.\ Math.\ Soc., Providence, RI, 1991, pp.~457--470.
  \MR{1116370 (92f:52020)}

\bibitem{Makeev2007}
V.~V. Makeev, \emph{Equilateral simplices in a four-dimensional normed space},
  Zap.\ Nauchn.\ Sem.\ S.-Peterburg.\ Otdel.\ Mat.\ Inst.\ Steklov.\ (POMI) Geom.\ i Topol.\ 
  \textbf{329} (2005), no.~9, 88--91, 197; English translation in 
J. Math.\ Sci.\ (N. Y.) \textbf{140} (2007), no.~4, 548--550. \MR{2215334 (2007b:52010)}

\bibitem{Malitz1992}
Seth~M. Malitz and Jerome~I. Malitz, \emph{A bounded compactness theorem for
  {$L^1$}-embeddability of metric spaces in the plane}, Discrete Comput.\ Geom.\
  \textbf{8} (1992), no.~4, 373--385. \MR{1176377 (93i:51034)}

\bibitem{martini-soltan-2005}
H.~Martini and V.~Soltan, \emph{Antipodality properties of finite sets in
  {E}uclidean space}, Discrete Math.\ \textbf{290} (2005), no.~2-3, 221--228.
  \MR{2123391 (2005i:52017)}

\bibitem{MS2006}
Horst Martini and Konrad~J. Swanepoel, \emph{Low-degree minimal spanning trees
  in normed spaces}, Appl.\ Math.\ Lett.\ \textbf{19} (2006), no.~2, 122--125.
  \MR{2198397 (2007f:52009)}

\bibitem{MSO2009}
H.~Martini, K.~J. Swanepoel, and P.~Oloff de~Wet, \emph{Absorbing angles,
  {S}teiner minimal trees, and antipodality}, J. Optim.\ Theory Appl.\
  \textbf{143} (2009), no.~1, 149--157. \MR{2545946 (2010m:05080)}

\bibitem{Martini2001}
Horst Martini, Konrad~J. Swanepoel, and Gunter Wei{\ss}, \emph{The geometry of
  {M}inkowski spaces---a survey. {I}}, Expo.\ Math.\ \textbf{19} (2001), no.~2,
  97--142. \MR{MR1835964 (2002h:46015a)}.
\emph{Erratum}, Expo.\ Math.\
  \textbf{19} (2001), no.~4, 364. \MR{1876256 (2002h:46015b)}

\bibitem{Matousek2002}
Ji{\v{r}}{\'{\i}} Matou{\v{s}}ek, \emph{Lectures on discrete geometry},
  Graduate Texts in Mathematics, vol.~212, Springer-Verlag, New York, 2002.
  \MR{1899299 (2003f:52011)}

\bibitem{Matousek}
Ji{\v{r}}{\'{\i}} Matou{\v{s}}ek, \emph{The number of unit distances is almost linear for most norms},
  Adv.\ Math.\ \textbf{226} (2011), no.~3, 2618--2628. \MR{2739786 (2011k:52008)}

\bibitem{Mercourakis2015}
S.~K. Mercourakis and G.~Vassiliadis, \emph{Equilateral sets in Banach spaces
  of the form $C(K)$}, Studia Math.\ \textbf{231} (2015), no.~3, 241--255.
  \MR{3471052}

\bibitem{Mercourakis2014}
S.~K. Mercourakis and G.~Vassiliadis, \emph{Equilateral sets in infinite dimensional {B}anach spaces}, Proc.\ Amer.\ Math.\ Soc.\ \textbf{142} (2014), no.~1, 205--212. \MR{3119196}

\bibitem{MQ1994}
T.~S. Michael and T.~Quint, \emph{Sphere of influence graphs: edge density and
  clique size}, Math.\ Comput.\ Modelling \textbf{20} (1994), no.~7, 19--24.
  \MR{1299482 (95i:05103)}

\bibitem{MQ1999}
T.~S. Michael and T.~Quint, \emph{Sphere of influence graphs in general metric spaces}, Math.\ Comput.\ Modelling \textbf{29} (1999), no.~7, 45--53. \MR{1688596 (2000c:05106)}

\bibitem{MQ2003}
T.~S. Michael and T.~Quint, \emph{Sphere of influence graphs and the
  {$L_\infty$}-metric}, Discrete Appl.\ Math.\ \textbf{127} (2003), no.~3,
  447--460. \MR{1976026 (2004g:05139)}

\bibitem{Milman1971}
V.~D. Milman, \emph{A new proof of {A}. {D}voretzky's theorem on cross-sections
  of convex bodies}, Funktsional.\ Anal.\ i Prilozhen.\ \textbf{5} (1971), no.~4,
  28--37; English translation in Functional Anal.\ Appl.\ \textbf{5} (1971), 288--295.
  \MR{0293374 (45 \#2451)}

\bibitem{Milman1985}
V.~D. Milman, \emph{Almost {E}uclidean quotient spaces of subspaces of a
  finite-dimensional normed space}, Proc.\ Amer.\ Math.\ Soc.\ \textbf{94} (1985),
  no.~3, 445--449. \MR{0787891 (86g:46025)}

\bibitem{Morgan}
F.~Morgan, \emph{Minimal surfaces, crystals, networks, and undergraduate
  research}, Math.\ Intelligencer \textbf{14} (1992), 37--44. \MR{1184317 (93h:53012)}

\bibitem{Morganbook}
F.~Morgan, \emph{Riemannian Geometry, A Beginner's Guide}, 2nd ed., 
A.~K.~Peters, Wellesley, MA, 1998.
\MR{1600519 (98i:53001)}

\bibitem{Musin2003}
O.~R. Musin, \emph{The problem of the twenty-five spheres}, Uspekhi Mat.\ Nauk
  \textbf{58} (2003), no.~4(352), 153--154. English translation: Russian Math.\ Surveys \textbf{58} (2003), no.~4, 794--795. \MR{2042912 (2005a:52016)}

\bibitem{Musin2006}
Oleg~R. Musin, \emph{The one-sided kissing number in four dimensions}, Period.\
  Math.\ Hungar.\ \textbf{53} (2006), no.~1-2, 209--225. \MR{2286472
  (2007j:52019)}

\bibitem{Musin2008}
O.~R.~Musin,
\emph{Bounds for codes by semidefinite programming}, Tr.\ Mat.\ Inst.\ Steklova \textbf{263} (2008), Geometriya, Topologiya i Matematicheskaya Fizika.~I, 143--158; reprinted in 
Proc.\ Steklov Inst.\ Math.\ \textbf{263} (2008), no. 1, 134--149.
\MR{2599377 (2011c:94085)}

\bibitem{Naszodi2016}
M.~{Nasz\'odi}, \emph{Flavors of translative coverings}, this volume, pp.~... -- .... arXiv:1603.04481, 2017.

\bibitem{NPS2016}
M.~{Nasz\'odi}, J.~Pach, and K.~J. Swanepoel, \emph{Sphere-of-influence graphs in normed spaces}, arXiv:1603.04481, 2016.

\bibitem{NPS2016b}
M.~{Nasz\'odi}, J.~Pach, and K.~J. Swanepoel, \emph{Arrangements of homothets
  of a convex body}, arXiv:1608.04639, 2016.

\bibitem{OS1979}
A.~M. Odlyzko and N.~J.~A. Sloane, \emph{New bounds on the number of unit
  spheres that can touch a unit sphere in {$n$} dimensions}, J. Combin.\ Theory
  Ser.\ A \textbf{26} (1979), no.~2, 210--214. \MR{530296 (81d:52010)}

\bibitem{Ostrovskii2013}
M. I. Ostrovskii, \emph{Metric embeddings. Bilipschitz and coarse embeddings into Banach spaces}, De Gruyter Studies in Mathematics, \textbf{49}, De Gruyter, Berlin, 2013. \MR{3114782}


\bibitem{Pach1996}
J.~Pach and G.~T{\'o}th, \emph{On the independence number of coin graphs},
  Geombinatorics \textbf{6} (1996), 30--33. \MR{1392795 (97d:05176)}

\bibitem{Perkal1947}
J. Perkal, \emph{Sur la subdivision des ensembles en parties de diam\`etre inf\'erieur}, Colloq.\ Math.\ \textbf{1} (1947), no.~1, p.~45.

\bibitem{Petty1971}
C.~M. Petty, \emph{Equilateral sets in {Minkowski} spaces}, Proc.\ Amer.\ Math.\
  Soc. \textbf{29} (1971), 369--374. \MR{0275294 (43 \#1051)}

\bibitem{Pollack1985}
R.~Pollack, \emph{Increasing the minimum distance of a set of points}, J.
  Combin.\ Th.\ Ser A \textbf{40} (1985), p.~450. \MR{0814430 (87b:52020)}

\bibitem{Polyanskii2016}
A.~Polyanskii, \emph{Pairwise intersecting homothets of a convex body}, arXiv:1610.04400, 2016.

\bibitem{Por}
A.~P\'or, \emph{On e-antipodal polytopes}, manuscript, 2003.

\bibitem{PS} H. J. Pr\"omel and A. Steger, \emph{The Steiner tree problem. A tour through graphs, algorithms, and complexity}, Advanced Lectures in Mathematics, Vieweg, Braunschweig, 2002. \MR{1891564 (2003a:05047)}

\bibitem{Raigorodskii2001}
A.~M. Raigorodskii, \emph{The {B}orsuk problem and the chromatic
  numbers of some metric spaces}, Uspekhi Mat.\ Nauk \textbf{56} (2001),
  no.~1(337), 107--146; English translation in Russian Math.\ Surveys \textbf{56} (2001), no.~1, 
  103--139. \MR{1845644 (2002m:54033)}

\bibitem{Raigorodskii2004}
A.~M. Raigorodskii, \emph{On the chromatic number of a space with the metric {$l_q$}},
  Uspekhi Mat.\ Nauk \textbf{59} (2004), no.~5(359), 161--162; English translation in 
Russian Math.\ Surveys 59 (2004), no.~5, 973--975. \MR{2125940 (2006e:05171)}

\bibitem{Raigorodskii2007}
A.~M. Raigorodskii, \emph{Around the {B}orsuk conjecture}, Sovrem.\ Mat.\ Fundam.\ Napravl.\
  \textbf{23} (2007), 147--164; English translation in 
J. Math.\ Sci.\ (N. Y.) \textbf{154} (2008), no.~4, 604--623. \MR{2342528 (2008j:52035)}

\bibitem{Raigorodskii2013}
A.~M. Raigorodskii, \emph{Coloring distance graphs and graphs of
  diameters}, Thirty essays on geometric graph theory, Springer, New York,
  2013, pp.~429--460. \MR{3205167}

\bibitem{Reutter1972}
O.~Reutter, \emph{{Problem} {664A}}, Elem.\ Math.\ \textbf{27} (1972),
  19.

\bibitem{Robins1995}
G.~Robins and J.~S. Salowe, \emph{Low-degree minimum spanning trees}, Discrete
  Comput.\ Geom.\ \textbf{14} (1995), 151--165. \MR{1331924 (96f:05180)}

\bibitem{Rogers-Zong}
C.~A. Rogers and C.~Zong, \emph{Covering convex bodies by translates of convex
  bodies}, Mathematika \textbf{44} (1997), no.~1, 215--218. \MR{1464387
  (98i:52026)}

\bibitem{Sachs1986}
H.~Sachs, \emph{No more than nine unit balls can touch a closed unit
  hemisphere}, Studia Sci.\ Math.\ Hungar.\ \textbf{21} (1986), no.~1-2, 203--206.
  \MR{0898858 (88k:52021)}

\bibitem{Schechtman2006}
Gideon~Schechtman, 
\emph{Two observations regarding embedding subsets of Euclidean spaces in normed spaces},
Adv.\ Math.\ \textbf{200} (2006), no.~1, 125--135.
\MR{2199631 (2006j:46015)}

\bibitem{Schramm1988}
Oded Schramm, \emph{Illuminating sets of constant width}, Mathematika
  \textbf{35} (1988), no.~2, 180--189. \MR{0986627 (89m:52013)}

\bibitem{Schurmann2006}
Achill Sch{\"u}rmann and Konrad~J. Swanepoel, \emph{Three-dimensional antipodal
  and norm-equilateral sets}, Pacific J. Math.\ \textbf{228} (2006), no.~2,
  349--370. \MR{2274525 (2007m:52024)}

\bibitem{SvdW1953}
K.~Sch{\"u}tte and B.~L. van~der Waerden, \emph{Das {P}roblem der dreizehn
  {K}ugeln}, Math.\ Ann.\ \textbf{125} (1953), 325--334. \MR{0053537 (14,787e)}

\bibitem{Shannon}
C.~E. Shannon, \emph{Probability of error for optimal codes in a {Gaussian}
  channel}, Bell System Tech.\ J. \textbf{38} (1959), 611--656. \MR{0103137 (21 \#1920)}

\bibitem{Soifer2009}
Alexander Soifer, \emph{The mathematical coloring book}, Springer, New York,
  2009. \MR{2458293 (2010a:05005)}

\bibitem{Soltan1975}
P.~S. Soltan, \emph{Analogues of regular simplexes in normed spaces}, Dokl.\
  Akad.\ Nauk SSSR \textbf{222} (1975), no.~6, 1303--1305, English translation:
  Soviet Math.\ Dokl.\ \textbf{16} (1975), no.~3, 787--789. \MR{0383246 (52 \#4127)}

\bibitem{Solymosi2008}
J{\'o}zsef Solymosi and Van~H. Vu, \emph{Near optimal bounds for the {E}rd{\H
  o}s distinct distances problem in high dimensions}, Combinatorica \textbf{28}
  (2008), no.~1, 113--125. \MR{2399013 (2009f:52042)}

\bibitem{MR86m:52015}
J.~Spencer, E.~Szemer{\'e}di, and W.~Trotter, Jr., \emph{Unit distances in the
  {E}uclidean plane}, Graph theory and combinatorics (Cambridge, 1983),
  Academic Press, London, 1984. pp.~293--303. \MR{0777185 (86m:52015)}

\bibitem{Strasmax}
S.~Straszewicz, \emph{Sur un probl\`eme g\'eom\'etrique de {P}. {E}rd{\H{o}}s},
  Bull.\ Acad.\ Polon.\ Sci.\ Cl.~III.\ \textbf{5} (1957), 39--40, IV--V.
  \MR{0087117 (19,304f)}

\bibitem{Swanepoel1999}
K.~J. Swanepoel, \emph{Cardinalities of {$k$}-distance sets in {M}inkowski
  spaces}, Discrete Math.\ \textbf{197/198} (1999), 759--767. \MR{1674902 (99k:52028)}

\bibitem{Swanepoel1999b}
K.~J. Swanepoel, \emph{New lower bounds for the {H}adwiger numbers of {$\ell_p$} balls for
  {$p<2$}}, Appl.\ Math.\ Lett.\ \textbf{12} (1999), no.~5, 57--60. \MR{1750139
  (2001e:94024)}

\bibitem{Swanepoel1999c}
K. J. Swanepoel, 
\emph{Vertex degrees of Steiner Minimal Trees in 
$\ell_p^d$ and other smooth Minkowski spaces}, 
Discrete Comput.\ Geom.\ \textbf{21} (1999), 437--447.
\MR{1672996 (2000g:05054)}

\bibitem{Swanepoel2000}
K.~J. Swanepoel, \emph{The local {S}teiner problem in normed planes}, Networks
  \textbf{36} (2000), 104--113. \MR{1793318 (2001f:05049)}

\bibitem{Swanepoel2000a}
K.~J. Swanepoel, \emph{Gaps in convex disc packings with an application to $1$-Steiner
minimum trees}, Monatsh.\ Math.\ \textbf{129} (2000), 217--226.
\MR{1746760 (2001a:52015)}

\bibitem{Swanepoel2002}
K.~J. Swanepoel, \emph{Independence numbers of planar contact graphs}, Discrete Comp.\
  Geom.\ \textbf{28} (2002), 649--670. \MR{1949907 (2003j:52016)}

\bibitem{Swanepoel2004}
K.~J. Swanepoel, \emph{Equilateral sets in finite-dimensional normed spaces},
  Seminar of Mathematical Analysis, Colecc.\ Abierta, vol.~71, Univ.\ Sevilla Secr.\ Publ., 2004. pp.~195--237. \MR{2117069 (2005j:46009)}

\bibitem{Swanepoel2005}
K.~J. Swanepoel, \emph{Quantitative illumination of convex bodies and vertex degrees of geometric Steiner minimal trees}, Mathematika \textbf{52} (2005), no.~1-2, 47--52.
\MR{2261841 (2008f:52009)}

\bibitem{Swanepoel2007}
K.~J. Swanepoel, \emph{The local {S}teiner problem in finite-dimensional
  normed spaces}, Discrete Comput.\ Geom.\ \textbf{37} (2007), no.~3, 419--442.
  \MR{2301527 (2008b:52003)}

\bibitem{Swanepoel2007b}
K.~J. Swanepoel, \emph{Upper bounds for edge-antipodal and subequilateral polytopes},
  Period.\ Math.\ Hungar.\ \textbf{54} (2007), no.~1, 99--106. \MR{2310370
  (2008k:52020)}

\bibitem{sw-lenz}
K.~J. Swanepoel, \emph{Unit distances and diameters in {E}uclidean spaces}, Discrete
  Comput.\ Geom.\ \textbf{41} (2009), no.~1, 1--27. \MR{2470067 (2010f:52031)}

\bibitem{Swanepoel2010}
Konrad~J. Swanepoel and Pavel Valtr, \emph{Large convexly independent subsets
  of {M}inkowski sums}, Electron.\ J. Combin.\ \textbf{17} (2010), no.~1,
  Research Paper 146, 7 pages. \MR{2745699 (2012c:52036)}

\bibitem{SwanepoelTAMS}
K.~J. Swanepoel, \emph{Sets of unit vectors with small subset sums}, Trans.\ Amer.\ Math.\
  Soc.\ \textbf{368} (2016), 7153--7188. \MR{3471088}

\bibitem{Swanepoel2008}
K.~J. Swanepoel and R.~Villa, \emph{A lower bound for the equilateral number of
  normed spaces}, Proc.\ Amer.\ Math.\ Soc.\ \textbf{136} (2008), 127--131.
  \MR{2350397 (2008j:46010)}

\bibitem{Swanepoel2013}
Konrad~J. Swanepoel and Rafael Villa, \emph{Maximal equilateral sets}, Discrete
  Comput.\ Geom.\ \textbf{50} (2013), no.~2, 354--373. \MR{3090523}

\bibitem{Swinnerton-Dyer1953}
H.~P.~F. Swinnerton-Dyer, \emph{Extremal lattices of convex bodies}, Proc.\
  Cambridge Philos.\ Soc.\ \textbf{49} (1953), 161--162. \MR{0051880 (14,540f)}

\bibitem{Talata1998}
I.~Talata, \emph{Exponential lower bound for the translative kissing numbers of
  $d$-dimensional convex bodies}, Discrete Comput.\ Geom.\ \textbf{19} (1998),
  no.~3, 447--455.  \MR{1615129 (98k:52046)}

\bibitem{Talata1999}
I.~Talata, \emph{The translative kissing number of tetrahedra is $18$}, Discrete
  Comput.\ Geom.\ \textbf{22} (1999), no.~2, 231--248. \MR{1698544 (2000e:52021)}

\bibitem{Talata2006}
I.~Talata, \emph{A legnagyobb minim\'alis szomsz\'edsz\'am egy okta\'eder
  eltoltjainak v\'eges elhelyez\'eseiben [{Determining} the largest possible
  minimum number of neighbours in a finite packing of translates of an
  octahedron]}, {Tudom\'anyos} {K\"ozlem\'enyek}, Szent Istv\'an Egyetem M\H{u}szaki F\H{o}iskolai Kar, 2006, pp.~122--125.

\bibitem{Talata1998b}
I.~Talata, \emph{On a lemma of {M}inkowski}, Period.\ Math.\ Hungar.\
  \textbf{36} (1998), no.~2-3, 199--207. \MR{1694585 (2000i:52035)}

\bibitem{Talata1999b}
I.~Talata, \emph{On extensive subsets of convex bodies}, Period.\ Math.\ Hungar.\
  \textbf{38} (1999), no.~3, 231--246. \MR{1756241 (2001b:52035)}

\bibitem{Talata2000}
I.~Talata, \emph{A lower bound for the translative kissing numbers of simplices},
  Combinatorica \textbf{20} (2000), no.~2, 281--293. \MR{1767027 (2001d:52030)}

\bibitem{Talata2002}
I.~Talata, \emph{On minimum kissing numbers of finite translative packings of a
  convex body}, Beitr\"age Algebra Geom.\ \textbf{43} (2002), no.~2, 501--511.
  \MR{1957754 (2003j:52018)}

\bibitem{Talata2005}
I.~Talata, \emph{On {H}adwiger numbers of direct products of convex bodies},
  Combinatorial and computational geometry, Math.\ Sci.\ Res.\ Inst.\ Publ.,
  vol.~52, Cambridge Univ.\ Press, Cambridge, 2005, pp.~517--528. \MR{2178337
  (2006g:52030)}

\bibitem{Talata2011}
I.~Talata, \emph{Finite translative packings with large minimum kissing numbers},
  Stud.\ Univ.\ \v Zilina Math.\ Ser.\ \textbf{25} (2011), no.~1, 47--56.
  \MR{2963987}

\bibitem{Terenzi1987}
Paolo Terenzi, \emph{Successioni regolari negli spazi di {B}anach (Regular
  sequences in {B}anach spaces)}, Rend.\ Sem.\ Mat.\ Fis.\ Milano \textbf{57}
  (1987), 275--285 (1989). \MR{1017856 (90m:46022)}

\bibitem{Terenzi1989}
Paolo Terenzi, \emph{Equilater sets in {B}anach spaces}, Boll.\ Un.\ Mat.\ Ital.\ A (7)
  \textbf{3} (1989), no.~1, 119--124. \MR{0990095 (90c:46017)}

\bibitem{Thompson1996}
A.~C. Thompson, \emph{Minkowski Geometry}, Encyclopedia of Mathematics and its
  Applications, vol.~63, Cambridge University Press, Cambridge, 1996.
  \MR{1406315 (97f:52001)}

\bibitem{T88}
G. T. Toussaint, \emph{A graph-theoretical primal sketch}, in Computational morphology, A computational geometric approach to the analysis of form, ed.~G.~T.~Toussaint, Machine Intelligence and Pattern Recognition, \textbf{6}, North-Holland, Amsterdam, 1988. pp.~229--260. \MR{0993994 (89k:68151)}

\bibitem{T2014}
G. T. Toussaint, \emph{The sphere of influence graph: Theory and applications}, International Journal of Information Technology \& Computer Science \textbf{14} (2014), no.~2, 37--42.

\bibitem{Vaisala2012}
Jussi V{\"a}is{\"a}l{\"a}, \emph{Regular simplices in three-dimensional normed
  spaces}, Beitr.\ Algebra Geom.\ \textbf{53} (2012), no.~2, 569--570.
  \MR{2971762}

\bibitem{Valtr}
P.~Valtr, \emph{Strictly convex norms allowing many unit distances and related
  touching questions}, manuscript, 2005.

\bibitem{Watson1971}
G.~L. Watson, \emph{The number of minimum points of a positive quadratic form},
  Dissertationes Math.\ Rozprawy Mat.\ \textbf{84} (1971). 42 pp. \MR{0318061 (47 \#6610)}

\bibitem{Wyner}
A.~D. Wyner, \emph{Capabilities of bounded discrepancy decoding}, Bell System
  Tech.\ J. \textbf{44} (1965), 1061--1122. \MR{0180417 (31 \#4652)}

\bibitem{Xu2007}
Lanju Xu, \emph{A note on the kissing numbers of superballs}, Discrete Comput.\
  Geom.\ \textbf{37} (2007), no.~3, 485--491. \MR{2301531 (2008b:52027)}

\bibitem{Yu2009}
Long Yu, \emph{Blocking numbers and fixing numbers of convex bodies}, Discrete
  Math.\ \textbf{309} (2009), no.~23-24, 6544--6554. \MR{2558619 (2010j:52035)}

\bibitem{YZ2009}
Long Yu and Chuanming Zong, \emph{On the blocking number and the covering
  number of a convex body}, Adv.\ Geom.\ \textbf{9} (2009), no.~1, 13--29.
  \MR{2493260 (2010d:52007)}

\bibitem{Zahl2013}
Joshua Zahl, \emph{An improved bound on the number of point-surface incidences
  in three dimensions}, Contrib.\ Discrete Math.\ \textbf{8} (2013), no.~1,
  100--121. \MR{3118901}

\bibitem{Zong1993}
C.~Zong, \emph{Packing and covering}, Ph.D. thesis, Technische Universit{\"a}t
  Wien, 1993.

\bibitem{Zong1994}
C.~M.~Zong, \emph{An example concerning the translative kissing number of a convex
  body}, Discrete Comput.\ Geom.\ \textbf{12} (1994), 183--188. \MR{1283886 (95e:52033)}

\bibitem{Zong1995}
C.~M.~Zong, \emph{Some remarks concerning kissing numbers, blocking
  numbers and covering numbers}, Period.\ Math.\ Hungar.\ \textbf{30} (1995),
  no.~3, 233--238. \MR{1334968 (96g:52039)}

\bibitem{Zong1996}
C.~Zong, \emph{The kissing numbers of tetrahedra}, Discrete Comput.\
  Geom.\ \textbf{15} (1996), no.~3, 239--252. \MR{1380392 (97c:11070)}

\bibitem{ZongBook}
C.~Zong, \emph{Strange phenomena in convex and discrete geometry},
  Springer-Verlag, New York, 1996. \MR{1416567 (97m:52001)}

\bibitem{Zong1997}
C.~M.~Zong, \emph{The translative kissing number of the {C}artesian product of two
  convex bodies, one of which is two-dimensional}, Geom.\ Dedicata \textbf{65}
  (1997), no.~2, 135--145. \MR{1451968 (98e:52022)}

\bibitem{ZongSurvey1998}
C.~Zong, \emph{The kissing numbers of convex bodies --- a brief survey}, Bull.\
  London Math.\ Soc.\ \textbf{30} (1998), 1--10. \MR{1479030 (98k:52048)}

\bibitem{Zong2008}
C.~Zong, \emph{The kissing number, blocking number and covering number of a
  convex body}, Surveys on discrete and computational geometry, Contemp.\ Math.,
  vol.~453, Amer.\ Math.\ Soc., Providence, RI, 2008, pp.~529--548. \MR{2405694 (2010b:52029)}

\end{thebibliography}

\providecommand{\MR}{MR}
\providecommand{\MRhref}[2]{%
\href{http://www.ams.org/mathscinet-getitem?mr=#1}{#2}
}

}
\end{document}